\documentclass[11pt]{scrartcl}    

\usepackage{graphicx}
\usepackage{color}
\usepackage{amsmath}
\usepackage{amssymb}
\usepackage{esint}
\usepackage[bookmarks=false]{hyperref}
\usepackage{mathtools}
\usepackage{float}
\usepackage[most]{tcolorbox}
\usepackage[colorinlistoftodos,prependcaption,textsize=tiny]{todonotes}
\usepackage{todonotes}
\usepackage{subcaption}
\usepackage{amsthm}
\usepackage[margin=2.5cm]{geometry}
\usepackage{tikz}
\usepackage{floatpag}

\usepackage{mathtools}

\makeatletter
\DeclareFontFamily{OMX}{MnSymbolE}{}
\DeclareSymbolFont{MnLargeSymbols}{OMX}{MnSymbolE}{m}{n}
\SetSymbolFont{MnLargeSymbols}{bold}{OMX}{MnSymbolE}{b}{n}
\DeclareFontShape{OMX}{MnSymbolE}{m}{n}{
	<-6>  MnSymbolE5
	<6-7>  MnSymbolE6
	<7-8>  MnSymbolE7
	<8-9>  MnSymbolE8
	<9-10> MnSymbolE9
	<10-12> MnSymbolE10
	<12->   MnSymbolE12
}{}
\DeclareFontShape{OMX}{MnSymbolE}{b}{n}{
	<-6>  MnSymbolE-Bold5
	<6-7>  MnSymbolE-Bold6
	<7-8>  MnSymbolE-Bold7
	<8-9>  MnSymbolE-Bold8
	<9-10> MnSymbolE-Bold9
	<10-12> MnSymbolE-Bold10
	<12->   MnSymbolE-Bold12
}{}

\let\llangle\@undefined
\let\rrangle\@undefined
\DeclareMathDelimiter{\llangle}{\mathopen}%
{MnLargeSymbols}{'164}{MnLargeSymbols}{'164}
\DeclareMathDelimiter{\rrangle}{\mathclose}%
{MnLargeSymbols}{'171}{MnLargeSymbols}{'171}
\makeatother

\DeclareMathOperator{\Tr}{Tr}

\DeclareMathOperator{\Span}{span}

\DeclareMathOperator*{\argmin}{arg\,min}

\DeclareMathOperator{\supp}{supp}
\DeclareMathOperator{\Ran}{Ran}

\DeclareMathOperator{\sgn}{sgn}

\theoremstyle{plain}
\newtheorem{theorem}{Theorem}
\newtheorem{lemma}[theorem]{Lemma}

\newtheorem{corollary}[theorem]{Corollary}
\newtheorem{proposition}[theorem]{Proposition}
\newtheorem{assumption}{Assumption}
\newtheorem{problem}{Problem}
\newtheorem{conjecture}[theorem]{Conjecture}

\theoremstyle{definition}
\newtheorem{definition}[theorem]{Definition}
\newtheorem{example}[theorem]{Example}

\theoremstyle{remark}
\newtheorem{remark}[theorem]{Remark}

\begin{document}

\title{Constructing sampling schemes via coupling: Markov semigroups and optimal transport}

\author{N. N{\"u}sken and G. A. Pavliotis
}

\maketitle

\begin{abstract}
	In this paper we develop a general framework for constructing and analysing coupled Markov chain Monte Carlo samplers, allowing for both (possibly degenerate) diffusion and piecewise deterministic Markov processes. For many performance criteria of interest, including the asymptotic variance, the task of finding efficient couplings can be phrased in terms of problems related to optimal transport theory. We investigate general structural properties, proving a singularity theorem that has both geometric and probabilistic interpretations. Moreover, we show that those problems can often be solved approximately and support our findings with numerical experiments. For the particular objective of estimating the variance of a Bayesian posterior, our analysis suggests using novel techniques in the spirit of antithetic variates. Addressing the convergence to equilibrium of coupled processes we furthermore derive a modified Poincar{\'e} inequality.
\end{abstract}

\section{Introduction and motivation}

Many computational problems arising in machine learning, Bayesian statistics, molecular dynamics and various other fields require the approximation of probability distributions (in the following denoted by $\pi$) on a high-dimensional space $E$. In particular, uncertainty quantification in a Bayesian framework is intimately related to the evaluation of appropriate summary statistics such as the variance of the posterior \cite[Chapter 10]{ghanem2017handbook}, \cite[Chapter 8]{smith2013uncertainty}, \cite[Chapter 6]{sullivan2015introduction}. Often, this task is approached by considering empirical measures associated to an ensemble of $n$ particles, i.e. approximations of the form 
\begin{equation}
\pi \approx \frac{1}{n}\sum_{i=1}^{n} \delta_{X^{(i)}} =: \tilde{\pi},
\end{equation}
where $X^{(i)}$ stands for the location of the $i$th particle and $\delta_x$ denotes the Dirac measure centred at $x \in E$. Usually, the particles are moved according to some (more often than not stochastic) dynamics, judiciously crafted in order for the empirical measure $\tilde{\pi}_t := \frac{1}{n}\sum_{i=1}^{n} \delta_{X_t^{(i)}}$ to approach $\pi$ when $t$ reaches a terminal value (finite or infinite). This methodology has been particularly influential in statistical inference of hidden-state Markov models (stochastic filtering or sequential Monte Carlo, see for instance \cite{RC2015}, \cite{dM2013} and references therein). Ensemble based methods have also been employed in the contexts of optimisation \cite{pinnau2017consensus,revees2004genetic}, molecular dynamics \cite{rousset:hal-00008276}, Markov chain Monte Carlo \cite{LMW2018,neal2017circularly}, or variational Bayesian inference \cite{liu2016stein}. Let us also mention the works \cite{andrieu2010particle} and \cite{heng2015gibbs}, combining different aspects of various sampling strategies. The increasing availability of parallel-processing computational architectures has further encouraged the development and analysis of similar methodologies.

From an abstract perspective, many of the aforementioned algorithms targeting a probability measure $\pi$ on some state space $E$ naturally produce probability measures $\bar{\pi}$ on the product space $\bar{E} =\prod_{i = 1}^n E_i$, where $E_i$ is an identical copy of $E$, standing for the state space of the $i$th particle. Denoting by $P_i:\mathcal{P}(\bar{E}) \rightarrow \mathcal{P}(E_i)$ the mappings that send probability measures on $\bar{E}$ to their marginals on $E_i$, one then obtains the measure
\begin{equation}
\label{eq:mean of marginals}
\frac{1}{n}\sum_{i=1}^n P_i(\bar{\pi})
\end{equation}
as an approximation for $\pi$. Clearly, the map 
\begin{equation}
\label{eq:sum of marginals}
\Pi :\mathcal{P}(\bar{E}) \rightarrow \mathcal{P}(E), \quad \bar{\pi} \mapsto \frac{1}{n} \sum_{i = 1}^n P_i(\bar{\pi})
\end{equation}
is far from injective since $\Pi(\bar{\pi})$ only depends on the marginals of $\bar{\pi}$. This viewpoint shows that there is a considerable flexibility when generating the joint measure $\bar{\pi}$, immediately suggesting fruitful connections to the theory of couplings of probability measures \cite{lindvall2002,thorisson2000} prominently encountered for instance in relation to optimal transport problems \cite{V2003,V2009} or decay estimates in Wasserstein distances (see for instance \cite{eberle2016reflection}). Since in applications $\Pi(\bar{\pi})$ is only an approximation of the target measure of interest, the freedom to design appropriate couplings can be used to suppress bias, variance and discretisation errors. This general idea has proved to be very versatile, leading to powerful simulation techniques such as multilevel Monte Carlo \cite{Giles2015}, coupling from the past \cite{Propp1998coupling} and antithetic variates \cite[Section 9.2]{kroese2013handbook}. 

\subsection{Couplings and Markov Chain Monte Carlo}
In this paper we focus on coupling techniques in the context of Markov chain Monte Carlo simulations. Assume that we are interested in computing the expectation 
\begin{equation}
\mathbb{E}_{\pi}f = \int_{E} f \mathrm{d}\pi
\end{equation}
of a given test function (henceforth called \emph{observable}) $f:E \rightarrow \mathbb{R}$ with respect to some probability measure $\pi$ on $E$. As approximations relying on quadratures tend to be computationally infeasible in high dimensions, a standard approach is to construct a Markov process $(X_t)_{t \ge 0}$ on $E$ such that
\begin{equation}
\label{eq:ergodicity}
\lim_{T \rightarrow \infty} \frac{1}{T} \int_0^T f(X_s)\, \mathrm{d} s = \int_{E} f \mathrm{d} \pi, 
\end{equation}
i.e. the process $(X_t)_{t \ge 0 }$ is supposed to be ergodic with respect to $\pi$. More generally, one often constructs a Markov process $(\bar{X}_t)_{t \ge 0}$ on an extended state space $\bar{E}$, ergodic with respect to a measure $\bar{\pi}$ that has $\pi$ as its marginal,
\begin{equation}
\int \bar{\pi}(x,y) \, \mathrm{d}y = \pi(x),
\end{equation}
where $(x,y) \in \bar{E}$ and $x \in E$. This idea is used for instance in Hamiltonian Monte Carlo \cite{neal2011mcmc} or sampling schemes based on underdamped Langevin dynamics \cite[Chapter 2]{Free_energy_computations}. We also refer to the introduction in \cite{DNP2017} for a more general perspective. In this work we take this approach further, in the sense that we consider extended measures $\bar{\pi}$ that have fixed marginals with respect to (multiple) complimentary subspaces of $\bar{E}$. Immediately, this viewpoint suggests fruitful connections to theory of optimal (multimarginal) transportation.  

To explain our approach, let us consider $n$ identical copies of $E$, $(X_t)_{t \ge 0}$ and $\pi$, denoted by $E_i$, $(X^i_t)_{t \ge 0}$ and $\pi_i$, for $i \in \{1, \ldots, n\}$. Our main object of study is the class of Markovian couplings $(\bar{X}_t)_{t \ge 0}$ of $\left\{(X_t^i)_{t \ge 0} :\, i \in \{1, \ldots, n\}\right\}$ on the product space $\bar{E} = \prod_{i=1}^n E_i$  that obey certain mild regularity assumptions. In particular, we characterise those couplings in terms of their infinitesimal generators in Section \ref{sec:general framework} (see Proposition \ref{prop:Gamma coupling}). One of the recurring themes of this work is the use of the latter in the analysis of coupled processes. From the coupling property of $(\bar{X}_t)_{t \ge 0}$ it follows immediately that if this process is ergodic, then its invariant measure (denoted by $\bar{\pi}$) is a coupling of the $n$ copies of $\pi$. 

 For an observable $f \in L^1(\pi)$, we can define the extended observable $F:\bar{E} \rightarrow \mathbb{R}$ by 
\begin{equation}
\label{eq:extension observable}
F(x_1,\ldots,x_n) = \frac{1}{n} \sum_{i=1}^n f(x_i).
\end{equation} 
From \eqref{eq:ergodicity} it is then immediate that 
\begin{equation}
\label{eq:extended ergodicity}
\lim_{T \rightarrow \infty} \frac{1}{T} \int_0^T F(\bar{X}_t) \,\mathrm{d} t = \int_{E} f \,\mathrm{d} \pi, 
\end{equation}
i.e. the coupled process $(\bar{X}_t)_{t \ge 0}$ in conjunction with the observable \eqref{eq:extension observable} provides a valid sampling scheme. Let us remark that the framework we develop in Section \ref{sec:general framework} accommodates the case when the spaces $E_i$, the processes $(X^i_t)_{t \ge 0}$ and the measures $\pi_i$ are not identical, allowing for considerable flexibility in the construction of coupled samplers.

The study of observables of the form \eqref{eq:extension observable} provides a compelling dual perspective on the `sum of marginals' operator \eqref{eq:sum of marginals}. Denoting by $\mathcal{B}_b(E) $ the space of bounded measureable functions we can consider the `extension operator' 
\begin{equation}
\label{eq:extension operator}
\Pi^*:\mathcal{B}_b(E) \rightarrow \mathcal{B}_b(\bar{E}), \quad f \mapsto \frac{1}{n} \sum_{i=1}^n f_i,
\end{equation}  
provided by \eqref{eq:extension observable}. For $\bar{\pi} \in \mathcal{P}(\bar{E})$ and $f \in \mathcal{B}_b(E)$ we clearly have $(\Pi \bar{\pi})(f) = \bar{\pi}(\Pi^* f)$, showing that understanding the class of observables given by \eqref{eq:extension observable} is sufficient for analysing the properties of measures of the form \eqref{eq:mean of marginals}. This idea features in particular in Section \ref{ch:coupling_spectral_gap} in the analysis of the exponential convergence to equilibrium for coupled processes. 

Clearly, it is desirable to choose the coupling in such a way that the convergence in \eqref{eq:extended ergodicity} is as fast as possible. Reasonable criteria involve the asymptotic variance (related to appropriate central limit theorems) and the spectral gap (related to the speed of convergence to equilbrium), both of which will be addressed in the present paper. We refer the reader to \cite[Section 1]{DNP2017} for a more detailed discussion of these quantities.  

Similar constructions to ours have been considered in the literature, in particular in a discrete time setting. In \cite{frigessi2000antithetic}, the authors construct coupled Gibbs samplers using a very related rationale (see also \cite{holmes2009antithetic} and \cite{neal1998suppressing}). Coupled Metropolis-Hastings samplers have been put forward in \cite{craiu2007acceleration}. The work \cite{craiu2005} provides a theoretical framework that is however is quite different from the one developed in the present paper. Further algorithmic ideas related to coupled samplers can also be found in \cite{kwak2016antithetic} and \cite{rigat2012parallel}.  

\subsection{Overview of the main results by means of a simple example}
In this section we present our main findings informally by means of a very simple example, pointing to the exact statements in the forthcoming sections. Let us stress that our results hold in much greater generality, in particular also including the recently fashionable piecewise deterministic Markov processes (PDMPs).

Let us consider $n=2$ particles (the locations of which are denoted by $X_t$ and $Y_t$) moving each in one dimension according to the overdamped Langevin dynamics defined by the SDEs 
\begin{subequations}
	\label{eq:running example sdes}
	\begin{align}
	\label{eq:particle x}
	\mathrm{d}X_t & = - V'(X_t)\,\mathrm{d}t + \sqrt{2}\,\mathrm{d}B^x_t, \\
	\label{eq:particle y}
	\mathrm{d}Y_t & = - V'(Y_t)\,\mathrm{d}t + \sqrt{2}\,\mathrm{d}B^y_t,
	\end{align}
\end{subequations}
where $V \in C^{\infty}(\mathbb{R})$ is a fixed potential such that
\begin{equation}
Z:= \int_\mathbb{R} e^{-V} \mathrm{d} x < \infty,
\end{equation}
and $(B_t^x)_{t \ge 0}$, $(B_t^y)_{t \ge 0}$ denote standard one-dimensional Brownian motions. As is well-known, each of these processes considered separately is ergodic with respect to $\pi= \frac{1}{Z} e^{-V} \mathrm{d}x$, i.e. \eqref{eq:ergodicity} holds for an appropriate class of observables.
Note that we have deliberately refrained from stating that the Brownian motions $(B^x_t)_{t \ge 0}$ and $(B^y_t)_{t \ge 0}$ are independent. Indeed, notwithstanding any dependence between these, it is immediate that \eqref{eq:extended ergodicity} holds for the extended observable $F(x,y) = \frac{1}{2} (f(x) + f(y))$, as defined in \eqref{eq:extension observable}.
One of the main objectives of our analysis is to find couplings between $(B^x_t)_{t \ge 0}$ and $(B^y_t)_{t \ge 0}$ such that the induced joint process $(X_t,Y_t)_{t \ge 0}$ has favourable properties, in terms of the asymptotic variance associated to \eqref{eq:extended ergodicity} as well as in terms of convergence to equilibrium of $\Pi(\bar{\pi}_t)$, where $\bar{\pi}_t$ denotes the joint law of $(X_t,Y_t)$ and $\Pi$ has been defined in \eqref{eq:sum of marginals}.
The dependence between $(B_t^x)_{t \ge 0}$ and $(B_t^y)_{t \ge 0}$ can be conveniently encoded in a suitable matrix-valued function $G:\mathbb{R}^2 \rightarrow \mathbb{R}^{2 \times 2}$, writing 
\begin{equation}
\label{eq:G_example}
\mathrm{d}\begin{pmatrix}
X_t	\\
Y_t
\end{pmatrix} =
\begin{pmatrix}
-V'(X_t)  \\
-V'(Y_t) 
\end{pmatrix}\mathrm{d}t + \sqrt{2}G(X_t,Y_t)
\begin{pmatrix}
\mathrm{d}W_t^x \\
\mathrm{d}W_t^y
\end{pmatrix},
\end{equation}
for two independent Brownian motions $(W_t^x)_{t \ge 0}$ and $(W_t^y)_{t \ge 0}$. In this sense, the optimisation problem alluded to above is naturally posed over an appropriate set of  matrix-valued functions.

In Section \ref{sec:general framework} we introduce the general framework, leading to a characterisation of possible couplings in terms of infinitesimal generators of the dynamics. In the present example, the generators of the one-particle dynamics are given by
\begin{equation}
\mathcal{L}_x = -V'(x) \partial_x + \partial_x^2, \quad
\mathcal{L}_y = -V'(y) \partial_y + \partial_y^2.
\end{equation}
The generators of possible couplings $(X_t,Y_t)_{t \ge 0}$ turn out to be of the form
\begin{equation}
\label{eq:L_G example}
\bar{\mathcal{L}}_{\Gamma} := \mathcal{L}_x + \mathcal{L}_y + \Gamma, \quad \Gamma = 2 \alpha \partial_x \partial_y, 
\end{equation}
where $\alpha:\mathbb{R}^2 \rightarrow [-1,1]$ is a function with suitable regularity properties. The connection between $\alpha$ and $G$ will be made precise in Section \ref{ch:coupling_examples}. We use the term `coupling operator' when referring to $\Gamma$ and denote the set of such operators by $\mathcal{G}$. Note that $\Gamma$ as defined in \eqref{eq:L_G example} vanishes on functions that depend either only on $x$ or only on $y$. In Proposition \ref{prop:Gamma coupling} we will see that this property essentially characterises coupling operators in general. As it turns out (see the discussion in Section \ref{sec:ergodic couplings}), not every coupling of ergodic Markov processes is such that the joint process is ergodic. Hence, we introduce the subset $\mathcal{G}^0 \subset \mathcal{G}$ of ergodic coupling operators that do preserve ergodicity. In the present example, $\bar{\mathcal{L}}_\Gamma$ is elliptic whenever $-1 < \alpha < 1$ pointwise and therefore the corresponding coupling operators are ergodic. Intuitively, the nonergodic coupling operators in $\mathcal{G} \setminus \mathcal{G}^0$ can hence be thought of as lying `at the boundary' of $\mathcal{G}^0$. Although we have not been successful in proving a rigorous version of this statement in a general context, the reader is encouraged to keep this picture in mind.

On  $\mathcal{G}^0$  we can consider the map $\Gamma \mapsto \bar{\pi}_\Gamma$, where $\bar{\pi}_\Gamma$ stands for the unique invariant measure associated with $\bar{\mathcal{L}}_\Gamma$. It is immediately clear from the construction that any $\bar{\pi}_\Gamma$ arising in this way is a coupling of $\pi$ to itself (i.e. $\bar{\pi}_\Gamma$ has marginal $\pi$ in both directions). We argue in Section \ref{ch:coupling_asymvar} that a wide range of optimisation problems in our context can be cast in the following form, very closely linked to the theory of optimal transportation,
\begin{equation}
\label{eq:OT problem example}
\min_{\Gamma \in \mathcal{G}^0} \int_{\bar{E}} c \, \mathrm{d}\bar{\pi}_\Gamma,
\end{equation}
where $c$ is an appropriate cost function. Indeed, we show in Section \ref{sec:coupling_CLT} that the task of optimising the asymptotic variance of a coupled process with respect to a given observable is equivalent to \eqref{eq:OT problem example}, for a cost function that is constructed from the solution of a related Poisson equation. 
Addressing the problem \eqref{eq:OT problem example}, we first note that the dependence $\Gamma \mapsto \int_{\bar{E}} c \, \mathrm{d}\bar{\pi}_\Gamma$ is highly nonlinear, in particular, for $\Gamma \in \mathcal{G}^0$ and $\lambda \in [0,1]$, the mapping $\lambda \mapsto \int_{\bar{E}} c \, \mathrm{d}\bar{\pi}_{\lambda \Gamma}$ generally exhibits many local minima and maxima\footnote{ This claim is made assuming that $\lambda \Gamma \in \mathcal{G}^0$ for all $\lambda \in [0,1]$.}. Nevertheless, we find that under suitable conditions the function $\Gamma \mapsto \int_{\bar{E}} c \, \mathrm{d}\bar{\pi}_\Gamma$ does not attain its extrema on interior points. This is the main result of Section \ref{sec:OT} and is stated rigorously in Theorem \ref{thm:max boundary}. In the example under consideration, this implies that optimal couplings necessarily satisfy $\Vert \alpha \Vert_{\infty} = 1$, leading to singular (i.e. degenerately elliptic) generators $\bar{\mathcal{L}}_\Gamma$. This conclusion is interesting in two respects:
Firstly, it complements standard results from optimal transport theory showing that optimal couplings are typically singular in a certain sense. We stress, however, that the problem \eqref{eq:OT problem example} is genuinely different from problems occurring in optimal transport theory, and that our proof uses fundamentally different techniques. Secondly, this result supports the folklore that optimal MCMC samplers use the least amount of noise necessary to guarantee their ergodicity. 

While the results from Section \ref{ch:coupling_asymvar} indicate the possible locations of optimal coupling operators $\Gamma$ in the set $\mathcal{G}$, they do not help to actually find them in practice. In Section \ref{ch:perturbative approach} we address this problem by considering small perturbations around the trivial coupling $\mathbf{0} \in \mathcal{G}$ corresponding to independent Brownian motions. This leads to a much more tractable optimisation problem that can be solved explicitly in concrete examples and gives promising results in our numerical experiments. In the present example, `mirror coupling' ($B_t^x = -B_t^y$) turns out to be optimal in terms of reducing the asymptotic variance of monotone observables, in the sense of the optimisation problem just referred to. However, for different observables (perhaps exhibiting other types of symmetries) more intricate coupling strategies turn out to be advisable. We wish to stress that those observables are of particular relevance for the quantification of uncertainty in a Bayesian framework, for instance in the computation of the variance or related quantities of a posterior distribution.

In Section \ref{ch:coupling_spectral_gap} we analyse the rate of convergence to equilibrium for coupled processes. As we will see, the former can be characterised in terms of an inequality of Poincar{\'e} type that is in turn related to an appropriate Hilbert space constructed in terms of the coupling. Applied to the present example, this result shows that the rate of convergence can be improved relative to the one-particle dynamics if the potential $V$ is symmetric, i.e. $V(x) = V(-x)$. In general, the speed of convergence to equilibrium can also be slower, in the sense that there might appear a constant $C>1$ in front of the exponential decay estimate. We leave a more detailed exploration of this phenomenon for future study.

The structure of the paper is as follows: In Section \ref{sec:general framework}, we introduce our framework in a general setting. In particular, we fix the notation (Section \ref{subsec:notation}), characterise coupled processes in terms of their generators (Section \ref{subsec:coupled processes}), discuss ergodic properties (Section \ref{sec:ergodic couplings}) and provide a means of construction coupling operators given the generators of the marginal processes (Section \ref{sec:general construction}). In Section \ref{ch:coupling_examples}, we illustrate our theory with concrete examples, namely diffusion processes overdamped Langevin dynamics (Section \ref{sec:overdamped_all}), underdamped Langevin dynamics (Section \ref{ex:underdamped}), as well as the zigzag process (Section \ref{sec:zigzag}), representing the class of piecewise deterministic Markov processes. In Section \ref{sec:coupling_CLT} we derive a central theorem for coupled processes. The ensuing expression for the asymptotic variance is connected to the theory of optimal transportation, as exhibited and analysed in Section \ref{sec:OT}. In Section \ref{ch:perturbative approach} we take a perturbative approach towards the solutions of the aforementioned optimal transport problems and exemplify our results in the context of the examples presented in Section \ref{ch:coupling_examples}. Finally, in Section \ref{ch:coupling_spectral_gap} we analyse the convergence of coupled processes to equilibrium relying on a suitable functional inequality of Poincar{\'e} type. The appendix comprises additional material required for some of the proofs throughout the article. 

\section{Coupled processes and coupling operators}
\label{sec:general framework}
This section is devoted to the interplay between couplings of Markov processes and their infinitesimal generators. We start by specifying the setting and notations.
\subsection{Preliminaries, notation and setting}
\label{subsec:notation}
\subsubsection{Feller semigroups}
For a given locally compact Polish space $E$ we will denote the space of bounded, Borel measurable functions by $\mathcal{B}_b(E)$, the space of bounded continuous functions by $C_b(E)$, and the space of continuous functions vanishing at infinity\footnote{Recall that a function $f: E_i \rightarrow \mathbb{R}$ vanishes at infinity if for all $\varepsilon > 0$ there exists a compact set $K \subset E_i$ such that $\vert f(x) \vert \le \varepsilon$ for all $x \in E_i \setminus K$.} by $C_0(E)$. The space of probability measures on $E$ (equipped with the Borel $\sigma$-algebra $\mathcal{B}(E)$) will be denoted by $\mathcal{P}(E)$. All of theses spaces become Banach spaces when equipped with the supremum norm, denoted by $\Vert \cdot \Vert_{\infty}$. An $E$-valued Markov process $(X_t)_{t \ge 0}$ induces a semigroup of linear operators $(S_t)_{t \ge 0}$ on $\mathcal{B}_b(E)$ via
\begin{equation}
(S_t f) (x) = \mathbb{E} [f(X_t) \vert X_0 = x], \quad f \in \mathcal{B}_b(E), \, x \in E.
\end{equation}
Since the terminology varies slightly across the literature, we next give the definition of Feller processes used in this paper, mostly  adopting the notations and conventions from \cite[Chapter 1]{LevyMatters2013}. For more details we furthermore refer to \cite[Chapter 17]{K2002}.
\begin{definition}[Feller processes]
	\label{def:Feller}
	 A Markov process $(X_t)_{t \ge 0}$ satisfies the \emph{Feller property} if the following hold for the corresponding semigroup $(S_t)_{t \ge 0}$:
	\begin{enumerate}
		\item 
		$(S_t)_{t \ge 0}$ leaves $C_0(E)$ invariant, i.e. $S_t f \in C_0(E)$ for all $f \in C_0(E)$ and $t \ge 0$.
		\item
		$(S_t)_{t \ge 0}$ is strongly continuous on $C_0(E)$, i.e. 
		\begin{equation}
		\Vert S_t f - f \Vert_{\infty} \xrightarrow{t \rightarrow 0} 0
		\end{equation} 
		for all $f \in C_0(E)$.
	\end{enumerate}
\end{definition}
Provided that $(S_t)_{t \ge 0}$ is a Feller semigroup as specified above, we define its generator $(\mathcal{L},\mathcal{D}(\mathcal{L}))$ in the usual way \cite{EN2000}[Chapter 2]. Throughout this paper, we will assume for convenience that the state space $E$ has a differential structure such that the space $C_c^{\infty}(E)$ of compactly supported smooth functions is meaningfully defined. We can then make the following assumption on the domain of the generator $\mathcal{L}$:  

\begin{assumption}
	\label{ass:rich}
	All considered Feller processes are rich\footnote{We adopt this terminology following for instance \cite[Section 1.5]{Kuehn2017} and references therein.}, i.e. $C_c^{\infty}(E) \subset \mathcal{D}(\mathcal{L})$.
\end{assumption}

The state spaces encountered in the examples in Section \ref{ch:coupling_examples} naturally admit differentiable structures and the corresponding generators fulfil Assumption \ref{ass:rich}. Let us remark, however, that our framework can be extended to more general scenarios (including for instance infinite dimensional examples), replacing $C_c^{\infty}(E)$ by suitable function spaces adapted to the particular setting.

\begin{remark}
	\label{rem:constants}
	Clearly, $C_0(E)$ does not contain constant functions (apart from the zero function) if $E$ is not compact. In preparation for Definition \ref{def:coupling operators}, we mention that $\mathcal{D}(\mathcal{L})$ can naturally be extended to a subset of $C_b(E)$ by endowing the latter with the topology of uniform convergence on compact subsets of $E$. Following \cite{Sch98} (see also \cite[Section 4.8]{Jac2001vol1}), the extended generator $(\tilde{\mathcal{L}},\mathcal{D}(\tilde{\mathcal{L}}))$ can then be defined by
 		\begin{subequations}
		\begin{align}
		\label{eq:extended generator}
		\mathcal{D}(\tilde{\mathcal{L}}) & = \left\{ f \in C_b(E): \, \lim_{t \rightarrow 0} \frac{S_t f -f}{t}  \quad \text{exists uniformly on compact sets}\right\}, \\
		\tilde{\mathcal{L}} f & = \lim_{t \rightarrow 0} \frac{S_t f -f}{t}, \quad f \in \mathcal{D}(\tilde{\mathcal{L}}).
		\end{align}
	\end{subequations}
Since $(S_t)_{t \ge 0}$ is conservative\footnote{Conservativeness of the semigroup $(S_t)_{t \ge 0}$ means that $S_1 \mathbf{1} = \mathbf{1}$ for all $t \ge 0$, encoding the conservation of total probability mass.} we immediately see that $\mathbf{1} \in \mathcal{D}(\tilde{\mathcal{L}})$ and $\tilde{\mathcal{L}}\mathbf{1} = 0$, i.e. $\tilde{\mathcal{L}}$ vanishes on constant functions. Moreover, $(\tilde{\mathcal{L}},\mathcal{D}(\tilde{\mathcal{L}}))$ is an extension of $(\mathcal{L},\mathcal{D}(\mathcal{L}))$, i.e. $\mathcal{D}(\mathcal{L}) \subset \mathcal{D}(\tilde{\mathcal{L}})$ and $\tilde{\mathcal{L}}\vert_{\mathcal{D}(\mathcal{L})} = \mathcal{L}$. Henceforth we will thus drop the tilde when no confusion is possible.
\end{remark}
\subsubsection{Product spaces}
\label{sec:product spaces}
We will be dealing with a collection of locally compact Polish spaces $E_i$, indexed by $i \in \{1,\ldots, n\}$, and denote their cartesian product by $\bar{E} := E_1 \times \ldots \times E_n$. For $f \in \mathcal{B}_b(E_i)$, it is of course understood that also $f \in \mathcal{B}_b(\bar{E})$, then depending only on the coordinate $x_i$ in $\bar{x} \equiv (x_1, \ldots, x_n)$. To a given function $f \in C_c^{\infty}(E_i)$ or $f \in C_0(E_i)$, we will also associate the canonical element in $\mathcal{B}_b(\bar{E})$, but wish to emphasize that clearly $f$ does not in general have compact support or does not vanish at infinity when considered as a function on $\bar{E}$. Frequently, the spaces $E_i$ will be identical copies of each other, i.e. $\bar{E} = E^n$. Given $f \in \mathcal{B}_b(E)$, we will then write $f_i \in \mathcal{B}_b(\bar{E})$ for the function given by
\begin{equation}
f_i(x_1,\ldots,x_n) = f(x_i), \quad (x_1, \ldots, x_n) \in \bar{E}.
\end{equation}  
Sums of unbounded operators are defined in the usual way: For two operators $(A,\mathcal{D}(A))$ and $(B,\mathcal{D}(B))$ defined on the same Banach space $X$ (i.e. $\mathcal{D}(A) \subset X$ and $\mathcal{D}(B) \subset X$), their sum is defined via
\begin{equation}
(A+B)f := Af + Bf, \quad f \in \mathcal{D}(A+B) := \mathcal{D}(A) \cap \mathcal{D}(B),
\end{equation}
see for instance \cite[Chapter III]{EN2000}.
In the case when $(A,\mathcal{D}(A))$ and $(B,\mathcal{D}(B))$ are defined on two distinct spaces $\mathcal{B}_b(E_i)$ and $\mathcal{B}_b(E_j)$, $i \neq j$,  (i.e. $\mathcal{D}(A) \subset \mathcal{B}_b(E_i)$ and $\mathcal{D}(B) \subset \mathcal{B}_b(E_j)$), their sum is defined as
\begin{equation}
A + B := A \otimes I + I \otimes B, \quad \mathcal{D}(A + B) := \mathcal{D}(A) \widehat{\otimes} \mathcal{D}(B), 
\end{equation}
where $\widehat{\otimes}$ denotes the canonical  topological tensor product on $\mathcal{B}_b(E_i \times E_j)$ (see \cite{AS1957} and \cite[A-I 3.7]{Oneparsempos}).
\subsection{Coupled processes}
\label{subsec:coupled processes}

Assume that for $i \in \{1,\ldots,n\}$, we are given locally compact Polish spaces $E_i$, representing the state spaces of $n$ distinct particles.
Furthermore, for $i\in \{1, \ldots, n\}$, let us fix Feller semigroups $(S^i_t)_{t \ge 0}$ on $E_i$ with generators $(\mathcal{L}_i,\mathcal{D}(\mathcal{L}_i))$ and associated Feller processes $(X^i_t)_{t \ge 0}$ on appropriate stochastic bases $(\Omega_i,\mathbb{P}_i,(\mathcal{F}^i_t)_{t\ge 0})$, representing the dynamics of those particles (in the following these processes will be referred to as the `one-particle dynamics'). Let us also assume that the spaces $C_c^{\infty}(E_i)$ are cores for the semigroups $(S_t^i)_{t \ge 0}$.
\begin{remark}
	By Watanabe's Theorem (see for instance \cite[Proposition 17.9]{K2002}), $C_c^\infty(E_i)$ is a core for $\mathcal{L}_i$ if it is dense in $\mathcal{D}(\mathcal{L}_i)$ and invariant under $(S^i_t)_{t \ge 0}$. It is possible to extend our framework by exchanging $C_c^{\infty}(E_i)$ for other cores, say $\mathcal{D}_i$, as long as the first condition in Defintion \ref{def:coupling operators} is altered accordingly.  
\end{remark} 
Consider now a Feller process $(\bar{X}_t)_{t\ge 0}$ on the product space $\bar{E} := E_1 \times \ldots \times E_n$, together with its associated semigroup $(\bar{S}_t)_{t \ge 0}$ on $\mathcal{B}_b(\bar{E})$ and generator $(\bar{\mathcal{L}},\mathcal{D}(\bar{\mathcal{L}}))$ in $C_0(\bar{E})$. We will denote the $E_i$-valued coordinate processes of $(\bar{X}_t)_{t \ge 0}$ by $(\bar{X}_t^i)_{t \ge 0}$. 
\begin{definition}[Feller couplings]
	\label{def:coupling}
	The process $(\bar{X}_t)_{t \ge 0}$ is called a \emph{Feller coupling} of the processes $(X^i_t)_{t \ge 0 }$, if  its marginals are given by these processes, i.e. if for all $i \in \{1,\ldots,n\}$, the processes $(\bar{X}^i_t)_{t \ge 0}$ and $(X_t^i)_{t \ge 0}$ induce the same law on the space of c{\`a}dl{\`a}g functions $D([0,\infty),E_i)$ \footnote{Every Feller process has a c{\`a}dl{\`a}g modification, see \cite[Theorem 1.19]{LevyMatters2013}.}.  
\end{definition}
 Our aim in this section is to characterise the infinitesimal generators of such coupled processes. 
\begin{remark}
	\label{rem:Feller coupling}
	We are making the following two assumptions when considering the class of processes described above. Firstly, we assume certain continuity properties of the process $(\bar{X}_t)_{t \ge 0}$, encoded mainly in the fact that the space $C_0(\bar{E})$ is invariant under the action of the corresponding semigroup (see for instance \cite[Lemma 1.4]{LevyMatters2013} for more details). Restricting our attention to the class of Feller processes allows us to use the theory of strongly continuous semigroups on Banach spaces \cite{EN2000} for the development of the theory in this section. In examples and applications however (see Sections \ref{ch:coupling_examples} and \ref{ch:perturbative approach}), we will relax this assumption a bit, allowing for more general processes.
	
	 Secondly, we consider processes $(\bar{X}_t)_{t \ge 0}$ that are Markovian. Obviously there are many non-Markovian couplings of the underlying processes $(X_t^i)_{t \ge 0}$ and indeed those might be  of particular interest for applications. Hence we plan to investigate the possibility of extending our framework in this direction in a forthcoming project. 
\end{remark}

We now proceed to introduce a class of linear (unbounded) operators $(\Gamma, \mathcal{D}(\Gamma))$ on $\mathcal{B}_b(\bar{E})$:
\begin{definition}[Coupling operators]
	\label{def:coupling operators}
	Let $(\Gamma,\mathcal{D}(\Gamma))$ be a (possibly unbounded) linear operator on $\mathcal{B}(\bar{E})$. Then $(\Gamma,\mathcal{D}(\Gamma))$ is called a \emph{coupling operator} if the following conditions are satisfied:
	\begin{enumerate}
		\item
		\label{it:kernel} 
		Test functions that depend on only one component of $\bar{x} = (x_1,\ldots,x_n)$ are in the kernel of $\Gamma$: 
		
		For all $i \in \{1,\ldots,n\}$ and $f \in C_c^{\infty}(E_i)$ it holds that $f \in \mathcal{D}(\Gamma)$ and 
			\begin{equation}
		\label{eq:kernel property}
		\Gamma f = 0.
		\end{equation}
	\item
	\label{it:generator}  
	 The operator  
	\begin{equation}
	\label{eq:generator}
	\bar{\mathcal{L}}_{\Gamma} := \sum_{i=1}^{n} \mathcal{L}_{i} + \Gamma 
	\end{equation}
	with domain $\mathcal{D}(\bar{\mathcal{L}}_\Gamma) = \bigotimes_{i=1}^n \mathcal{D}(\mathcal{L}_i) \cap \mathcal{D}(\Gamma)$ is closable and its closure
	is the infinitesimal generator of a Feller process on $\bar{E}$.
	\end{enumerate}
The Feller semigroup corresponding to a coupling operator $\Gamma$ will be referred to by $(\bar{S}_t^\Gamma)_{t \ge 0}$. Furthermore, the set of coupling operators will be denoted by $\mathcal{G}$, i.e.
\begin{equation*}
\mathcal{G} = \{\Gamma:\mathcal{D}(\Gamma) \subset \mathcal{B}_b(\bar{E}) \rightarrow \mathcal{B}_b(\bar{E}): \, \text{Conditions \ref{it:kernel} and \ref{it:generator} are satisfied.}\}.
\end{equation*}
\end{definition}

\begin{remark}
	\label{rem:domain gamma}
	We will not distinguish (notationally) between $\bar{\mathcal{L}}_\Gamma$ and its closure. Notice also that the first condition in Definition \ref{def:coupling operators} necessitates to think of $\Gamma$ as an operator defined on (a subspace) of $C_b(\bar{E})$ (rather than $C_0(\bar{E})$), because of $C_c^{\infty} (E_i) \not\subset C_0(\bar{E})$. The second condition is naturally concerned with $\bar{\mathcal{L}}_\Gamma$ being the generator of a semigroup on $C_0(\bar{E})$ (and hence with the appropriate restriction of $\Gamma$). We refer to Remark \ref{rem:constants} for a discussion about the extended generator on $C_b(\bar{E})$.
\end{remark}

 We have the following result, characterising completely the set of rich Feller couplings in terms of the coupling operators $\mathcal{G}$:
\begin{proposition}
	\label{prop:Gamma coupling}
	For any $\Gamma \in \mathcal{G}$, the Feller process generated by $\bar{\mathcal{L}}_{\Gamma}$ as defined in \eqref{eq:generator} is a coupling of the processes $\left( (X^i_t)_{t \ge 0 }, \, i \in \{1,\ldots n\} \right)$. Conversely, if $(\bar{X}_t)_{t \ge 0}$ is a Feller coupling of the processes $\left( (X^i_t)_{t \ge 0 }, \, i \in \{1,\ldots n\} \right)$, then its generator is of the form \eqref{eq:generator}, with $\Gamma \in \mathcal{G}$.
\end{proposition}
\begin{proof}
	Let $\Gamma \in \mathcal{G}$, and consider the process $(\bar{X}_t)_{t \ge 0}$ generated by the corresponding operator $\bar{\mathcal{L}}_{\Gamma}$ as defined in \eqref{eq:generator}. Let us fix $i \in \{1,\ldots,n\}$. Clearly, $C_c^{\infty}(E_i) \in \mathcal{D}(\bar{\mathcal{L}}_\Gamma)$ and $\bar{\mathcal{L}}_\Gamma f = \mathcal{L}_i f$ for $f \in C^{\infty}_c(E_i)$. Hence, for all $f\in C_c^{\infty}(E_i)$, the process
	\begin{equation}
	 f(\bar{X}_t) - f(\bar{X}_0) - \int_0^t (\mathcal{L}_i f)(\bar{X}_s) \,\mathrm{d}s, \quad t \ge 0,
	\end{equation}
	is a martingale with respect to the natural filtration $(\mathcal{F}^{\bar{X}}_t)_{t \ge 0}$ generated by $(\bar{X}_t)_{t \ge 0}$.  From the uniqueness of the martingale problem for the generator $\mathcal{L}_i$ (see for instance \cite[Section 4.4]{EthierKu86}) and the fact that $C_c^{\infty}(E_i)$ is a core for $\mathcal{L}_i$, it follows that  $(\bar{X}^i_t)_{t \ge0}$ has indeed the same law as $(X^i_t)_{t \ge 0}$. 
	
	Conversely, assume that $(\bar{X}_t)_{t \ge 0}$ is a Feller coupling of the processes 
	$\left( (X^i_t)_{t \ge 0 }, \, i \in \{1,\ldots n\} \right)$ and denote its generator by $(\tilde{\mathcal{L}},\mathcal{D}(\mathcal{\tilde{L}}))$. For fixed $i \in \{1, \ldots, n\}$, we first argue that $C^{\infty}_c(E_i) \subset \mathcal{D}(\tilde{\mathcal{L}})$, referring to the domain of the extended generator defined in \eqref{eq:extended generator}. Indeed, this amounts to showing that for all $f \in C_c^{\infty}(E_i)$ the limit 
	\begin{equation}
	\label{eq:generator limit}
	\lim_{t \rightarrow 0}\frac{1}{t}(\tilde{S}_t f - f)
	\end{equation}
	exists uniformly on compact sets. By the coupling property (Definition \ref{def:coupling}), we have that $\tilde{S}_t f = S_t^i f$ for $f \in C_c^{\infty}(E_i)$. Therefore (and since $C_c^{\infty}(E_i) \subset \mathcal{D}(\mathcal{L}_i)$ by assumption), it follows that the limit \eqref{eq:generator limit} even exists uniformly on the whole of $\bar{E}$.
	
	We can now define $\tilde{\Gamma} := \tilde{\mathcal{L}} - \sum_{i = 1}^n \mathcal{L}_i$ on $\mathcal{D}(\tilde{\Gamma}) :=\mathcal{D}(\tilde{\mathcal{L}}) \cap \bigotimes_{i=1}^n\mathcal{D}(\mathcal{L}_i)$.
	It is then sufficient to show that $\Gamma$ satisfies the first condition of Definition \ref{def:coupling operators}. To this end, take  $f \in C_c^{\infty}(E_i)$ in the martingale problem for $\tilde{\mathcal{L}}$ to see that 
	\begin{equation}
	f(\bar{X}^i_t) - f(\bar{X}^i_0) - \int_0^t (\mathcal{L}_{i} f)(\bar{X}^i_s) \,\mathrm{d}s - \int_0^t (\tilde{\Gamma} f) (\bar{X}_s) \,\mathrm{d}s, \quad t \ge 0,
	\end{equation}
	is a martingale, again with respect to the natural filtration $(\mathcal{F}^{\bar{X}}_t)_{t \ge 0}$ generated by $(\bar{X}_t)_{t \ge 0}$. 
	Since $(\bar{X}_t^i)_{t \ge 0}$ and $(X_t^i)_{t \ge 0}$ are equal in law by assumption, it follows that $((\bar{X}^i_t)_{t \ge 0}, (\mathcal{F}^{\bar{X}}_t)_{t \ge 0})$ is a solution to the martingale problem for $\mathcal{L}_i$. Hence, $\int_0^t (\tilde{\Gamma} f) (\bar{X}_s) \mathrm{d}s$ has to be a martingale as well. Since this process is of finite variation (and the initial condition for the process $(\bar{X}_t)_{t \ge 0}$ can be chosen arbitrarily), this implies $\tilde{\Gamma} f =0$. 
\end{proof}
\begin{remark}
	Similar approaches, describing couplings in terms of coupling operators, are known from the literature. See for instance \cite[Chapter 2, Definition 2.7]{CM2005} and references therein. The exact result of Proposition \ref{prop:Gamma coupling} and its proof using martingale problems seems to be new and in particular relevant for Conjecture 2.18 and Open Problem 2.19 in \cite{CM2005}. 
\end{remark}
\begin{example}
	\label{ex:independent coupling}
	\emph{Independent (or trivial) coupling: }The zero operator $\Gamma = \mathbf{0}$ is always in $\mathcal{G}$, as the conditions of Definition \ref{def:coupling operators} clearly hold. Indeed, consider the operator 
	\begin{equation}
	\label{eq:L0 def}
	\bar{\mathcal{L}}_0 := \sum_{i=1}^n \mathcal{L}_i
	\end{equation}
	 on the domain $\mathcal{D}(\bar{\mathcal{L}}_0):= \widehat{\bigoplus}_i \mathcal{D}(\mathcal{L}_i)$. It is  straightforward (see for instance \cite[A-I 3.7]{Oneparsempos}) to show that $\bar{\mathcal{L}}_0$ is the generator of a Feller semigroup $(\bar{S}^0_t)_{t \ge 0}$ given by
	\begin{equation}
	\label{eq:independent semigroup}
	\bar{S}^0_t f = \left(\prod_{i=1}^n S^i_t\right) f,\quad f \in \mathcal{B}_b(\bar{E}),\, t \ge 0,
	\end{equation}
	and that the associated Feller process is just $(\bar{X}^0_t)_{t \ge 0} = (X_t^1, \ldots, X_t^n)_{t \ge 0}$, i.e. it is obtained from independent copies of the underlying processes.  
\end{example}
Let us briefly discuss some of the implications of the conditions in Definition \ref{def:coupling operators}. As can be seen from the proof of Proposition \ref{prop:Gamma coupling}, the first condition is instrumental in guaranteeing that the coupled process $(\bar{X}_t)_{t \ge 0}$ has the correct marginals. To put the second condition into context, we remark that generators of Feller semigroups can be characterised by means of the Hille-Yosida-Ray theorem in terms of the positive maximum principle (see \cite[Lemma 1.28]{LevyMatters2013} and
\cite[Theorem 1.30]{LevyMatters2013}). As we will see in Examples (Section \ref{ch:coupling_examples}), the latter often restricts the `size' of coupling operators, so that the set $\mathcal{G}$ usually turns out to be `bounded' in a certain sense. Let us close this section by mentioning the following conjecture:
\begin{conjecture}
	\label{conj:G convex}
	The set $\mathcal{G}$ is convex.
\end{conjecture}
Resolving the above conjecture would shed further light on the structure of $\mathcal{G}$, especially in connection with the results obtained in Section \ref{ch:coupling_asymvar}.
\subsection{Ergodicity and regularity of couplings}
\label{sec:ergodic couplings}

From here on, let us make the following assumption, natural in the context of MCMC samplers:
\begin{assumption}
	\label{ass:marginal ergodicity}
	The underlying one-particle processes $(X^i_t)_{t \ge 0}$ are ergodic, i.e. for every $i \in \{1, \ldots, n\}$ there exists a unique probability measure $\pi_i \in \mathcal{P}(E_i)$ on $E_i$ such that
	\begin{equation}
	\int_{E_i} (\mathcal{L}_i f) \, \mathrm{d} \pi_i = 0, \quad f \in \mathcal{D}(\mathcal{L}_i),
	\end{equation}
	and, furthermore,
	\begin{equation}
	\label{eq:ergodic averages}
	\lim_{T \rightarrow \infty} \frac{1}{T} \int_0^T f(X_t^i) \, \mathrm{d}t = \int_{E_i} f \, \mathrm{d}\pi_i, \quad f \in C_b(E_i).
	\end{equation}
\end{assumption}   
Following up on Example \eqref{ex:independent coupling}, we see that the semigroup $(\bar{S}^0_t)_{t \ge 0}$ as given in \eqref{eq:independent semigroup} with generator $\bar{\mathcal{L}}_0$ as defined in \eqref{eq:L0 def} is ergodic with respect to the product measure
\begin{equation}
\bar{\pi}_0 := \bigotimes_{i=1}^n \pi_i
\end{equation} 
on $\bar{E}$. 
Unfortunately, it turns out that not all coupling operators $\Gamma \in \mathcal{G}$ induce ergodic coupled processes, even under the Assumption \ref{ass:marginal ergodicity} (for an example, see \cite[Section 3.1]{LPP2015}). We therefore make the following definition:
\begin{definition}
	\emph{Ergodic couplings:}
	A coupling operator $\Gamma \in \mathcal{G}$ is called \emph{ergodic}, if the Feller process generated by $\bar{\mathcal{L}}_{\Gamma}$ is ergodic. The corresponding subset of ergodic coupling operators will be denoted by $\mathcal{G}^0$. The unique invariant measure associated to $\Gamma \in \mathcal{G}^0$ will be denoted by $\bar{\pi}_{\Gamma}$.
\end{definition}
\begin{remark}
	By construction, the measures $\bar{\pi}_{\Gamma}$ are couplings of the one-particle invariant measures $(\pi_i)_{i=1}^n$.
\end{remark}
\begin{remark}
	\label{rem:practical ergodicity}
	For the analysis, ergodicity of the coupling is a crucial requirement (although with more work it might be possible to extend some of the results to the case when ergodicity fails to hold). Let us emphasize however that the validity of \eqref{eq:extended ergodicity} does not depend on this, as only the marginal property of the coupling is used in its derivation. Hence in practice it is harmless to use nonergodic couplings, and in fact our results obtained in Section \ref{ch:coupling_asymvar} (in particular, Theorem \ref{thm:max boundary}) suggest using couplings that are at least not straightforwardly seen to be ergodic. In this case, quantities measuring the performance of the sampler (such as the asymptotic variance corresponding to certain observables) might be undefined or depend on the initial condition.
\end{remark}
In general, ergodicity might fail in various ways. For instance, the process might not admit any invariant measure at all, or convergence of ergodic averages (in the sense of \eqref{eq:ergodic averages}) might not hold. The following result shows that the situation is simpler in our context.
\begin{lemma}
	\label{lem:coupled ergodicity}
	Let Assumption \ref{ass:marginal ergodicity} be satisfied. Then the following hold:
	\begin{enumerate}
		\item 
		Every Feller coupling admits at least one invariant measure.
		\item 
		If a Feller coupling admits a unique invariant measure, then the process is ergodic, i.e. \eqref{eq:ergodic averages} holds.
	\end{enumerate}
\end{lemma}
\begin{proof}
	We proceed along the lines of the proof of the Krylov-Bogolyubov theorem  \cite[Section 3.1]{DaZa1996}. Let $\nu_i \in \mathcal{P}(E_i)$ be arbitrary initial conditions for the processes $(X_t^i)_{t \ge 0}$. By ergodicity, the families $(\tilde{\pi}^i_t)_{t \ge 0}$ of C{\'e}saro averages 
	\begin{equation*}
	\tilde{\pi}^i_t(A) = \frac{1}{t} \int_0^t \left((S_s^i)^*\nu_i\right)(A) \, \mathrm{d}s, \quad A \in \mathcal{B}(E_i),
	\end{equation*}
	are convergent, and therefore tight. Let $(\bar{X}_t)_{t \ge 0}$ be a Feller coupling and denote the corresponding  C{\'e}saro averages by $(\tilde{\bar{\pi}}_t)_{t \ge 0}$. For any $t \ge 0$, $\tilde{\bar{\pi}}_t$ is a coupling of $(\tilde{\pi}^i_t)_{i = 1}^n$. Using an obvious extension of \cite[Lemma 4.4]{V2009} to the multimarginal case, we see that $(\tilde{\bar{\pi}}_t)_{t \ge 0}$ is tight. By Prokhorov's theorem, there exists a weakly converging subsequence, the limit of which is an invariant measure (as in the proof of the Krylov-Bogolyubov theorem). This proves the first claim. Now let us assume that there exists a unique invariant measure. Since any convergent subsequence of $(\tilde{\bar{\pi}}_t)_{t \ge 0}$ has to converge to the same limit, and the sequence is tight, the second claim follows. 
\end{proof}
By the results obtained in \cite{kliemann1987recurrence}, uniqueness of the invariant measure is implied by certain regularity properties of the process. This leads to the following convenient criterion.
\begin{corollary}[Regular couplings]
	\label{cor:regular couplings}
	Let Assumption \ref{ass:marginal ergodicity} be satisfied and consider a Feller coupling $(\bar{X}_t)_{t \ge 0}$. If the corresponding transition functions $(\rho_t(x,\cdot))_{t \ge 0, x \in \bar{E}}$ are mutually absolutely continuous (i.e. if the process is \emph{regular}), then $(\bar{X}_t)_{t \ge 0}$ is ergodic.
\end{corollary}
The measures $\pi_i$, as well as $\bar{\pi}_\Gamma$ (for $\Gamma \in \mathcal{G}^0$) induce the usual Hilbert spaces $L^2(\pi_i)$ and $L^2(\bar{\pi}_\Gamma)$ of square-integrable functions. A crucial role will be played furthermore by the corresponding subspaces of centred functions, defined by
\begin{equation}
L^2_0(\bar{\pi}_\Gamma) = \{f \in L^2(\bar{\pi}_\Gamma) \, \vert \, \bar{\pi}_\Gamma (f) = 0\},
\end{equation}
and analogously for $L_0^2(\pi_i)$. Since any Feller process has a right-continuous version, the semigroups $(\bar{S}_t^\Gamma)_{t \ge 0}$ as well as the corresponding generators $(\bar{\mathcal{L}}_\Gamma,\mathcal{D}(\bar{\mathcal{L}}_\Gamma))$ have unique extensions to strongly continuous semigroups on $L^2(\bar{\pi}_\Gamma)$ by Jensen's inequality. Slightly abusing the notation, we will denote those semigroups and their generators by the same letters. Before moving on to a somewhat more explicit description of coupling operators, let us mention the following open question, related to Conjecture \ref{conj:G convex}:
\begin{conjecture}
	\label{conj:G0 convex}
	The set $\mathcal{G}^0$ is convex.
\end{conjecture}

\subsection{A general way of constructing coupling operators}
\label{sec:general construction}

In this section we describe an approach to construct coupling operators explicitly in applications. The particular form presented here is also theoretically important since some of the calculations in later sections depend on it (especially the proof of Theorem \ref{thm:max boundary}).

As in the previous section, we assume that the marginal processes are ergodic with respect to invariant measures $\pi_i$ (see Assumption \ref{ass:marginal ergodicity}). Furthermore, let us assume that the generators $\mathcal{L}_i$ can be written as
\begin{equation}
\label{eq:generator decomposition}
\mathcal{L}_i = -\sum_{k=1}^{K_i} (A_k^i)^*A_k^i + B^i,
\end{equation}
where $(A_k^i)^*$ denotes the adjoint of $A_k^i$ in $L^2(\pi_i)$, and $B^i$ is antisymmetric in $L^2(\pi_i)$. Clearly, this decomposition into symmetric and antisymmetric parts is always possible, and in many cases the operators $A_k^i$ and $B^i$ can be chosen to have convenient forms. Note however that the decomposition \eqref{eq:generator decomposition} is not unique, since there are (infinitely) many ways of choosing the operators $A^i_k$. A particular choice of decomposing the generators $\mathcal{L}_i$ as in \eqref{eq:generator decomposition} hence essentially amounts to the choice of square-roots for the symmetric parts. We remark here that naturally $C_c^{\infty}(\bar{E}) \subset \mathcal{D}(A_k^i)$ and $C_c^{\infty}(\bar{E}) \subset \mathcal{D}(B^i)$ are implicitly assumed, authorising the computations in later sections. The following Lemma is essential for the construction in this subsection: 

\begin{lemma}
	\label{lem:decomposition kernel}
	Let $A_k^i$ and $B^i$ be given as in equation \eqref{eq:generator decomposition}. Then 
	\begin{equation}
	 \Span \mathbf{1} \subset \ker A_k^i, \quad \Span \mathbf{1} \subset \ker B^i,
	\end{equation}
	for all $i \in \{1,\ldots,n\},\,k \in \{1,\ldots K_i\}$.
\end{lemma}

\begin{proof}
	See \cite[Proposition 2]{villani2009hypocoercivity}.
\end{proof}
We may now set
\begin{equation}
\label{eq:Gamma construction}
\Gamma = \sum_{(i,j,k,l) \in \mathcal{J}} \alpha_{ijkl}(x_1,\ldots,x_n) A_{k}^i A_{l}^j
\end{equation}
for appropriate measurable functions $\alpha_{ijkl}: \bar{E} \rightarrow \mathbb{R}$ and where we have introduced the set of \emph{admissible indices}
\begin{equation}
\label{eq:admissible indices}
\mathcal{J} = \{ (i,j,k,l) \in \mathbb{N}^4: 1 \le i,j \le n,\, i \neq j, 1 \le k \le K_i, \, 1 \le l \le K_j\},
\end{equation}
associated to the decomposition \eqref{eq:generator decomposition}.
 Applying Lemma \ref{lem:decomposition kernel}, we see immediately that the first condition of Definition \ref{def:coupling operators} is satisfied. The second condition will typically enforce certain bounds on the functions $\alpha_{ijkl}$ via the positive maximum principle as well as regularity constraints if we are interested in Feller couplings. Those properties will have to be determined according to the particular form of the generators $\mathcal{L}_i$. Furthermore, whether $\Gamma$ as defined in \eqref{eq:Gamma construction} belongs to $\mathcal{G}^0$ will also depend on the choice of the functions $\alpha_{ijkl}$.

It is not clear whether the construction presented in this section exhausts the class of coupling operators $\mathcal{G}$. We present this problem as a conjecture:

\begin{conjecture}
	\label{conj:general construction}
	For $i \in \{1,\ldots,n\}$, assume that we are given generators $(\mathcal{L}_{i}, \mathcal{D}(\mathcal{L}_i))$ of ergodic Feller semigroups. Then there exist decompositions of the form \eqref{eq:generator decomposition} and a set of functions
	\begin{equation}
	\mathcal{U} = \{\left(\alpha_{ijkl}\right)_{(i,j,k,l)\in \mathcal{J}}:\bar{E} \rightarrow \mathbb{R} \}
	\end{equation} 
	such that
	\begin{equation}
	\mathcal{G} = \left\{\Gamma = \sum_{(i,j,k,l) \in \mathcal{J}} \alpha_{ijkl} A_{k}^i A_{l}^j: \quad \alpha_{ijkl} \in \mathcal{U}\right\}.
	\end{equation}
\end{conjecture}
In the case when the underlying processes $(X^i_t)_{t \ge 0}$ are $\mathbb{R}^{d_i}$-valued (i.e. $E_i = \mathbb{R}^{d_i}$) and have continuous paths almost surely, Courr{\`e}ge's Theorem (\cite[Theorem  0.1]{Cour65}, see also \cite[Section 4.5]{Jac2001vol1} for a more recent account) provides an explicit characterisation of Feller generators. If furthermore these processes are ergodic with respect to given invariant measures, the decomposition of their generators into symmetric and antisymmetric part can be made explicit (see \cite[Theorem 1]{DNP2017}).  Combining these theorems, we obtain the following partial result:
\begin{proposition}
	\label{prop:general construction}
	Let $E_i = \mathbb{R}^{d_i}$ for positive integers $d_i \in \mathbb{N}$ and assume that the processes $(X_t^i)_{t \ge 0}$ are ergodic and solve the It{\^o} SDEs
	\begin{equation}
	\mathrm{d} X_t^i = b^i(X_t^i) \,\mathrm{d}t + \sqrt{2} \sigma^i(X_t^i)\,\mathrm{d}W_t^i, 
	\end{equation}
	where $b^i \in C^1(\mathbb{R}^{d_i},\mathbb{R}^{d_i})$, $\sigma \in C^1(\mathbb{R}^{d_i},\mathbb{R}^{d_i \times m_i})$, and $(W_t^i)_{t \ge 0}$ are standard $m_i$-dimensional Brownian motions. Then the conclusion of Conjecture \ref{conj:general construction} holds.
\end{proposition}
\section{Examples of coupled processes}
\label{ch:coupling_examples}
Here we will illustrate the framework developed in the last section with concrete examples. Throughout we consider the task of sampling from the measure
\begin{equation}
\label{eq:coupling_invariant measure}
\pi(\mathrm{d}x) = \frac{1}{Z}e^{-V(x)}\mathrm{d}x, \quad x \in \mathbb{R}^d,
\end{equation}
where $V\in C^{\infty}(\mathbb{R}^d)$ is a potential satisfying
\begin{equation}
Z:= \int_{\mathbb{R}^d} e^{-V(x)} \mathrm{d}x < \infty.
\end{equation} 
\subsection{Overdamped Langevin dynamics}
\label{sec:overdamped_all}
Our first group of examples is concerned with the overdamped Langevin dynamics \cite[Section 4.5]{pavliotis2014stochastic}. Let us start with the one-dimensional case, already encountered in the introduction.
\subsubsection{Two particles in one dimension}
\label{sec:overdamped 1d}
We consider $n=2$ particles moving in dimension $d=1$, each of them according to the dynamics
\begin{equation}
\label{eq:1d overdamped}
\mathrm{d}X_t = -V'(X_t) \,\mathrm{d}t + \sqrt{2}\,\mathrm{d}W_t.
\end{equation}
Note that in order to precisely fit into our framework developed in the previous section, the process $(X_t)_{t \ge 0}$ is required to be a Feller process according to Definition \ref{def:Feller}. This property can be guaranteed by imposing certain growth conditions on the potential $V$, see \cite[Proposition 5.9]{LMS2016} and \cite[Theorem 5.3.2, Example 5.3.3]{LM2007}. However, we wish to remark that the Feller property is not crucial in practice and dispensing with this regularity requirement still leads to perfectly well-defined couplings as will become clear in Lemma \ref{lem:1d BMs} below. The generator of \eqref{eq:1d overdamped} is given by
\begin{equation}
\label{eq:decomp 1d overdamped}
\mathcal{L}  = -V'(x)\partial_x + \partial_x^2  
= -\partial_x^* \partial_x,
\end{equation}
where the adjoint is taken in $L^2(\pi)$. In particular, $\mathcal{L}$ can naturally be written in the form \eqref{eq:generator decomposition}, with $A=\partial_x$ and $B=0$. To illustrate trivial couplings (see Example \ref{ex:independent coupling}), consider first two independent identical copies of \eqref{eq:1d overdamped}, denoted by $(X_t,Y_t)_{t \ge 0}$, hence evolving according to the dynamics
\begin{subequations}
	\label{eq:1d independent coupling}
	\begin{align}
	\mathrm{d}X_t & = -V'(X_t) \,\mathrm{d}t + \sqrt{2}\,\mathrm{d}W^x_t, \\
	\mathrm{d}Y_t & = -V'(Y_t) \,\mathrm{d}t + \sqrt{2}\,\mathrm{d}W^y_t,
	\end{align}
\end{subequations} 
on the product space $\bar{E} = \mathbb{R}^2$. Since for now the processes $(W^x_t)_{t \ge 0}$ and $(W^y_t)_{t \ge 0}$ are supposed to be two independent standard Brownian motions, the generator of \eqref{eq:1d independent coupling} is given by
\begin{equation}
\bar{\mathcal{L}}_0 = \mathcal{L}_x + \mathcal{L}_y,
\end{equation}
with $\mathcal{L}_x = -V'(x)\partial_x + \partial_x^2 = -\partial_x^* \partial_x$ and $\mathcal{L}_y  = -V'(y)\partial_y + \partial_y^2 = -\partial_y^* \partial_y$, in agreement with Example \ref{ex:independent coupling}, equation \eqref{eq:L0 def}. The invariant measure of \eqref{eq:1d independent coupling} is given by the product
\begin{equation}
\bar{\pi}_0 = \pi_x \otimes \pi_y = \frac{1}{Z^2}e^{-\left(V(x)+V(y)\right)} \mathrm{d}x \mathrm{d}y.
\end{equation}
Now let us consider nontrivial couplings. Following Section \ref{sec:general construction}, we may set
\begin{equation}
\label{eq:1d overdamped Gamma}
\Gamma = 2\alpha(x,y)\partial_x \partial_y,
\end{equation}
for an appropriate measurable function $\alpha:\mathbb{R}^2 \rightarrow \mathbb{R}$ (we have  inserted a factor of $2$ for convenience). According to Proposition \ref{prop:general construction}, the set of operators of the form \eqref{eq:1d overdamped Gamma} exhausts the set of coupling operators $\mathcal{G}$.
Clearly, the first condition of Definition \ref{def:coupling operators} is satisfied for this set of operators (this is already guaranteed by using the construction from Section \ref{sec:general construction}).
The second condition enforces
\begin{equation}
\label{eq:1d overdamped constraint}
-1 \le \alpha(x,y) \le 1, \quad \mbox{for all }x,y \in \mathbb{R}.
\end{equation} 
Indeed, observe that
\begin{equation}
\label{eq:1d overdamped full generator}
\bar{\mathcal{L}}_{\Gamma} = \bar{\mathcal{L}}_0 + \Gamma = \nabla_z U(z) \cdot \nabla_z + Q(z) : \nabla_z \nabla_z, 
\end{equation}
where $z=(x,y)$, $U(z) = U(x,y) = V(x) + V(y)$, 
\begin{equation}
\label{eq:1d overdamped Q}
Q(z) = Q(x,y) = 
\begin{pmatrix}
1 & \alpha(x,y) \\
\alpha(x,y) & 1
\end{pmatrix},
\end{equation}
and where $:$ denotes the Frobenius inner product of matrices.
According to Courr{\`e}ge's theorem, $\bar{\mathcal{L}}_\Gamma$  satisfies the positive maximum principle (required by the Hille-Yosida-Ray Theorem) only if $Q(z)$ is nonnegative definite for every $z \in \mathbb{R}^2$ . From this, we immediately deduce the constraint \eqref{eq:1d overdamped constraint}. 
\begin{remark}
	\label{rem:regularity alpha}
	We are deliberately vague about the regularity properties of $\alpha$. If we restrict our attention to Feller processes, $\alpha$ certainly has to be at least continuous, and there are multiple results in the literature guaranteeing the Feller property under mild further assumptions on $\alpha$, in particular, H{\"o}lder regularity  \cite{LMS2010,LM2007}. For a discussion of the martingale problem for generators with discontinuous coefficients see the recent preprint \cite{Kuehn2018} and references therein. Note that even in this simple case, it is very challenging to characterise exactly the set $\mathcal{G}$ as introduced in Definition \ref{def:coupling operators}.  In applications, however, the Feller property is not crucial. Lemma \ref{lem:1d BMs} below shows that measurability of $\alpha$ is sufficient to ensure that a reasonable coupled process can be constructed.
\end{remark}
Assuming that \eqref{eq:1d overdamped constraint} is satisfied, the dynamics induced by the generator \eqref{eq:1d overdamped full generator} are (at least formally) given by
\begin{equation}
\label{eq:1d coupled overdamped}
\mathrm{d}
\begin{pmatrix}
X_t \\ Y_t
\end{pmatrix} =
\begin{pmatrix}
-V'(X_t) \\ -V'(Y_t)
\end{pmatrix} \mathrm{d}t + \sqrt{2} G(X_t,Y_t) 
\begin{pmatrix}
\mathrm{d}W^x_t \\ \mathrm{d}W^y_t
\end{pmatrix},
\end{equation}
where $G(x,y)G(x,y)^T = Q(x,y)$, for instance
\begin{equation}
\label{eq:1d coupling matrix}
G(x,y) = 
\begin{pmatrix}
\cos \beta(x,y) & g(x,y)\sin \beta(x,y) \\	g(x,y)\sin \beta(x,y) & \cos \beta(x,y)
\end{pmatrix},
\end{equation}
with $\beta (x,y) = \frac{1}{2}\arcsin \vert \alpha(x,y)\vert$ and $g(x,y) = \sgn \alpha(x,y)$. We have chosen this parametrisation since it generalises readily to higher dimensions (see below). Let us stress that writing the dynamics in the form \eqref{eq:1d coupled overdamped} is vital for applications, since it enables its simulation in a straightforward manner. The following lemma shows that the process constructed in this way is indeed a coupling in the sense of Definition \ref{def:coupling}. Furthermore, it turns out that only minimal regularity of $\alpha$ is required.
\begin{lemma}
	\label{lem:1d BMs}
	Let $\alpha$ be measurable. Then the dynamics \eqref{eq:1d coupled overdamped} can be written as 
	\begin{equation}
	\label{eq:1d coupled BM}
	\mathrm{d}
	\begin{pmatrix}
	X_t \\ Y_t
	\end{pmatrix} = 
	\begin{pmatrix}
	-V'(X_t) \\ -V'(Y_t)
	\end{pmatrix} \mathrm{d}t + \sqrt{2}  
	\begin{pmatrix}
	\mathrm{d}B^x_t \\ \mathrm{d}B^y_t
	\end{pmatrix},
	\end{equation} 
	with two Brownian motions $(B^x_t)_{t \ge 0}$ and $(B^y_t)_{t \ge 0}$ that are in general not independent.
\end{lemma} 
\begin{proof}
	The claim follows from applying Lemma \ref{lem:random orthogonal transformations} to the components of the SDE
	\begin{equation}
	\begin{pmatrix}
	\mathrm{d}B^x_t \\ \mathrm{d}B^y_t
	\end{pmatrix} =  G
	\begin{pmatrix}
	\mathrm{d} W^x_t \\ \mathrm{d} W^y_t
	\end{pmatrix},
	\end{equation}
	noting that $G^x := (g \cos \beta  \, \sin \beta)$ and $G^y := (g \sin \beta  \, \cos \beta)$ indeed satisfy condition \eqref{eq:orthogonality condition} with $N=1$ and $M=2$.
\end{proof}
Note that the function $\alpha$ (equivalently the pair $\beta$ and $g$) encodes the coupling between the Brownian motions $(B^x_t)_{t \ge 0}$ and $(B^y_t)_{t \ge 0}$. The parameter $\beta \in [0,\frac{\pi}{4}]$ is related to the \emph{strength} of the coupling, whereas $g \in \{-1,1\}$ is related to its \emph{direction}.
Indeed, if $\beta \equiv 0$, then $(B_t^x)_{t \ge 0}$ and $(B_t^y)_{t \ge 0}$ are independent (`trivial coupling', see \eqref{eq:1d independent coupling} and Example \ref{ex:independent coupling}).
If $\beta \equiv \frac{\pi}{4}$ and $g \equiv 1$, then  $B^x_t = B^y_t$ for $t \ge 0$, almost surely (`synchronous coupling'). Likewise, if $\beta \equiv \frac{\pi}{4}$ and $g \equiv -1$, then $B^x_t = -B^y_t$ (`mirror coupling').
\begin{remark}[Ergodic couplings]
	\label{rem:erg loose couplings}
	If the bound \eqref{eq:1d overdamped constraint} is satisfied with strict inequalities, then the generator \eqref{eq:1d overdamped full generator} is elliptic, and hence, by Corollary \ref{cor:regular couplings}, the coupled process is ergodic. Let us mention that this condition is not necessary for ergodicity. Indeed, consider the coupling operator $\Gamma = -2 \partial_x \partial_y$, inducing the so-called `two-point motion' \cite{Baxendale1991}, i.e. $(X_t)_{t \ge 0}$ and $(Y_t)_{t \ge 0}$ are driven by the same Brownian motion, only differing by their initial laws. Under mild regularity conditions (i.e. Lipschitz continuity of the coefficients), it can be shown that $(X_t,Y_t)_{t \ge 0}$ is ergodic with respect to $\bar{\pi}_\Gamma = \frac{1}{Z} e^{-V(x)}\delta_{x-y}  (\mathrm{d}x\mathrm{d}y)$, see for instance \cite[Theorem 2.1]{LPP2015}.
\end{remark} 
\subsubsection{The general case} 
\label{ex:overdamped}
Here, we will extend the discussion from the previous section to the general case of $n$ particles moving in $d$ dimensions, i.e. we are concerned with couplings of the dynamics
\begin{equation}
\label{eq:overdamped trivial coupling}
\mathrm{d}X^i_t = -\nabla V(X^i_t)\,\mathrm{d}t + \sqrt{2} \,\mathrm{d}W^i_t, \quad i=1,\ldots, n,
\end{equation}
the processes $(X_t^i)_{t \ge 0}$ being $\mathbb{R}^d$-valued. If the Brownian motions $(W_t^i)_{t \ge 0}$ are independent, then the joint process $(\bar{X}_t)_{t \ge 0} = (X^1_t,\ldots,X^n_t)_{t \ge 0}$ is ergodic with respect to the product measure $\bar{\pi}_0 = \bigotimes_{i=1}^n \pi_i$ on $\mathbb{R}^{nd}$ (see Example \ref{ex:independent coupling})
and the corresponding generator is given by
\begin{equation}
\label{eq:generator overdamped}
\bar{\mathcal{L}}_0 = \sum_{i=1}^n \mathcal{L}_i, \quad \mathcal{L}_i = -\nabla V(x_i) \cdot \nabla_{x_i} + \Delta_{x_i} = -\sum_{k=1}^d (\partial_k^{i})^*\partial_k^{i}.
\end{equation}
Here, $\partial_k^{i}$ denotes the derivative with respect to the $k$-th component of $x_i$, and the adjoints are taken in the spaces $L^2(\pi_i)$. Clearly, the generators $\mathcal{L}_i$ are decomposed as in \eqref{eq:generator decomposition}, with $A_k^i = \partial_k^i$ and $B^i = 0$.
\begin{remark}
	\label{rem:general dynamics overdamped}
	Instead of \eqref{eq:overdamped trivial coupling}, we can also consider the more general dynamics
	\begin{equation}
	\mathrm{d}X^i_t = -Q_i(X^i_t)\nabla V(X^i_t)\,\mathrm{d}t + (\nabla \cdot Q_i)(X^i_t)\, \mathrm{d}t + J_i \nabla V(X_t^i)\,\mathrm{d}t +  \sqrt{2Q_i(X^i_t)}\,\mathrm{d}W^i_t,
	\end{equation} 
	with $i=1,\ldots,n$, $J_i \in \mathbb{R}^{d \times d}_{\mathrm{skew}}$ being skew-symmetric matrices and $Q_i :\mathbb{R}^d \rightarrow \mathbb{R}^{d \times d}_{\mathrm{sym}}$ being positive definite matrix-valued functions, as discussed in \cite[Section 2]{DNP2017}. 
	Note that in this case the processes $(X^i_t)_{t \ge 0}$ are not copies of each other, since $Q_i$ and $J_i$ may not be the same for different particles.
\end{remark}
To construct nontrivial couplings, we may set 
\begin{equation}
\label{eq:overdamped Gamma}
\Gamma = \sum_{i,j =1, i\neq j}^n \sum_{k,l=1}^d \alpha_{ijkl} \partial_k^i\partial_l^j
\end{equation}
for appropriate\footnote{Concerning the regularity of these functions, the discussion from the previous section applies, see in particular Remark \ref{rem:regularity alpha}.}  functions $\alpha_{ijkl}:\mathbb{R}^{nd} \rightarrow \mathbb{R}$, following Section \ref{sec:general construction}. Note that $\partial_k^i\partial_l^j$ is symmetric with respect to the interchange of indices $(i,k) \leftrightarrow (j,l)$, and so we may assume that $\alpha_{ijkl} = \alpha_{jilk}$. As in the one-dimensional case, the generator of the coupled system $\bar{\mathcal{L}}_{\Gamma} = \bar{\mathcal{L}}_0 + \Gamma$ is a second order differential operator which we require to be (possibly degenerately) elliptic in order for the second condition of Definition \ref{def:coupling operators} to be satisfied (again with reference to Courr{\`e}ge's theorem). 

To derive more easily verifiable conditions on the functions $\alpha_{ijkl}$, let us introduce a matrix-valued function (or \emph{matrix field}) $Q:(\mathbb{R}^d)^n \rightarrow \mathbb{R}^{nd \times nd}$ as follows. Firstly, it is helpful to view the target space of $Q$ as $\mathbb{R}^{nd \times nd} \equiv (\mathbb{R}^{d \times d})^{n \times n}$, i.e. we think of $Q(x_1,\ldots,x_n)$ as an $n \times n$-matrix the entries of which are themselves $d \times d$ matrices. In other words, $Q_{ij}(x_1,\ldots,x_n)$ is a $d\times d$ matrix for every pair $(i,j) \in \{1,\ldots,n\}^2$. The matrix field $Q$ can then be defined by
\begin{equation}
\label{eq:Q definition}
Q_{ij}(x_1,\ldots,x_n) = 
\begin{cases}
I_{d \times d},\quad & i=j, \\
\alpha_{ij}(x_1,\ldots,x_n),\quad & i \neq j,
\end{cases}
\end{equation}
where $\alpha_{ij}(x_1,\ldots,x_n) \in \mathbb{R}^{d \times d}$ denotes the matrix with entries $(\alpha_{ijkl}(x_1,\ldots,x_n))_{k,l=1,\ldots,d}$. Note that $Q(x_1,\ldots,x_n)$ as defined in \eqref{eq:Q definition} is symmetric as a matrix in $\mathbb{R}^{nd \times nd}$ by our assumption that $\alpha_{ijkl} = \alpha_{jilk}$. 
\begin{remark}
	The matrices $Q_{ij}$ can now be thought of as describing the coupling between the particles $i$ and $j$.
\end{remark}
As in the one-dimensional case, the generator of the fully coupled system can be written as
\begin{equation}
\label{eq:overdamped full generator}
\bar{\mathcal{L}}_{\Gamma} = \bar{\mathcal{L}}_0 + \Gamma = \nabla_z U(z) \cdot \nabla_z + Q(z) : \nabla_z \nabla_z, 
\end{equation}
introducing the notation $z\equiv (x_1,\ldots,x_n) \in \mathbb{R}^{nd \times nd}$ and $U(z) = \sum_{i=1}^n V(x_i)$. 
To make the connection to SDEs and arrive at a description analogous to \eqref{eq:1d coupled overdamped}, let us consider matrix fields $G:(\mathbb{R}^d)^n \rightarrow (\mathbb{R}^{d \times d})^{n \times n}$ satisfying
\begin{equation}
\label{eq:G constraint}
\sum_{j=1}^n G_{ij}(x_1,\ldots,x_n) G_{ij}(x_1,\ldots,x_n)^T = I_{d \times d},
\end{equation}
for all $i=1,\ldots,n$. Here the transposition $^T$ is taken in $\mathbb{R}^{d \times d}$. Note that matrix fields of this form give rise to the matrix fields $Q:(\mathbb{R}^d)^n \rightarrow \mathbb{R}^{nd \times nd}$ defined in \eqref{eq:Q definition} via $Q=GG^T$ (where the transposition is taken in $\mathbb{R}^{nd \times nd}$). 
\begin{remark}
	\label{rem:orthogonal matrices}
	The advantage of constructing the coupling in terms of the matrix field $G$ is that pointwise positive semi-definiteness of $Q$ is automatically satisfied.
	A practical way to fulfil the constraint \eqref{eq:G constraint} is to choose matrix fields $g_{ij} :(\mathbb{R}^d)^n \rightarrow \mathbb{R}^{d \times d}$ that are orthogonal pointwise, i.e.
	\begin{equation}
	\label{eq:g orthogonality}
	g_{ij}(x_1,\ldots,x_n) g_{ij}(x_1,\ldots,x_n)^T = g_{ij}(x_1,\ldots,x_n)^T g_{ij}(x_1,\ldots,x_n) = I_{d \times d},
	\end{equation} 
	for all $i,j=1,\ldots,n$ and $(x_1,\ldots,x_n) \in (\mathbb{R}^d)^n$, as well as weights $w_{ij}:(\mathbb{R}^d)^n \rightarrow \mathbb{R}$ satisfying
	\begin{equation}
	\label{eq:weights}
	\sum_{j=1}^n w_{ij}^2(x_1,\ldots,x_n) = 1,
	\end{equation}
	for all $i=1,\ldots,n$. Then, setting $G_{ij} = w_{ij} g_{ij}$, condition \eqref{eq:G constraint} holds. Intuitively, the orthogonal matrices $g_{ij}$ encode the coupling between  particle $i$ and $j$ through a rotation of the noise. The weights $w_{ij}$ can be interpreted as the relative coupling strengths between the particles. Observe that both $g_{ij}$ and $w_{ij}$ may depend on $(x_1,\ldots,x_n)$, i.e. on the locations of all the particles. This construction is a direct generalisation of \eqref{eq:1d coupling matrix}. 
\end{remark}
Assuming $GG^T=Q$, the dynamics associated to the generator \eqref{eq:overdamped full generator} is (again, at least formally) given by
\begin{equation}
\label{eq:overdamped full SDE}
\mathrm{d}X_t^i = - \nabla V(X_t^i) \,\mathrm{d}t + \sqrt{2}\sum_{j=1}^n G_{ij}(X_t^1,\ldots,X_t^n) \,\mathrm{d}W_t^j, \quad i=1,\ldots,n,
\end{equation} 
where $\left((W_t^j)_{t\ge 0}\right)_{j=1}^n$ are assumed to be independent standard Brownian motions. As in the one-dimensional case, we have the following lemma:
\begin{lemma}
	\label{lem:overdamped BM}
	There exist $\mathbb{R}^d$-valued standard Brownian motions $\left((B_t^j)_{t \ge 0}\right)_{ j=1}^n$, not necessarily independent, such that the dynamics \eqref{eq:overdamped full SDE} can be written as
	\begin{equation}
	\label{eq:overdamped full BM}
	\mathrm{d}X_t^i = - \nabla V(X_t^i)\, \mathrm{d}t + \sqrt{2} \,\mathrm{d}B_t^i, \quad i=1,\ldots,n.
	\end{equation}
\end{lemma}
\begin{proof}
	The argument is identical to the one used in the proof of Lemma \ref{lem:1d BMs}.
\end{proof}
Concerning the ergodicity of couplings, we have similar findings to those of Remark \ref{rem:erg loose couplings}. If the matrix field \eqref{eq:Q definition} is positive definite at every point (i.e. the generator \eqref{eq:overdamped full generator} is elliptic), then Corollary \ref{cor:regular couplings} implies that the coupled process is ergodic. An example of nonergodic couplings can be found in \cite[Section 3.1]{LPP2015}. 
\begin{example}[Two particles]
	\label{ex:overdamped 2 particles}
	The foregoing constructions become more explicit when considering only $n=2$ particles. To simplify the notation, we denote their positions by $x \equiv x_1$ an $y \equiv x_2$.
	The diffusion matrices \eqref{eq:Q definition} reduce to
	\begin{equation}
	Q(x,y) = 
	\begin{pmatrix}
	I_{d \times d} & \alpha(x,y) \\
	\alpha^T(x,y) & I_{d \times d}
	\end{pmatrix},
	\end{equation}
	with $\alpha:(\mathbb{R}^d)^2 \rightarrow \mathbb{R}^{d \times d}$. Using Schur complements \cite[Appendix 5.5]{BV2004}, we see that $Q(x,y)$ is positive semidefinite if and only if
	\begin{equation}
	\label{eq:overdamped 2 alpha bound}
	\alpha(x,y)^T  \alpha(x,y) \le I_{d \times d},
	\end{equation}
	in the sense of symmetric matrices (Loewner ordering). The corresponding coupling operator is given by
	\begin{equation}
	\label{eq:Gamma 2 overdamped}
	(\Gamma f)(x,y) = 2\Tr\left(\alpha^T(x,y) \nabla^2_{xy}f(x,y)\right), \quad f \in C_c^{\infty}(\mathbb{R}^d \times \mathbb{R}^d),
	\end{equation} where the matrix of mixed derivatives $\nabla^2_{xy}\phi$ is given by
	$
	\left(\nabla_{xy}^2 \phi \right)_{ij} = (\partial_{x_i}\partial_{y_j}\phi)_{ij}
	$.
	In order to illustrate the construction from Remark \ref{rem:orthogonal matrices}, let $g:(\mathbb{R}^d \times \mathbb{R}^d) \rightarrow \mathbb{R}^{d \times d}$ be a field of orthogonal matrices (see \eqref{eq:g orthogonality}) and put 
	\begin{equation}
	G(x,y) =
	\begin{pmatrix}
	\cos \beta(x,y) I_{d \times d} & \sin \beta (x,y) g(x,y) \\
	\sin \beta(x,y)g^T(x,y) & \cos \beta (x,y) I_{d \times d}
	\end{pmatrix},
	\end{equation}
	the function $\beta:\mathbb{R}^d \times \mathbb{R}^d \rightarrow [0,\frac{\pi}{4}]$ again regulating the strength of the coupling, and being associated with the weights $w_{ij}$ in \eqref{eq:weights}. From $GG^T = Q$ it follows that $\alpha$ and $g$ are connected via 
	\begin{equation}
	\label{eq:alpha g}
	\alpha = 2 \cos \beta \sin \beta \cdot g.
	\end{equation}
\end{example}

\subsection{Underdamped Langevin dynamics}
\label{ex:underdamped}
For fixed $\gamma > 0$ (`friction') and symmetric positive definite $M \in \mathbb{R}^{d\times d}_{\mathrm{sym}}$ (`mass'), the dynamics 
\begin{subequations}
	\begin{align}
	\mathrm{d}q_t & =  M^{-1}p_t \,\mathrm{d}t, \\
	\mathrm{d}p_t & =  -\nabla V(q_t) \,\mathrm{d}t - \gamma p_t \,\mathrm{d}t + \sqrt{2\gamma} \,\mathrm{d} W_t,
	\end{align}
\end{subequations}
is ergodic with respect to the measure
\begin{equation}
\pi = \frac{1}{\bar{Z}}e^{-\left(V(q) + \frac{1}{2}p^T M^{-1}p\right)} \mathrm{d}q\mathrm{d}p, \quad (q,p) \in \mathbb{R}^{2d},
\end{equation}
$\bar{Z}$ being an appropriate normalisation constant (see \cite[Chapter 6]{pavliotis2014stochastic} for details).
The generator is given by
\begin{subequations}
	\begin{align}
	\label{eq:generator underdamped}
	\mathcal{L} & = M^{-1}p \cdot \nabla_q - \nabla_q V(q) \cdot \nabla_p + \gamma (-p\cdot \nabla_p + \Delta_p) \\
	\label{eq:decomp underdamped}
	& = B - \sum_{k=1}^d A_k^* A_k,
	\end{align}
\end{subequations}	
where $\mathcal{T} = M^{-1}p \cdot \nabla_q - \nabla_q V(q) \cdot \nabla_p$ is skew-symmetric in $L^2(\pi)$, $A_k = \sqrt{\gamma} \partial_{p_k}$, and the adjoint is taken in $L^2(\pi)$. To construct a coupled sampler of $n$ processes, we may proceed as in the overdamped case and set
\begin{equation}
\Gamma = \sum_{i,j =1, i\neq j}^n \sum_{k,l=1}^d \alpha_{ijkl} \partial_{p_k}^i\partial_{p_l}^j,
\end{equation} 
for appropriate functions $\alpha_{ijkl}:(\mathbb{R}^{2d})^n \rightarrow \mathbb{R}$, denoting by $\partial_{p_k}^i$ the derivative with respect to the $k$-th component of $p$ of the $i$-th particle. Following very closely the discussion in Section  \ref{sec:overdamped_all}, we can introduce matrix fields $Q:(\mathbb{R}^{2d})^n \rightarrow (\mathbb{R}^{d \times d})^{ n \times n}$ and $G:(\mathbb{R}^{2d})^n \rightarrow (\mathbb{R}^{d \times d})^{ n \times n}$ satisfying \eqref{eq:Q definition} and \eqref{eq:G constraint} (with $(x_1,\ldots,x_n)$ replaced by $(q_1,p_1,\ldots,q_n,p_n)$) such that the generator of the coupled system is given by 
\begin{equation}
\label{eq:underdamped coupling operator}
\bar{\mathcal{L}}_{\Gamma} = \sum_{i=1}^n \left( M^{-1} p_i \cdot \nabla_{q_i} - \nabla_q V(q_i) \cdot \nabla_{p_i} - \gamma p_i \cdot \nabla_{p_i}\right) + Q(z):\nabla_{z_p} \nabla_{z_p},
\end{equation}
making again use of the notations $z \equiv (q_1,p_1,\ldots,q_n,p_n)$, $z_p \equiv (p_1,\ldots,p_n)$ and such that the associated dynamics are given by
\begin{subequations}
	\begin{align}
	\mathrm{d}q^i_t & =  M^{-1}p^i_t \,\mathrm{d}t, \\
	\mathrm{d}p^i_t & =  -\nabla_q V(q^i_t) \,\mathrm{d}t - \gamma p^i_t \,\mathrm{d}t + \sqrt{2\gamma} \sum_{j=1}^n G_{ij}(q_1,p_1,\ldots,q_n,p_n) \,\mathrm{d}W^j_t,
	\end{align}
\end{subequations}
for $i=1,\ldots,n$. Also in this case, it is straightforward to see that an appropriate version of the Lemmas \ref{lem:1d BMs} and \ref{lem:overdamped BM} holds.
\begin{remark}
	Generalising the above to the case where the particles have different frictions $\gamma_i$ and masses $M_i$ or some or all of them move according to perturbed versions of underdamped Langevin dynamics as considered in \cite{DNP2017} is straightforward (see also Remark \ref{rem:general dynamics overdamped}).
\end{remark} 
\begin{example}[Overdamped and underdamped Langevin dynamics]
	\label{ex:overdamped underdamped}
	It is also possible to couple different types of dynamics (such as overdamped and underdamped Langevin dynamics). For instance, consider the generators
	\begin{subequations}
		\begin{align}
		\mathcal{L}_1 & = -\nabla V(x) \cdot\nabla_x + \Delta_x, \\
		\mathcal{L}_2 & = M^{-1}p \cdot \nabla_q - \nabla_q V(q) \cdot \nabla_p + \gamma (-p\cdot \nabla_p + \Delta_p),
		\end{align}
	\end{subequations}
	as in \eqref{eq:generator overdamped} and \eqref{eq:generator underdamped}.
	Setting 
	\begin{equation}
	\Gamma = \sum_{k,l=1}^d \alpha_{kl}(x,q,p) \partial_{x_k}\partial_{p_l}
	\end{equation}
	and following along the lines of Sections \ref{sec:overdamped_all} and \ref{ex:underdamped} will result in the coupled dynamics
	\begin{subequations}
		\begin{align}
		\mathrm{d}X_t & = -\nabla V(X_t) \,\mathrm{d}t + \sqrt{2} \mathrm{d}B^{(1)}_t, \\
		\mathrm{d}q_t & = M^{-1}p_t \,\mathrm{d} t, \\
		\mathrm{d}p_t & = - \nabla_q V(q) \mathrm{d}t - \gamma p_t \,\mathrm{d}t + \sqrt{2\gamma}\,\mathrm{d}B_t^{(2)},
		\end{align}
	\end{subequations}
	where the Brownian motions $(B_t^{(1)})_{t \ge 0}$ and $(B_t^{(2)})_{t \ge 0}$ are in general not independent (and the exact dependence results from the choice of the functions $\alpha_{kl}$).
\end{example}
\subsection{The zigzag process} 
\label{sec:zigzag}
In recent years, there has been a growing interest in using piecewise deterministic Markov processes (PDMPs) \cite{Davis1984} in the context of sampling problems. These are processes that move deterministically between random events, usually along the trajectories of an ODE. At those events, a random transition (e.g. a `jump') occurs. Both the deterministic dynamics as well as the random transitions can be chosen with a great deal of flexibility, resulting in a range of possible PDMP algorithms. Let us mention here the bouncy particle sampler (BPS) \cite{BVD2017}, the zigzag sampler \cite{BFR2016}, randomised Hamiltonian Monte Carlo \cite{BouRabee17},  and event-chain Monte Carlo techniques \cite{MKK2014,Michel17}. The recent papers \cite{fearnhead2016piecewise} and \cite{VBDD2017} provide good overviews in a general framework.

The objective of this section is to show how the framework from Section \ref{sec:general framework} can be employed in the construction of coupled samplers from piecewise deterministic Markov processes, using the example of the zigzag process.
For ease of exposition, we furthermore restrict our attention to the one-dimensional case. The treatment here follows \cite{BFR2016} and \cite{BRZ2018} in style and notation.

The state space under consideration is $E = \mathbb{R} \times \{-1,+1\}$, and the generator of the zigzag process reads
\begin{equation}
\label{eq:zigzag generator}
\mathcal{L} f(x,\theta) = \theta \partial_x f(x,\theta) + \lambda(x,\theta) \left(f(x,-\theta) - f(x,\theta)\right), \quad f \in C_c^{\infty}(E),
\end{equation} 
where the \emph{switching rate} $\lambda$ is given by
\begin{equation}
\lambda(x,\theta) = \max(0,\theta V'(x)) + \gamma(x).
\end{equation}
Here, $\gamma:\mathbb{R} \rightarrow \mathbb{R}_{\ge 0}$ is a nonnegative continuous function, called the \emph{excess switching rate}.
Roughly speaking, the zigzag process moves along straight lines in the direction determined by $\theta \in \{-1,1\}$. At random times a switch occurs, i.e. $\theta$ is replaced by $-\theta$. Those events are sampled according to the switching rate $\lambda$, i.e. at a point $(x,\theta) \in E$, the probability for the switch $\theta \mapsto -\theta$ in the time span $[t,t + \varepsilon]$ is given by $\lambda(x,t) \varepsilon + o(\varepsilon)$. For more details on the construction and simulation of zigzag processes we refer the reader to \cite{BD2017} an \cite{BFR2016}, as well as to \cite{Davis1984} for more general piecewise deterministic Markov processes. According to \cite[Proposition 1]{BD2017} and \cite[Appendix B.2]{VBDD2017}, the zigzag process satisfies the Feller property;  for more general piecewise deterministic Markov processes this topic has been studied in \cite[see Theorem 27.6]{Davis1993}.
The measure
\begin{equation}
\label{eq:zigzag invariant measure}
\pi = \frac{1}{2Z} e^{-V(x)} \mathrm{d}x \otimes (\delta_{-1} + \delta_{+1})
\end{equation}
is invariant, and, under some additional assumptions\footnote{See \cite{BD2017} and \cite{BRZ2018}. Let us mention in particular that ergodicity is guaranteed whenever the excess switching rate $\gamma$ is strictly positive.}, ergodic.
The generator \eqref{eq:zigzag generator} can be decomposed in the form
\begin{equation}
\mathcal{L} = -A^*A + B,
\end{equation}
where
\begin{equation}
\label{eq:zigzag decomp}
A= \left( \frac{1}{4}\vert V'\vert + \frac{1}{2}\gamma \right)^{1/2}(R-1), \quad B = \theta\partial_x + \frac{1}{2}\theta V'(R-1),
\end{equation}
and the `flip operator' $R$ is given by
\begin{equation}
\quad (Rf)(x,\theta) = f(x,-\theta), \quad f \in C_c^{\infty}(\bar{E}).
\end{equation}
A short calculation shows that indeed $B$ is antisymmetric in $L^2(\pi)$, whereas $A$ is symmetric.

To construct a coupled sampler from two zigzag processes, let us introduce the following notation:
We consider the state space $\bar{E} = \mathbb{R}^2 \times \{-1,+1\}^2$, denoting its elements by $(x,y,\theta_x,\theta_y)$. Furthermore, we will make use of the flip operators
\begin{equation}
(R_x f)(x,y,\theta_x,\theta_y) = f(x,y,-\theta_x,\theta_y), \quad 
(R_y f)(x,y,\theta_x,\theta_y) = f(x,y,\theta_x,-\theta_y).
\end{equation}
Following Section \ref{sec:general construction}, let us set
\begin{equation}
\Gamma = \alpha(x,y,\theta_x,\theta_y) (R_x - 1)(R_y - 1),
\end{equation}
for an appropriate function $\alpha:\bar{E} \rightarrow \mathbb{R}$, i.e. $\Gamma$ acts as
\begin{subequations}
	\label{eq:zigzag Gamma}
	\begin{align}
	(\Gamma f)&(x,y,\theta_x,\theta_y) =  \alpha(x,y,\theta_x,\theta_y) \cdot  \tag{\ref*{eq:zigzag Gamma}}\\
	& \cdot \left(f(x,y,\theta_x,\theta_y) - f(x,y,-\theta_x,\theta_y) - f(x,y,\theta_x,-\theta_y) + f(x,y,-\theta_x,-\theta_y)\right)  \nonumber
	\end{align}
\end{subequations}
on test functions $f \in C_c^{\infty}(\bar{E})$. Note that $\Gamma$ vanishes on functions that either depend on only $x$ and $\theta_x$ or only on $y$ and $\theta_y$.
The next task is to obtain bounds on $\alpha$ that ensure that the second condition in Definition \ref{def:coupling operators} is satisfied. To this end, let us expand
\begin{subequations}
	\label{eq:generator 1d zigzag}
	\begin{align}
	(\bar{\mathcal{L}}_{\Gamma}f)(x,y,\theta_x,\theta_y) =\, &(\mathcal{L}_x + \mathcal{L}_y + \Gamma)f(x,y,\theta_x,\theta_y) 
	\\
	=\, &\theta_x \partial_x f(x,y,\theta_x,\theta_y) + \theta_y \partial_y f(x,y,\theta_x,\theta_y)
	\label{eq:zigzag derivatives}
	\\
	&-\left(\lambda(x,\theta_x) + \lambda(y,\theta_y) - \alpha(x,y,\theta_x,\theta_y)\right) f(x,y,\theta_x,\theta_y)
	\label{eq:lambda alpha1}
	\\
	&+ \left(\lambda(x,\theta_x) - \alpha(x,y,\theta_x,\theta_y) \right) f(x,y,-\theta_x,\theta_y) \\
	&+ \left(\lambda(y,\theta_y) - \alpha(x,y,\theta_x,\theta_y) \right) f(x,y,\theta_x,-\theta_y)
	\label{eq:lambda alpha3}
	\\
	\label{eq:double switch}
	&+\alpha(x,y,\theta_x,\theta_y)f(x,y,-\theta_x,-\theta_y).
	\end{align}
\end{subequations}
For \eqref{eq:generator 1d zigzag} to be the generator of a Markov process (in particular, for it to satisfy the positive maximum principle), the following inequalities have to be satified:
\begin{subequations}
	\label{eq:nonnegativity conditions}
	\begin{align}
	\lambda(x,\theta_x)  + \lambda(y,\theta_y) - \alpha(x,y,\theta_x,\theta_y) & \ge 0, \\
	\lambda(x,\theta_x) - \alpha(x,y,\theta_x,\theta_y) & \ge 0, \\
	\lambda(y,\theta_y) - \alpha(x,y,\theta_x,\theta_y) & \ge 0, \\
	\alpha(x,y,\theta_x,\theta_y) & \ge 0, \quad\quad (x,y,\theta_x,\theta_y) \in \bar{E}.
	\end{align}
\end{subequations}
These conditions can be interpreted as saying that the transition probabilities for the coupled piecewise deterministic Markov process cannot be negative. Clearly, the conditions \eqref{eq:nonnegativity conditions} are equivalent to 
\begin{equation}
\label{eq:zig zag constraint}
0 \le \alpha(x,y,\theta_x,\theta_y) \le \min \left(\lambda(x,\theta_x),\lambda(y,\theta_y)\right), \quad (x,y,\theta_x,\theta_y) \in \bar{E}.
\end{equation}
Let us briefly comment on the dynamical behaviour that the coupling operator \eqref{eq:zigzag Gamma} introduces. As can be seen from \eqref{eq:double switch}, $\alpha$ is connected to `double flips', i.e. the event that both particles change their directions at the same time. Setting $\alpha$ to either the lower or the upper bound in \eqref{eq:zig zag constraint} will either discourage or encourage those double flips. As in the case of the overdamped and underdamped Langevin dynamics, the coupling behaviour (encoded in $\alpha$) is allowed to depend on the point $(x,y,\theta_x,\theta_y) \in \bar{E}$. We also remark that the process generated by $\bar{\mathcal{L}}_\Gamma$ as in \eqref{eq:generator 1d zigzag} can be simulated conveniently by using the methods summarised in \cite[Appendix B]{BD2017}. 
\begin{remark}
	The construction in this section can be generalised to couplings of multiple zigzag processes in arbitrary dimensions by following a similar approach to the one taken in Section \ref{sec:overdamped_all}.
\end{remark}
\section{Asymptotic variance and optimal transport}
\label{ch:coupling_asymvar}
In the following we analyse the asymptotic variance associated to estimators based on coupled processes (Section \ref{sec:coupling_CLT}) and connect the result to the theory of optimal transportation (Section \ref{sec:OT}).
\subsection{A central limit theorem for coupled processes}
\label{sec:coupling_CLT}
The objective of this section is to establish a central limit theorem characterising the convergence in \eqref{eq:extended ergodicity} and to find an expression for the associated asymptotic variance in terms of ergodic coupling operators $\Gamma \in \mathcal{G}^0$ and invariant measures $\bar{\pi}_\Gamma$. In particular, our aim is to compare between estimators based on couplings (as in \eqref{eq:extended ergodicity}) and the one-particle estimators \eqref{eq:ergodic averages}.
Naturally, Assumption \ref{ass:marginal ergodicity} is still in force. Moreover, let us assume the following:
\begin{assumption}[Invertibility of the one-particle generators]
	\label{ass:one-particle decay}
	The generators $\mathcal{L}_i$ are invertible on $L_0^2(\pi_i)$, i.e. for all $f \in L_0^2(\pi_i)$ there exists $\phi \in \mathcal{D}(\mathcal{L}_i) \cap L_0^2(\pi_i)$ such that
	\begin{equation}
	\label{eq:invertibility}
	-\mathcal{L}_i \phi = f.
	\end{equation}
\end{assumption}
It is well-known that the validity of the foregoing assumption is guaranteed by sufficiently fast decay of the semigroups $(S_t^i)_{t \ge 0}$ in $L^2(\pi_i)$, see for instance \cite{KomorowskiLandimOlla2012}.
In the following, let us fix observables of interest $f_i \in L_0^2(\pi_i)$ and denote the corresponding solutions to the Poisson equations \eqref{eq:invertibility} by $\phi_i$. Supposing $X^i_0 \sim \pi_i$, Assumption \ref{ass:one-particle decay} implies the central limit theorems
\begin{equation}
\label{eq:coupling_CLT}
\sqrt{T} \left( \frac{1}{T} \int_0^T f_i(X^i_t) \,\mathrm{d}t - \pi_i(f_i) \right) \xrightarrow[T \rightarrow \infty]{d} \mathcal{N}(0,2\sigma_{f_i}^2),
\end{equation}
where the asymptotic variances are given by
\begin{equation}
\label{eq:asymvar_onep}
\sigma_{f_i}^2 = \langle f_i, \phi_i \rangle_{L^2(\pi_i)},
\end{equation} 
see \cite{B1982,KomorowskiLandimOlla2012}.

We will now establish a similar central limit theorem for the coupled process $(\bar{X}_t)_{t \ge 0}$ induced by ergodic coupling operators $\Gamma \in \mathcal{G}^0$ and associated to extended observables of the form
\begin{equation}
F = \frac{1}{n} \sum_{i=1}^n f_i.
\end{equation}
\begin{theorem}
	\label{thm:asym var}
	\emph{Central limit theorem for coupled processes.}
	Let Assumption \ref{ass:one-particle decay} be satisfied and assume that $\Gamma \in \mathcal{G}^0$. Furthermore, let $\bar{X}_0 \sim \bar{\pi}_\Gamma$.
	Then
	\begin{equation}
	\label{eq:extended_CLT}
	\sqrt{T}\left(\frac{1}{T}\int_0^T F(\bar{X}_t)\,\mathrm{d}t - \bar{\pi}_{\Gamma}(F) \right) \xrightarrow[T \rightarrow \infty]{d} \mathcal{N}(0,2\sigma_{F}^2).
	\end{equation} 
	The asymptotic variance $\sigma_F^2$ is given by
	\begin{equation}
	\label{eq:asym_var}
	\sigma_F^2 = \frac{1}{n^2} \sum_{i=1}^n \sigma_{f_i}^2 - \int_{\bar{E}} \bar{\mathcal{L}}_0 \xi \,\mathrm{d} \bar{\pi}_{\Gamma}, 
	\end{equation}
	where
	\begin{equation}
	\label{eq:definition xi}
	\xi = \frac{1}{n^2}\sum_{i < j} \phi_i \phi_j,
	\end{equation}
	the functions $(\phi_i)_{i = 1}^n$ being the solutions to the Poisson equations \eqref{eq:invertibility}.
\end{theorem}

\begin{remark}
	Note that in the case of the trivial coupling $\Gamma = \mathbf{0}$ (see Example \ref{ex:independent coupling}), we have that
	\begin{equation}
	\int_{\bar{E}} \bar{\mathcal{L}}_0 \xi \,\mathrm{d} \bar{\pi}_{0} = 0,
	\end{equation}
	since $\bar{\pi}_0$ is the invariant measure associated to the process generated by $\bar{\mathcal{L}}_0$. Therefore in this case, we obtain the result that the asymptotic variance $\sigma_F^2$ is given by the arithmetic mean of the asymptotic variances $\sigma_{f_i}^2$ of the one-particle processes, divided by $n$. This result is expected, since the computational cost of computing the evolution of the processes $(X^i_t)_{t \ge 0}$ is likewise increased by a factor of $n$.
\end{remark}
\begin{remark}
	Observe furthermore that
	\begin{equation}
	\int_{\bar{E}} \bar{\mathcal{L}}_{\Gamma} \xi \,\mathrm{d}\bar{\pi}_{\Gamma} = 0,
	\end{equation}
	hence using $\bar{\mathcal{L}}_{\Gamma} = \bar{\mathcal{L}}_0 + \Gamma$
	we can equivalently express the asymptotic variance as
	\begin{equation}
	\label{eq:asym var Gamma}
	\sigma_F^2 = \frac{1}{n^2} \sum_{i=1}^n \sigma_{f_i}^2 + \int_{\bar{E}} \Gamma \xi \,\mathrm{d} \bar{\pi}_{\Gamma}.
	\end{equation}
\end{remark}
\begin{proof}[Proof of Theorem \ref{thm:asym var}]
	First observe that by the fact that $\bar{\pi}_\Gamma$ is a coupling of $(\pi_i)_{i=1}^n$, we have that $\bar{\pi}_\Gamma(F) = 0$. The Poisson equation
	\begin{equation}
	\label{eq:extended_Poisson}
	-\bar{\mathcal{L}}_{\Gamma} \Phi = F, \quad \bar{\pi}_{\Gamma}(\Phi) = 0, 	
	\end{equation}
	has a solution given by
	\begin{equation}
	\Phi = \frac{1}{n} \sum_{i=1}^n \phi_i,
	\end{equation}
	resting on the fact that $\Gamma \Phi = 0$ by the first condition of Definition \ref{def:coupling operators}. Again, the condition $\bar{\pi}_\Gamma(\Phi) = 0$ is satisfied by the coupling property of $\bar{\pi}_\Gamma$. Using \cite[Theorem 2.1]{B1982}, we see that the  central limit theorem \eqref{eq:extended_CLT} holds with asymptotic variance
	\begin{equation}
	\sigma_F^2 = \langle F, \Phi \rangle_{L^2(\bar{\pi}_{\Gamma})}.
	\end{equation}
	Expanding the above yields
	\begin{subequations}
		\begin{eqnarray}
		\sigma_{F}^{2} & = & \int_{\bar{E}} \left( \frac{1}{n}  \sum_{i=1}^n f_i \right) \left( \frac{1}{n} \sum_{j=1}^n \phi_j \right) \mathrm{d} \bar{\pi}_{\Gamma}  \nonumber 
		=  \frac{1}{n^2} \sum_{i=1}^n \int_{\bar{E}} f_i \phi_i \,\mathrm{d} \bar{\pi}_{\Gamma} + \frac{1}{n^2} \sum_{\substack{i,j = 1 \\ i \neq j}}^n \int_{\bar{E}} f_i \phi_j \,\mathrm{d} \bar{\pi}_{\Gamma} \nonumber \\
		& = & \frac{1}{n^2} \sum_{i = 1}^n \sigma_{f_i}^2 - \int_{\bar{E}} \bar{\mathcal{L}}_{0} \xi \,\mathrm{d}\bar{\pi}_{\Gamma}, \nonumber  
		\end{eqnarray}
	\end{subequations}
	where in the last equation we used the fact that $\bar{\pi}_{\Gamma}$ has marginal $\pi_i$ in the $i$-th coordinate, expression \eqref{eq:asymvar_onep}, as well as the definition of $\xi$ in \eqref{eq:definition xi}. 
\end{proof}
\subsection{Connections to the theory of optimal transportation}
\label{sec:OT}

In this section we will always assume that Assumption \ref{ass:one-particle decay} is satisfied, so that the central limit theorems from the previous section hold.
Theorem \ref{thm:asym var} then shows that, in order to reduce the asymptotic variance, we are led to the problem of minimising the expression
\begin{equation*}
\int_{\bar{E}} (-\bar{\mathcal{L}}_0 \xi)\, \mathrm{d}\bar{\pi}_{\Gamma}.
\end{equation*}
Remarkably, this expression depends on $\Gamma$ through the measure $\bar{\pi}_{\Gamma}$ only\footnote{Another way of saying this is that the map $\Gamma \mapsto \sigma_F^2$ factors through the map $\Gamma \mapsto \bar{\pi}_\Gamma$, i.e. with respect to the asymptotic variance, no information is lost by considering only the invariant measure of the joint process.}. We provide a sketch of this situation in Figure \ref{fig:commutative}. 
\begin{figure}
	\centering
	\large{
		
		\begin{tikzpicture}[every node/.style={midway}]
		
		\matrix[column sep={8em,between origins},
		row sep={8em}] at (0,0)		
		{ \node(fp)   {$\mathcal{G}^0 \ni \Gamma$}  ; & \node(gen) {$\bar{\pi}_\Gamma \in \mathcal{C}^0(\pi_1,\ldots,\pi_n)$}; \\
			
			\node(schr) {$\sigma_F^2(\bar{\mathcal{L}}_{\Gamma})$};                   \\};
		
		\draw[<-|] (schr) -- (fp) node[anchor=east]  {};
		
		\draw[<-|] (schr) -- (gen) node[anchor=north]  {};
		
		\draw[|->] (fp)   -- (gen) node[anchor=south] {};
		
		\end{tikzpicture}}
	
	\caption{Relationship between ergodic coupling operators, admissible couplings between the marginal invariant measures and the associated asymptotic variance. The diagramme commutes, in particular, all the information relevant for computing the asymptotic variance is contained in the invariant measure.}
	\label{fig:commutative}
\end{figure}
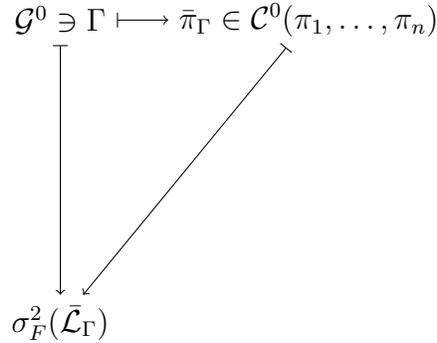
Since $\bar{\pi}_\Gamma$ has fixed marginals (i.e. they do not depend on $\Gamma$), this task is very reminiscent of the Kantorovich problem \cite[Chapter 1]{V2009} appearing in the theory of optimal transportation \cite{V2003,V2009}. To make this connection more precise, let us introduce the following terminology:

\begin{definition}[Admissible couplings]
	\label{def:admissible couplings}
	The set of couplings of the marginal invariant measures $(\pi_i)_{i=1}^n$ will be denoted by $\mathcal{C}$.
	A coupling $\bar{\pi} \in \mathcal{C}$ is called \emph{admissible}, if it arises as the invariant measure of an ergodically coupled process, i.e. if there exists $\Gamma \in \mathcal{G}^0$ such that 
	\begin{equation}
	\int_{\bar{E}} \bar{\mathcal{L}}_\Gamma f\,\mathrm{d}\bar{\pi} = 0,
	\end{equation}
	for all $f \in \mathcal{D}(\bar{\mathcal{L}}_\Gamma)$. The set of admissible measures will be denoted by $\mathcal{C}^0$, or, stressing the dependence on the marginal measures, by $\mathcal{C}^0(\pi_1,\ldots,\pi_n)$.
\end{definition}
Our aim in this section can be summarised in the following form, only replacing $\mathcal{C}$ by the subset $\mathcal{C}^0$ in the standard formulation of the Kantorovich problem:
\begin{problem}
	\label{prob:OT C0}
	For a fixed cost function $c \in C_b(\bar{E})$, find $\bar{\pi} \in \mathcal{C}^0$ such that
	\begin{equation}
	\label{eq:OT problem C0}
	\bar{\pi} \in \argmin_{\bar{\pi} \in \mathcal{C}^0} \int_{\bar{E}} c \, \mathrm{d}\bar{\pi}.
	\end{equation}
	Equivalently, find minimisers of the function
	\begin{equation}
	\label{eq:OT objective}
	\mathcal{G}^0 \rightarrow \mathbb{R}, \quad \Gamma \mapsto \int_{\bar{E}} c \, \mathrm{d}\bar{\pi}_\Gamma.
	\end{equation} 
\end{problem}
\begin{remark}
	As already pointed out, setting
	$c = -\bar{\mathcal{L}}_0\xi,
	$
	with $\xi$ as defined in \eqref{eq:definition xi},
	is equivalent to the problem of optimising the asymptotic variance for a particular observable. Other choices for $c$ might be of interest. For instance, one might aim to optimise the asymptotic variance across a set of observables simultaneously. In this case, it seems reasonable to consider cost functions of the form $c = \sum_{j} c_j$, where $c_j$ is the cost function associated with the $j$th observable.
	Assuming that all the particles evolve in the same state space $E$, another natural objective would be to maximise the average distance of the particles at equilibrium, leading to a cost function of the type
	\begin{equation}
	c(x_1,\ldots,x_n) = -\sum_{\substack{i,j = 1 \\ i \neq j}}^n d(x_i,x_j),
	\end{equation}
	for some metric $d$ on $E$. More generally, for some function $g: \mathbb{R} \rightarrow \mathbb{R}$ it might be worthwhile to consider
	\begin{equation}
	c(x_1,\ldots,x_n) = -\sum_{\substack{i,j = 1 \\ i \neq j}}^n g(d(x_i,x_j)).
	\end{equation}
	A cost function of this type would be reasonable if one aims to use the empirical measure of an ensemble of particles in order to precondition the dynamics (see \cite{LMW2018}), in which case the particles should neither be too close nor too far away from each other. We emphasize that since in our framework the marginal processes are held fixed, our results are not directly applicable to the algorithm presented in \cite{LMW2018}.  However, we expect that the results might be generalised to this context.
	
	Let us also remark that the assumption $c \in C_b(\bar{E})$ is mainly for technical convenience and both the continuity and the boundedness assumption can be weakened. Since we are interested is situations where the process $(\bar{X})_{t \ge 0}$ takes values in a compact set with high probability (i.e. the target measures $\pi_i$ are concentrated in a compact set), boundedness of $c$ is not a severe restriction.  
\end{remark}
\begin{remark}
	\label{rem:OT lower bound}
	Clearly, it holds that
	\begin{equation}
	\inf_{\bar{\pi} \in \mathcal{C}^0} \int_{\bar{E}} c \,\mathrm{d}\bar{\pi} \ge \inf_{\bar{\pi} \in \mathcal{C}} \int_{\bar{E}} c \,\mathrm{d}\bar{\pi},
	\end{equation} 
	so the solutions to the usual optimal transport problems provide lower bounds for Problem \ref{prob:OT C0}.
	
	In the Kantorovich formulation,  the cost function $c$ is often induced by a distance (for instance $c(x,y) = (d(x,y))^p$ for $n=2$, $1 \le p < \infty$), penalising couplings that put probability mass on pairs of points $(x,y)$ where $x$ and $y$ are far apart from each other (hence the name optimal transport). In the setting of MCMC (in particular in the context of variance reduction), it is plausible to encourage the particles to stay away from each other, leading to sample diversity and improved exploration of the state space. In this respect, our setting bears certain similarities with the use of optimal transport problems in functional density theory, see \cite{CFK2013}. 
\end{remark}
The set $\mathcal{C}^0$ depends on the generators $\mathcal{L}_i$. Furthermore, $\mathcal{C}^0$ is a strict subset of $\mathcal{C}$. The support of an ergodic invariant measure for a Markov process with continuous paths is necessarily connected, for instance, while in general the support of a coupling is not. The following example illustrates that $\mathcal{C}^0$ is indeed usually significantly smaller than $\mathcal{C}$:
\begin{example}[The set $\mathcal{C}^0$ contains only few singular measures]
	\label{ex:singular invariant measures}
	Consider the setting from Section \ref{sec:overdamped 1d}, i.e. two particles moving in one dimension according to overdamped Langevin dynamics. Let us fix a coupling $\Gamma \in \mathcal{G}^0$ and assume that the invariant measure $\bar{\pi}_\Gamma$ is supported on the zero set of a smooth function $H:\bar{E} \rightarrow \mathbb{R}$ with nowhere vanishing gradient, i.e. 
	\begin{equation}
	\label{eq:invariant measure zero set}
	\supp \bar{\pi}_\Gamma \subseteq \{(x,y) \in \bar{E}: \quad H(x,y) = 0\}.
	\end{equation}
	This implies that $\bar{\pi}_\Gamma$ is supported on a submanifold of $\bar{E}$ and is hence necessarily singular with respect to the Lebesgue measure. Frequently, optimisers of standard optimal transport problems are of this type (see for instance \cite[Theorem 1.2]{MPW2012}). It{\^o}'s formula implies that
	\begin{equation}
	H(\bar{X}_t) = H(\bar{X}_0) + \int_0^t (G^T \nabla H)(\bar{X}_s)\cdot \mathrm{d}W_s + \int_0^t (\bar{\mathcal{L}}_\Gamma H) (\bar{X}_s) \, \mathrm{d}s, \quad t \ge 0,
	\end{equation}
	where $G$ is given in \eqref{eq:1d coupling matrix}. Choosing the initial condition $\bar{X}_0 \sim \bar{\pi}_\Gamma$ results in $H(\bar{X}_t) = H(\bar{X}_0) = 0$ almost surely, for all $t \ge 0$. It then follows that both of the remaining integral terms individually have to be zero (owing to the decomposition into martingale and bounded variation part). The quadratic variation of the martingale part is given by
	\begin{equation}
	\int_0^t w(\bar{X}_s) \, \mathrm{d}s,\quad w = (\partial_x H)^2 + 4 \cos \beta \sin \beta \cdot (\partial_x H) (\partial_y H) + (\partial_y H)^2. 
	\end{equation} 
	Since the quadratic variation has be to be zero for all $t \ge 0$, it follows that $\partial_x H = \pm \partial_y H$ on $\supp \bar{\pi}_\Gamma$. Since $H$ is smooth with nonvanishing gradient, it turns out that $\supp \bar{\pi}_\Gamma$ is  contained in either one of the diagonals $x=y$ or $x=-y$, in fact either $\bar{\pi}_\Gamma = \frac{1}{Z}e^{-V(x)} \delta_{x-y}(\mathrm{d}x\mathrm{d}y)$ or $\bar{\pi}_\Gamma = \frac{1}{Z}e^{-V(x)} \delta_{x+y}(\mathrm{d}x\mathrm{d}y)$, noting that the latter is only possible if $V$ has the symmetry property $V(x) = V(-x)$. We conclude that, at least in the example considered here, $\mathcal{C}^0$ contains only very few singular measures. 
\end{example}  
From the theory of optimal transportation it is known that solutions of the Kantorovich problem are typically quite singular, in the sense that they are supported on small sets (see for instance \cite[Theorem 1.2]{MPW2012}).
As Example \ref{ex:singular invariant measures} shows, these measures often do not belong to $\mathcal{C}^0$. The aim of this section is to show a similar singularity property for Problem \ref{prob:OT C0}. Informally speaking, we will see that under reasonable conditions, the optimisers of \eqref{eq:OT objective} are not attained for coupling operators in the interior of $\mathcal{G}^0$.
To make this statement precise,  let us fix the decompositions
\begin{equation}
\label{eq:generator decomposition2}
\mathcal{L}_i = -\sum_{k=1}^{K_i} (A_k^i)^*A_k^i + B^i
\end{equation}
of the underlying generators (see Section \ref{sec:general construction}) and lay the focus on coupling operators of the form \eqref{eq:Gamma construction}, denoting this set by $\mathcal{G}(A)$:
\begin{equation}
\label{def:G(A)}
\mathcal{G}(A) = \left\{\Gamma = \sum_{(i,j,k,l) \in \mathcal{J}} \alpha_{ijkl} A_{k}^i A_{l}^j: \quad \alpha_{ijkl}: \bar{E} \rightarrow \mathbb{R} \right\} \cap \mathcal{G},
\end{equation}
where we recall the set $\mathcal{J}$ of admissible indices, defined in \eqref{eq:admissible indices}.
We wish to stress however that the distinction between $\mathcal{G}(A)$ and $\mathcal{G}$ is often obsolete (see Proposition \ref{prop:general construction}). The subset of ergodic coupling operators will similarly be denoted by $\mathcal{G}^0(A)$. Let us now introduce the `tangent space' to $\mathcal{G}(A)$:
\begin{equation}
\label{eq:def tangent space}
T \mathcal{G}(A) = \left\{\sum_{(i,j,k,l) \in \mathcal{J}}\alpha_{ijkl} A_{k}^i A_{l}^j\, \vert \,\alpha_{ijkl} \in C_c^{\infty} (\bar{E}) \right\}.
\end{equation} 
\begin{remark}
	The definition of $T\mathcal{G}(A)$ encapsulates the first condition of Definition \ref{def:coupling operators} in the sense that elements of $T\mathcal{G}(A)$ vanish on functions that only depend on one variable, whereas the second condition is not accounted for.
\end{remark}
In order to state our main result, we need the following definition:
\begin{definition}[Interior points]
	\label{def:interior point}
	An operator $\Gamma \in \mathcal{G}^0(A)$ is called an \emph{interior point}, if
	\begin{enumerate}
		\item 
		\label{it:curves}
		for all $\mathrm{d}\Gamma \in T\mathcal{G}(A)$ there exists $C>0$ such that $\bar{\mathcal{L}}_\Gamma + \varepsilon \mathrm{d}\Gamma \in \mathcal{G}^0(A)$ for all ${\varepsilon \in (-C,C)}$,
		\item
		\label{it:invertibility} 
		the operator $\bar{\mathcal{L}}_\Gamma$ is invertible on $L_0^2(\bar{\pi}_\Gamma)$,
		\item 
		\label{it:continuous composition}
		$\mathrm{d}\Gamma \mathcal{L}^{-1}_\Gamma f \in C(\bar{E})$ for all $\mathrm{d}\Gamma \in T\mathcal{G}(A)$ and all $f \in C(\bar{E}) \cap L_0^2(\bar{\pi}_\Gamma)$.  
	\end{enumerate}
\end{definition}
\begin{remark}
	The first condition is the essence of the foregoing definition, describing the geometric intuition of interior points. The third condition is mostly technical, since in applications $\mathcal{L}_{\Gamma}^{-1}$ usually possesses sufficient smoothing properties in order for the composition $\mathrm{d}\Gamma \mathcal{L}_\Gamma^{-1}$ to preserve continuity.  
\end{remark}
\begin{remark}[Lyapunov functions]
	It is possible and often convenient to replace the second condition by the weaker requirement that invertibility holds on a suitable subspace $\mathcal{V}_\Gamma$ of $L_0^2(\bar{\pi}_\Gamma)$. Our results in this section will then continue to hold, provided that the cost function $c$ satisfies $c - \bar{\pi}_\Gamma(c) \in \mathcal{V}_\Gamma$ for all interior points $\Gamma$. As an example, assume that there exist Lyapunov functions $\mathcal{K}_i: E_i \rightarrow [1,\infty)$ for the  one-particle dynamics, i.e.
	\begin{equation}
	\mathcal{L}_i \mathcal{K}_i \le -a_i \mathcal{K}_i + b_i,
	\end{equation} 
	for suitable constants $a_i > 0$, $b_i \ge 0$. Defining $\bar{\mathcal{K}} = \sum_{i =1}^n \mathcal{K}_i$, it follows immediately from $\Gamma \bar{\mathcal{K}} = 0$ that $\bar{\mathcal{K}}$ is a Lyapunov function for $\bar{\mathcal{L}}_\Gamma$, independently of the coupling operator $\Gamma$. Under certain minorisation (irreducibility) conditions (see \cite{HairerMattingly2011}, \cite[Chapter 2.4]{LS2016}), one can show that $\mathcal{L}_\Gamma^{-1}$ is invertible on 
	\begin{equation}
	\mathcal{V}_\Gamma = \left\{ f \in L_0^2(\bar{\pi}_\Gamma): \, \left\Vert \frac{f}{\bar{\mathcal{K}}}\right\Vert_{\infty} < \infty\right\}.
	\end{equation}
	See also \cite[Theorem 3.2]{GlynnMeyn1996}.
\end{remark}
\begin{example}
	\label{ex:interior loose}
	In the setting of Section \ref{sec:overdamped 1d} (overdamped Langevin dynamics), it is straightforward to see that $\Gamma$ as defined in \eqref{eq:1d overdamped Gamma} satisfies condition \ref{it:curves} of Definition \ref{def:interior point} if and only if $-1 < \alpha(x,y) < 1$ for all $(x,y) \in \mathbb{R}^2$, i.e. if and only if the bound \eqref{eq:1d overdamped constraint} is strict. More generally, $\Gamma$ as defined in \eqref{eq:overdamped Gamma} satisfies condition \ref{it:curves} if and only if the matrix $Q$ as defined in \eqref{eq:Q definition} is (strictly) positive definite pointwise. Those conditions are clearly equivalent to the (pointwise) ellipticity of the corresponding generators $\bar{\mathcal{L}}_\Gamma$. In the case of underdamped Langevin dynamics (Section \ref{ex:underdamped}), analogous statements are valid. Similarly, couplings of zigzag processes (Section \ref{sec:zigzag}) satisfy condition \ref{it:curves} if and only if the bound \eqref{eq:zig zag constraint} is strictly satisfied. 
\end{example}
For our further discussion, we will need the following derivative formula:
\begin{proposition}
	\label{prop:derivative}
	Let $\Gamma \in \mathcal{G}^0(A)$ be an interior point, $\mathrm{d}\Gamma \in T\mathcal{G}(A)$, and consider the family of operators $\bar{\mathcal{L}}_{\Gamma} + \varepsilon\mathrm{d}\Gamma \in \mathcal{G}^0$, for $\varepsilon$ small enough. Let the associated family of invariant measures be denoted by $\bar{\pi}^{\varepsilon}_{\Gamma}$ and fix $c \in C_b(\bar{E})$. Then the function $\varepsilon \mapsto \int_{\bar{E}} c \,\mathrm{d}\bar{\pi}^{\varepsilon}_{\Gamma}$ is differentiable in $\varepsilon = 0$, and the derivative is given by
	\begin{equation}
	\label{eq:derivative}
	\frac{\mathrm{d}}{\mathrm{d}\varepsilon} \bigg\rvert_{\varepsilon = 0} \left( \int_{\bar{E}} c \, \mathrm{d}\bar{\pi}^{\varepsilon}_{\Gamma} \right) = - \int_{\bar{E}} c \left[ \bar{\mathcal{L}}_{\Gamma} ^*\right]^{-1} (\mathrm{d}\Gamma^* \mathbf{1}) \,\mathrm{d}\bar{\pi}_{\Gamma}^{0},  
	\end{equation}
	where the adjoints are taken in $L^2(\bar{\pi}_{\Gamma}^0)$.
\end{proposition}
\begin{proof}
	The proof can be found in Appendix \ref{sec:proofs_OT}.
\end{proof}
\begin{remark}
	Notice that our notation entails that $\bar{\pi}_\Gamma \equiv \bar{\pi}^0_\Gamma$. Moreover, the right-hand side of \eqref{eq:derivative} is well defined. Indeed, by the second condition in Definition  \ref{def:interior point}, $\left[ \bar{\mathcal{L}}_{\Gamma} ^*\right]^{-1}$ is well defined on $L^2_0(\bar{\pi}_{\Gamma}^0)$ and furthermore $\Ran \mathrm{d}\Gamma^* \subseteq L^2_0(\bar{\pi}_{\Gamma}^0)$ due to 
	\begin{equation}
	\int_{\bar{E}} \mathrm{d}\Gamma^* f \, \mathrm{d}\bar{\pi}_{\Gamma}^0 = \int_{\bar{E}} (\mathrm{d} \Gamma \mathbf{1}) f \, \mathrm{d}\bar{\pi}_{\Gamma}^0= 0, \quad  f \in C_c^{\infty}(\bar{E}), 
	\end{equation}
	using that $\mathrm{d}\Gamma \mathbf{1} = 0$ according to Lemma \ref{lem:decomposition kernel}.
\end{remark}
For an interior point $\Gamma \in \mathcal{G}^0(A)$ and $\mathrm{d}\Gamma \in T \mathcal{G}(A)$ let us introduce the suggestive notation
\begin{equation}
\frac{\mathrm{d}}{\mathrm{d}\Gamma} \int_{\bar{E}} c \, \mathrm{d} \bar{\pi}_\Gamma := \frac{\mathrm{d}}{\mathrm{d}\varepsilon} \bigg\rvert_{\varepsilon = 0} \left( \int_{\bar{E}} c \, \mathrm{d}\bar{\pi}^{\varepsilon}_{\Gamma} \right),
\end{equation}
as well as the following terminology: 
\begin{definition}[Critical points]
	\label{def:critical point}
	Let $\Gamma \in \mathcal{G}(A)$ be an interior point. Then $\Gamma$ is called \emph{critical} if 
	\begin{equation}
	\frac{\mathrm{d}}{\mathrm{d}\Gamma} \int_{\bar{E}} c \, \mathrm{d} \bar{\pi}_\Gamma = 0
	\end{equation}
	for all $\mathrm{d}\Gamma \in T\mathcal{G}(A)$.
\end{definition}
In our aim to find minimisers of the function $\Gamma \mapsto \int_{\bar{E}} c \, \mathrm{d}\bar{\pi}_\Gamma$, it is natural to seek critical points. The following is our main result in this section:
\begin{theorem}
	\label{thm:max boundary}
	Let $c \in C_b(\bar{E})$. Then either all interior points are critical, or no interior point is critical.
\end{theorem}
\begin{example}
	Let $c$ be of the form
	\begin{equation}
	c(x_1,\ldots,x_n) = g_1(x_1) + \ldots + g_n(x_n),
	\end{equation}
	for appropriate functions $g_i: E_i \rightarrow \mathbb{R}$. Then, since $\bar{\pi}_\Gamma$ is a coupling of the fixed marginals $(\pi_i)_{i = 1}^n$, the function $\Gamma \mapsto \int_{\bar{E}} c \,\mathrm{d}\bar{\pi}_\Gamma = \sum_{i=1}^n \int_{E_i}g_i \, \mathrm{d}\pi_i$ is constant, and hence all interior points are critical.
\end{example}
Before proceeding to the proof of the theorem, let us give a few remarks:
\begin{remark}
	\label{rem:discussion main theorem}
	Informally, Theorem \ref{thm:max boundary} states that the mapping $\mathcal{G}^0(A) \ni \Gamma \mapsto \int_{\bar{E}} c \,\mathrm{d}\bar{\pi}_{\Gamma}$ is either locally constant or does not attain its extrema on interior points. In other words, if $\Gamma \mapsto \int_{\bar{E}} c \, \mathrm{d}\bar{\pi}_\Gamma$ is not constant, then its extrema lie `at the boundary' of $\mathcal{G}^0(A)$, although we have not rigorously defined this term, and moreover, $\int_{\bar{E}} c \,\mathrm{d}\bar{\pi}_\Gamma$ is not even well-defined for nonergodic couplings. 
\end{remark}
\begin{remark}
	A striking consequence of Theorem \ref{thm:max boundary} is that  under mild conditions, independent coupling (associated to $\mathbf{0} \in \mathcal{G}^0$) of overdamped or underdamped Langevin dynamics is not optimal for \emph{any} criterion of the form $\int_{\bar{E}}c \, \mathrm{d}\bar{\pi}_\Gamma$. Theorem \ref{thm:max boundary} complements results from the theory of optimal transportation that state that optimal couplings are generically singular in terms of their support. Indeed, considering the example of overdamped or underdamped Langevin dynamics, the `boundary of $\mathcal{G}^0(A)$' consists of couplings that lead to degenerately elliptic generators that are in general not hypoelliptic. In particular, the corresponding invariant measures are not in general absolutely continuous with respect to the Lebesgue measure.  
\end{remark}
\begin{remark}
	Theorem \ref{thm:max boundary} also supports the folklore that optimal Markov chain Monte Carlo samplers use as little noise as possible to guarantee ergodicity, as degenerately elliptic operators correspond to dynamics where noise only acts in certain directions. For example, it is by now well-documented that nonreversible samplers outperform their reversible counterparts in various settings (see for instance \cite{DLP2016,Hwang2005,ottobre2016markov,RBS2016}). The process of making a reversible sampler nonreversible can be thought of informally as decreasing the ratio between random and deterministic behaviour.
\end{remark}
\begin{remark}
	\label{rem:power expansion}
	Let us examine the function $\mathcal{G}^0(A) \ni \Gamma \mapsto \int_{\bar{E}} c \,\mathrm{d} \bar{\pi}_{\Gamma}$ along a ray. More precisely, fix $\mathrm{d}\Gamma \in T \mathcal{G}(A)$, set $\bar{\mathcal{L}}_{\varepsilon}:= \bar{\mathcal{L}}_0 + \varepsilon \mathrm{d}\Gamma$ for $\varepsilon$ small enough, and consider the function $\varepsilon \mapsto \int_{\bar{E}} c \, \mathrm{d}\bar{\pi}_{\varepsilon}$, where $(\bar{\pi}_{\varepsilon})_\varepsilon$ denotes the corresponding family of invariant measures. Since $\mathrm{d}\Gamma$ is relatively bounded with respect to $\bar{\mathcal{L}}_0$ in $L^2(\bar{\pi}_0)$, we have the following Neumann power expansion for $\varepsilon$ small enough:
	\begin{equation}
	\label{eq:power expansion ray}
	\int_{\bar{E}} c \, \mathrm{d} \bar{\pi}_{\varepsilon} = \int_{\bar{E}} c \left( 1 + \sum_{j=1}^{\infty} (-\varepsilon)^j [(\mathrm{d}\Gamma \bar{\mathcal{L}}_0^{-1})^*]^j \right) \mathbf{1}\, \mathrm{d}\bar{\pi}_0. 
	\end{equation} 
	For details, see \cite[Theorem 5.2]{LS2016}. The factor of $(-\varepsilon)^j$ in expression \eqref{eq:power expansion ray} signals oscillatory behaviour, and indeed it is straightforward to construct examples (for instance in the Gaussian case), where $\eqref{eq:power expansion ray}$ exhibits multiple local minima and maxima as a function of $\varepsilon$ (see for instance the graph related to linear coupling in Figure \ref{fig:mixed_obs} below). This finding is not in contradiction with Theorem \ref{thm:max boundary}. Indeed, as Theorem \ref{thm:max boundary} shows, at those extrema there are directions of ascent (or descent) in $T\mathcal{G}(A)$ not aligned with the considered ray and thus, those extrema turn out not to be critical when considered in the whole of $\mathcal{G}^0(A)$.
\end{remark}
Let us now prove Theorem \ref{thm:max boundary} and start with the following key lemma. Its significance derives from the fact that the second statement manifestly does not depend on $\Gamma$. 
\begin{lemma}
	\label
	{lem:no_Gamma_equivalence}
	Let $c \in C_b(\bar{E}) \cap L_0^2(\bar{\pi}_0)$ and $\Gamma \in \mathcal{G}^0(A)$ be an interior point.
	Then the following conditions are equivalent:
	\begin{enumerate}
		\item 
		\label{it:include_gamma}
		The following holds for all admissible indices $(i,j,k,l) \in \mathcal{J}$:
		\begin{equation}
		\label{eq:Gamma_pde}
		A_k^i A_l^j \bar{\mathcal{L}}_\Gamma^{-1}\left(c - \bar{\pi}_{\Gamma}(c)\right) = 0.
		\end{equation}
		\item
		\label{it:no_Gamma}
		There exists $f \in \mathcal{D}(\bar{\mathcal{L}}_0)$ such that both of the following hold:
		\begin{enumerate}
			\item
			\label{it:As_zero}
			for all admissible indices $(i,j,k,l) \in \mathcal{J}$ it holds that
			\begin{equation}
			\label{eq:As_zero}
			A_k^i A_l^j f = 0,
			\end{equation}
			\item
			\label{it:cost L0}
			\begin{equation}
			\label{eq:Poisson c}
			\bar{\mathcal{L}}_0 f = c.
			\end{equation}
		\end{enumerate} 
	\end{enumerate}
\end{lemma}
\begin{proof} First assume that \eqref{eq:Gamma_pde} holds for all $(i,j,k,l) \in \mathcal{J}$. Then setting $$f = \bar{\mathcal{L}}_{\Gamma}^{-1} \left(c - \bar{\pi}_{\Gamma}(c)\right)$$ immediately implies \eqref{eq:As_zero}. Furthermore, from \eqref{eq:As_zero} and \eqref{def:G(A)} it follows that $\bar{\mathcal{L}}_\Gamma f = \bar{\mathcal{L}}_0 f$,
	implying \begin{equation}
	\label{eq:L0 Pi_G}
	\bar{\mathcal{L}}_0 f = \left(c - \bar{\pi}_{\Gamma}(c)\right),
	\end{equation}
	as well as $f \in \mathcal{D}(\bar{\mathcal{L}}_0)$. Equation \eqref{eq:L0 Pi_G} clearly implies that $\bar{\pi}_0(c) = \bar{\pi}_\Gamma(c)$, and hence $\bar{\pi}_\Gamma(c) = 0$, leading to \eqref{eq:Poisson c}. 
	
	The reverse implication follows similarly by first observing that \eqref{eq:As_zero} and \eqref{eq:Poisson c} imply that $\bar{\mathcal{L}}_\Gamma f = c$, and hence $\bar{\pi}_\Gamma(c)=0$. Combining this with \eqref{eq:As_zero} shows that \eqref{eq:Gamma_pde} holds.
\end{proof}

\begin{proof} [Proof of Theorem \ref{thm:max boundary}.]
	Clearly, we can without loss of generality assume that $\bar{\pi}_0 (c) = 0$.
	According to Definition \ref{def:critical point} and Proposition \ref{prop:derivative}, an interior point $\Gamma \in \mathcal{G}^0(A)$ is critical if and only if 
	\begin{equation}
	\label{eq:der_zero}
	\int_{\bar{E}} c \left[ \bar{\mathcal{L}}_{\Gamma} ^*\right]^{-1} (\mathrm{d}\Gamma^* 1) \,\mathrm{d}\bar{\pi}_{\Gamma} = 0,
	\end{equation}
	for all $\mathrm{d}\Gamma \in T\mathcal{G}(A)$, which is equivalent to 
	\begin{equation}
	\label{eq:der_paramaters}
	\sum_{(i,j,k,l) \in \mathcal{J}} \int_{\bar{E}} \left( A_k^i A_l^j \bar{\mathcal{L}}_\Gamma^{-1}\left(c - \bar{\pi}_\Gamma(c) \right)  \right) \cdot \alpha_{ijkl} \, \mathrm{d}\bar{\pi}_\Gamma = 0,
	\end{equation}
	for all $(\alpha_{ijkl})_{(i,j,k,l) \in \mathcal{J}} \subset C_c^\infty(\bar{E})$. The latter statement is clearly equivalent to the first statement in Lemma \ref{lem:no_Gamma_equivalence}. 
	The result now follows by noting that the second statement in Lemma \ref{lem:no_Gamma_equivalence} does not depend on $\Gamma$.
\end{proof}\section{A perturbative approach for the study of the asymptotic variance}
\label{ch:perturbative approach}

Informally speaking, Theorem \ref{thm:max boundary} shows that the objective of optimising the asymptotic variance $\sigma_F^2$ leads to the requirement that $\Gamma \in \mathcal{G}^0(A)$ should be chosen to be a `boundary point' (see the discussion in Remark \ref{rem:discussion main theorem}). While being an interesting theoretical result, it does not give much guidance about how to choose a suitable coupling in practice (after all, both minima and maxima are obtained `at the boundary' of $\mathcal{G}^0(A)$). In this section we therefore develop a perturbative approach, based on operators of the form
\begin{equation}
\bar{\mathcal{L}}^{\varepsilon}_{\Gamma} = \bar{\mathcal{L}}_0 + \varepsilon \mathrm{d}\Gamma, \quad \mathrm{d}\Gamma \in T \mathcal{G}(A),
\end{equation}
for $\varepsilon$ small enough.
In the following, we will assume that $\varepsilon \in \mathcal{I}$, where $\mathcal{I} \subset \mathbb{R}$ is an appropriate interval such that $\bar{\mathcal{L}}^{\varepsilon}_{\Gamma} \in \mathcal{G}^0$ for all $\varepsilon \in \mathcal{I}$. As usual, we consider observables of the form $F = \frac{1}{n}\sum_{i=1}^n f_i$, for some $f_i \in L_0^2(\pi_i)$, and suppose that Assumption \ref{ass:one-particle decay} is satisfied. To stress the dependence of the asymptotic variance on the parameter $\varepsilon$ we will write $\sigma_F^2(\varepsilon)$. Note that a similar setting has already been considered in Remark \ref{rem:power expansion}. There, we investigated the dependence of the asymptotic variance (or more generally, of the quantity $\int_{\bar{E}} c \, \mathrm{d}\bar{\pi}_{\varepsilon}$) on the parameter $\varepsilon$. Here, we are rather interested in the choice of the `direction' $\mathrm{d}\Gamma \in T \mathcal{G}(A)$, starting from the trivial (independent) coupling $\bar{\mathcal{L}}_0$. 

Combining the expression \eqref{eq:asym_var} with either \eqref{eq:power expansion ray} or \eqref{eq:derivative} we see that
\begin{equation}
\label{eq:asym var local}
\frac{\mathrm{d}}{\mathrm{d}\varepsilon} \sigma_F^2 \big\vert_{\varepsilon = 0} = \int_{\bar{E}} \mathrm{d}\Gamma \xi \, \mathrm{d}\bar{\pi}_0,
\end{equation}
where $\xi$ is given by
\begin{equation}
\xi = \frac{1}{n^2}\sum_{i < j} \phi_i \phi_j,
\end{equation}
in terms of the solutions to the Poisson equations \eqref{eq:invertibility}.
The benefit of \eqref{eq:asym var local} is that its right-hand side consists of expressions that are known in principle, as the measure $\bar{\pi}_0$ is given by the product $\bar{\pi}_0 = \bigotimes_{i=1}^n \pi_i$. It therefore serves as a starting point for finding a suitable coupling operator $\Gamma \in \mathcal{G}$. Let us summarise our approach in this section in the following form:
\begin{problem}
	\label{prob:linearised OT}
	Given invariant measures $\pi_i \in \mathcal{P}(E_i)$ and observables $f_i \in L_0^2(\pi_i)$, find a coupling operator $\Gamma \in \mathcal{G}$ such that
	\begin{equation}
	\label{eq:linearised objective}
	\int_{\bar{E}} \Gamma \xi \, \mathrm{d}\bar{\pi}_0
	\end{equation}
	is minimised.
\end{problem}
Problem \ref{prob:linearised OT} can be thought of in two different ways: Firstly, it can be interpreted as a linearisation of Problem \ref{prob:OT C0}. Indeed, \eqref{eq:linearised objective} depends linearly on $\Gamma$, whereas \eqref{eq:OT objective} is highly nonlinear (for an illustration of this fact, see the power expansion \eqref{eq:power expansion ray}). Another way of seeing this is by noting the similarity between \eqref{eq:linearised objective} and the second term appearing on the right-hand side of \eqref{eq:asym var Gamma}. Not surprisingly, Problem \ref{prob:linearised OT} turns out to be much  easier to (approximately) solve in practice. Note that by linearity, properties similar to the one expressed in Theorem \ref{thm:max boundary} hold for Problem \ref{prob:linearised OT} (at least if $\mathcal{G}$ is convex). Choosing a coupling $\Gamma \in \mathcal{G}$ according to the formulation of Problem \ref{prob:linearised OT} is clearly heuristic. However, we have had good results with it in numerical experiments (see below).  

Secondly, when a solution of Problem \ref{prob:linearised OT} is available, it is reasonable in practice to only implement a small perturbation of the independent sampler (i.e. choose $\varepsilon$ to be small). Such a choice will not be optimal over all couplings in $\mathcal{G}$ according to Theorem \ref{thm:max boundary}. However, it is then guaranteed that the performance of the sampler is at least slightly improved. Let us also note that the formulation of Problem \ref{prob:linearised OT} does not require the coupling to be ergodic, as opposed to the formulation of Problem \ref{prob:OT C0}.  

The aim of this section is to analyse Problem \ref{prob:linearised OT} for some of the examples presented in Chapter \ref{ch:coupling_examples} and to present some numerical experiments.
To this end, let us introduce the shorthand notation
\begin{equation}
\label{eq:local asym var}
\delta\sigma_F^2 (\Gamma) := \int_{\bar{E}}\Gamma \xi \,\mathrm{d}\bar{\pi}_0,
\end{equation} 
stressing the infinitesimal (approximate) nature of the objective in Problem \ref{prob:linearised OT}.
In the sequel, $\Gamma$ will be given in terms of a function $\alpha$, belonging to a set $\mathcal{A}$. To emphasize this dependence we will write $\Gamma_{\alpha}$. We will not impose regularity constraints on the function $\alpha$ (beyond measurability), so that the operators $\Gamma_\alpha$ will in general not induce couplings that satisfy the Feller property (see Remark \ref{rem:regularity alpha}) and hence strictly speaking do not belong to $\mathcal{G}$.

\subsection{Overdamped Langevin dynamics in one dimension with two particles}
\label{sec:numerics 1d overdamped}
Consider the setting from the example presented in Section \ref{sec:overdamped 1d}. Then, \eqref{eq:local asym var} takes the form
\begin{equation}
\label{eq:1d overdamped asym var}
\delta\sigma_F^2 (\Gamma_{\alpha}) =\int_{\mathbb{R}^2} \alpha(x,y) \phi'(x)\phi'(y) e^{-\left(V(x)+V(y)\right)}\,\mathrm{d}x\mathrm{d}y,
\end{equation}
where $\phi$ is the solution to the Poisson equation
\begin{equation}
\label{eq:1d Poisson}
-\left(-V'\phi' + \phi'' \right) = f, \quad \pi(\phi) = 0, 
\end{equation}
and $f \in L_0^2(\pi)$ is an observable of interest. Furthermore, $\Gamma$ is given as in \eqref{eq:1d overdamped Gamma}, with 
\begin{equation}
\alpha \in  \mathcal{A} := \{\alpha: \mathbb{R}^2 \rightarrow \mathbb{R} \, \vert \, \alpha \, \text{measurable, } -1 \le \alpha \le 1\}.
\end{equation}
Recall from Section \ref{sec:overdamped 1d} that $\alpha \in \mathcal{A}$ induces a well-defined coupled process (Lemma \ref{lem:1d BMs}) that however does not satisfy the Feller property in general (further regularity assumptions would be required).
The following optimality result is immediate from an inspection of \eqref{eq:1d overdamped asym var}:
\begin{proposition}
	\label{prop:1d overdamped opt asym var}
	Let $\alpha^* \in \mathcal{A}$ be given by
	\begin{equation}
	\label{eq:1d overdamped optimal alpha}
	\alpha^*(x,y) = 
	\begin{cases}
	1 \quad & \mbox{if } \;\phi'(x) \phi'(y) \le 0, \\
	-1 \quad & \mbox{if } \;\phi'(x) \phi'(y) > 0, \\
	\end{cases}
	\end{equation}
	Then $\Gamma_{\alpha^*}$ solves Problem \ref{prob:linearised OT} in the sense that $\delta\sigma_F^2 (\Gamma_{\alpha^*}) \le \delta\sigma_F^2 (\Gamma_{\alpha})$ for all $\alpha \in \mathcal{A}$.
\end{proposition}
\begin{remark}
	\label{rem:optimal coupling improves}
	Clearly, we have
	\begin{equation}
	\int_{\bar{E}} \Gamma_{\alpha^*} \xi \, \mathrm{d}\bar{\pi}_\Gamma \le 0,
	\end{equation}
	for $\alpha^*$ as defined in \eqref{eq:1d overdamped optimal alpha}. We stress the difference between this expression and \eqref{eq:1d overdamped asym var}, where we compute the same integral, but with respect to $\bar{\pi}_0$. By comparison with \eqref{eq:asym var Gamma}, we see that $\sigma_F^2(\Gamma_{\alpha^*}) \le \sigma_F^2(\mathbf{0})$, i.e. $\Gamma_{\alpha^*}$ always improves on independent coupling. However, we do not know whether $\Gamma_{\alpha^*}$ is optimal in the sense of Problem \ref{prob:OT C0}. 
\end{remark}
\begin{figure}
	\thisfloatpagestyle{empty}
	\captionsetup[subfigure]{justification=centering}
	\caption{Comparison of invariant measures associated to solutions of Problem \ref{prob:linearised OT} and optimal transport maps as solutions to the Kantorovich problem for the example of overdamped Langevin dynamics in dimension one with quadratic potential.}
	\label{fig:overdamped coupling comparison}
	\begin{subfigure}[b]{0.5 \textwidth}
		\includegraphics[width=\textwidth]{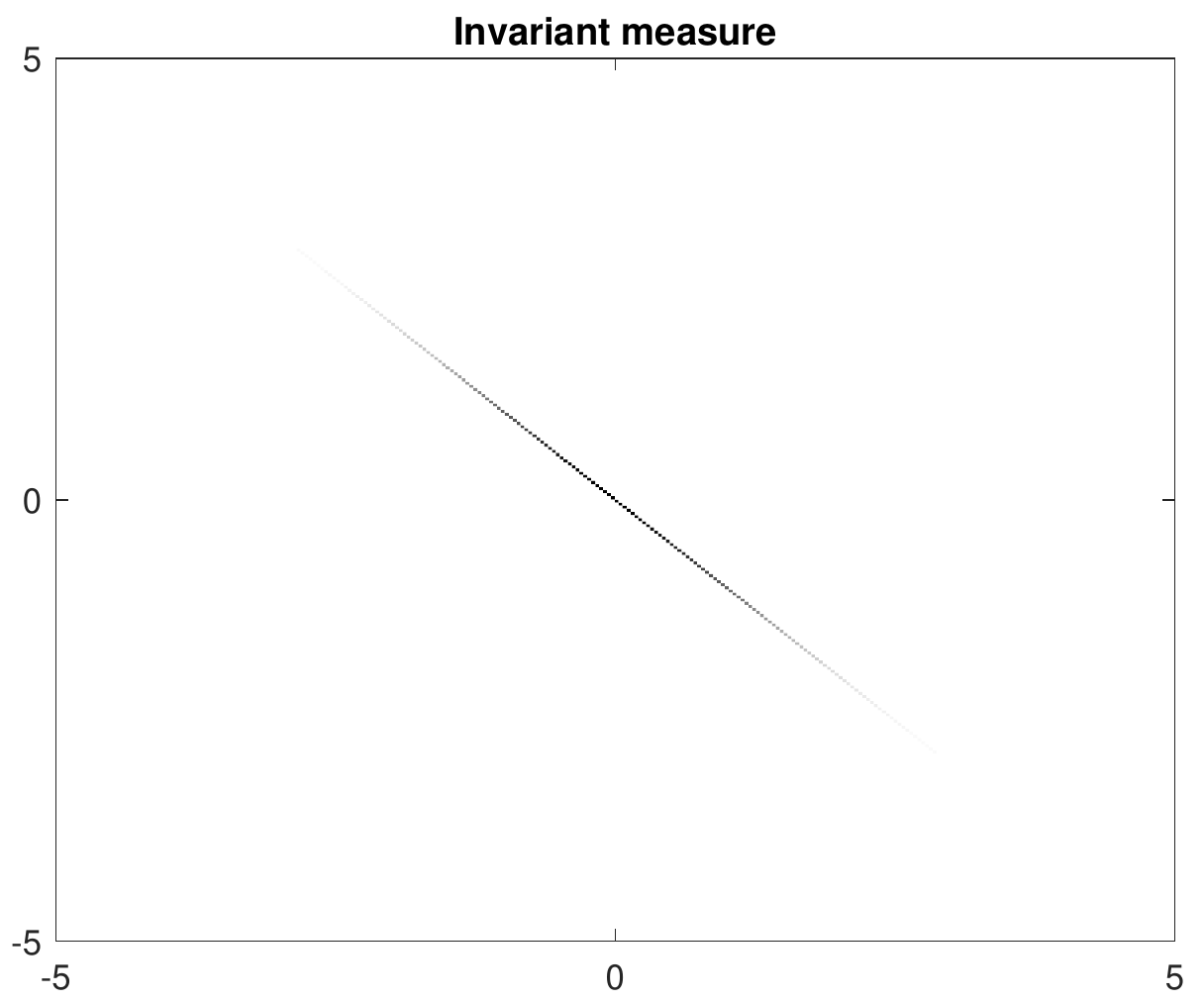}
	\end{subfigure}
	\begin{subfigure}[b]{0.5 \textwidth}
		\includegraphics[width=\textwidth]{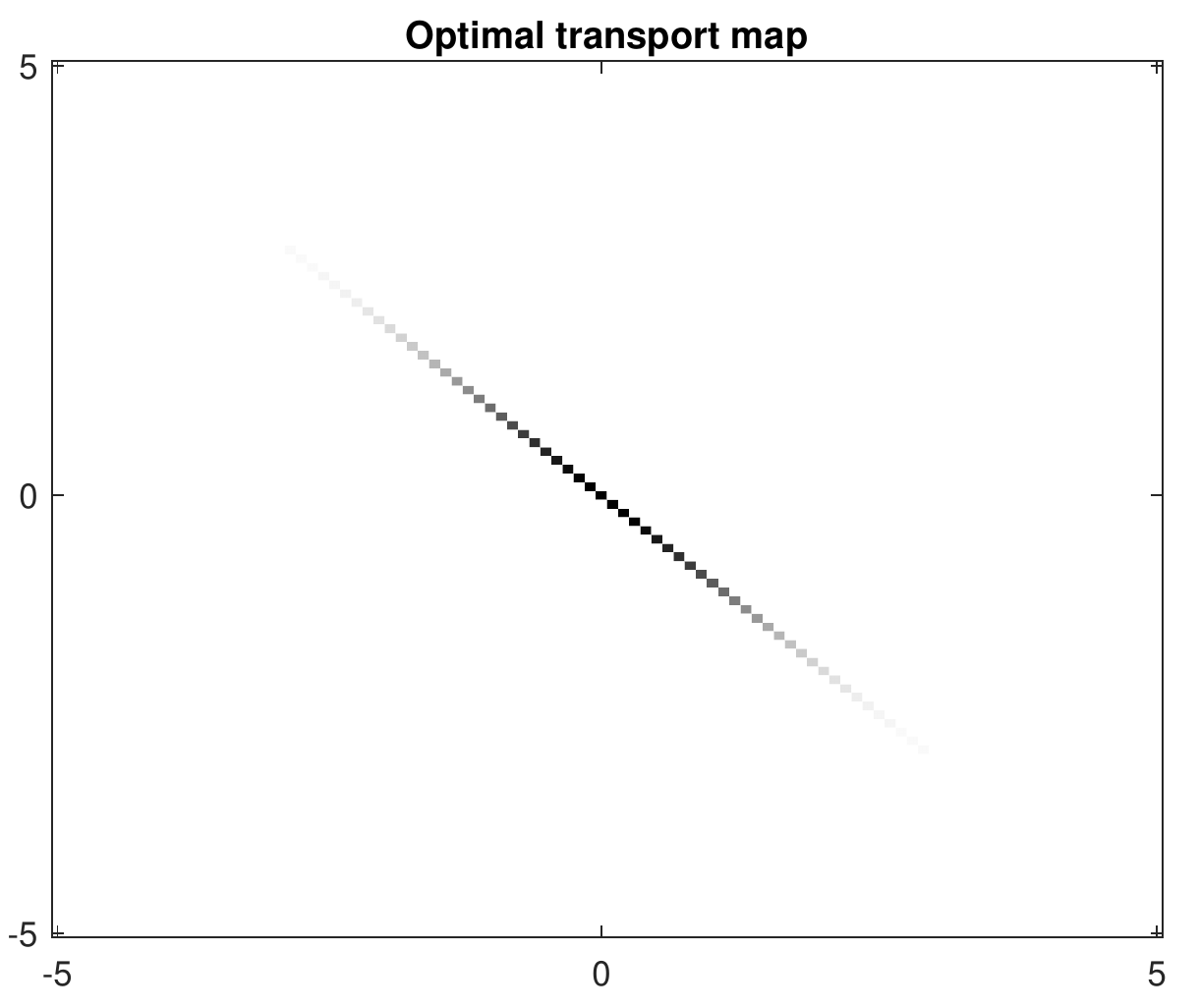}
	\end{subfigure}
	\subcaption{Linear observable: $f_1(x) = x$.}
	\label{fig:lin OT map}
	\begin{subfigure}[b]{0.5 \textwidth}
		\includegraphics[width=\textwidth]{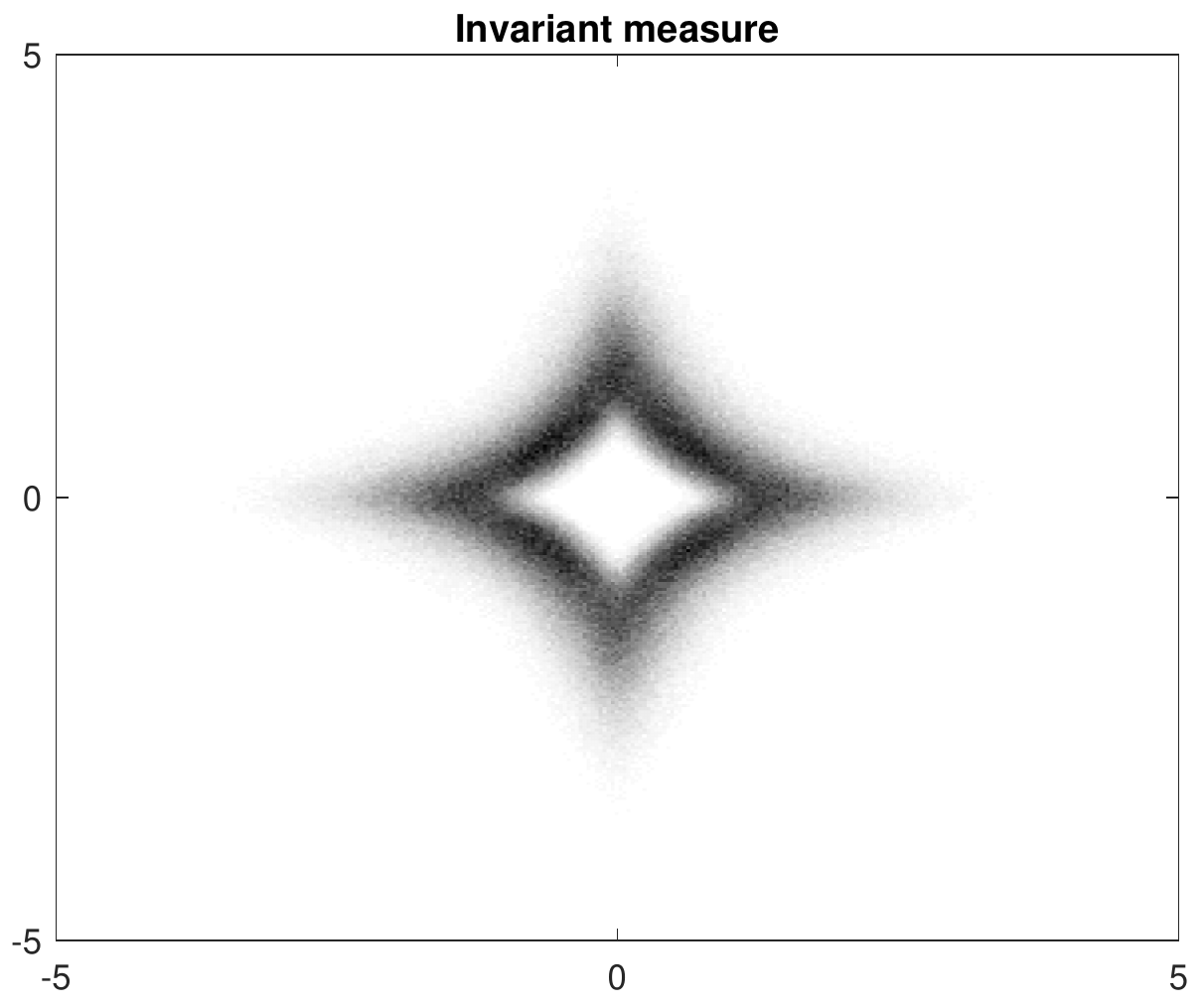}
	\end{subfigure}
	\begin{subfigure}[b]{0.5 \textwidth}
		\includegraphics[width=\textwidth]{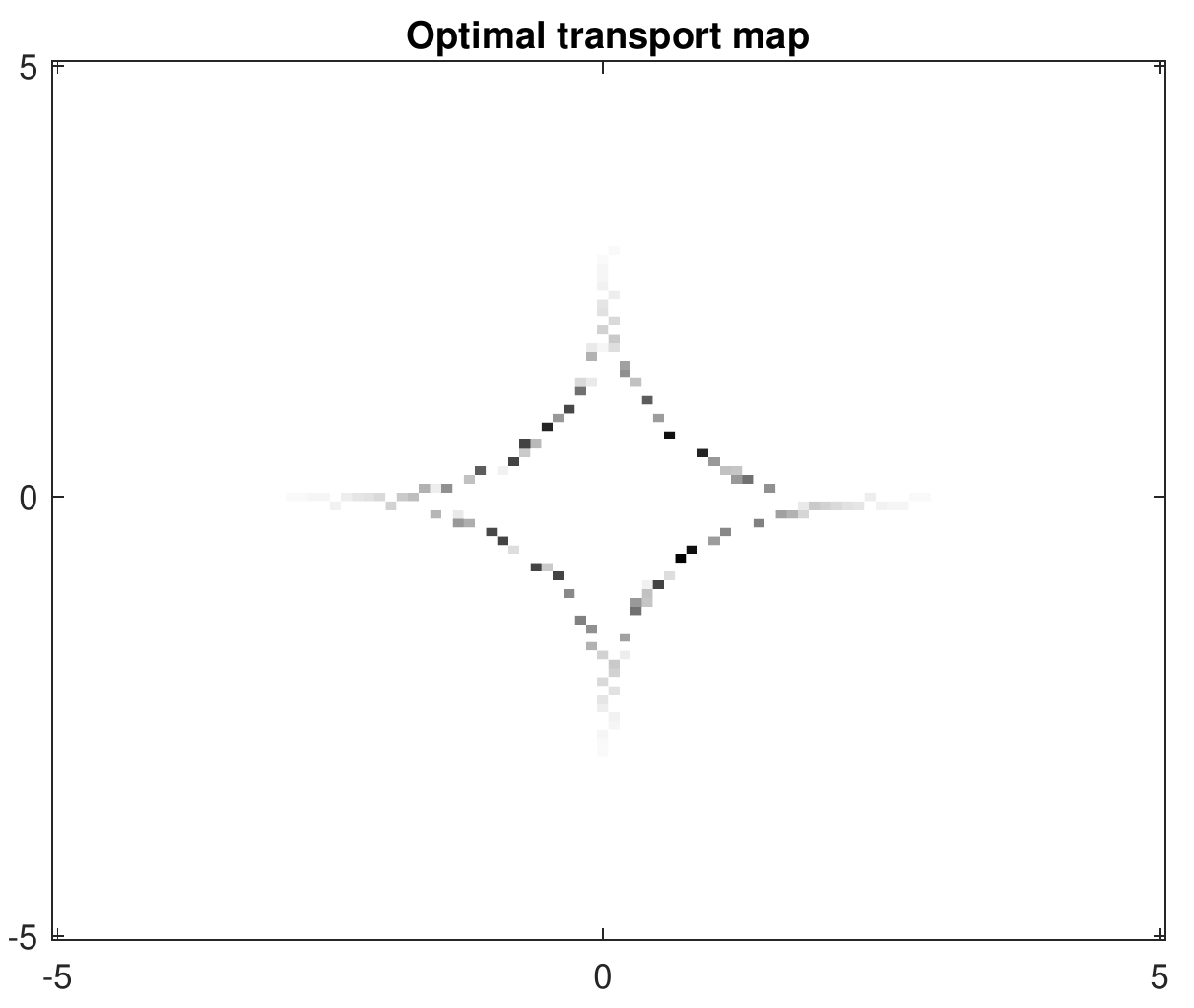}
	\end{subfigure}
	\subcaption{Quadratic observable: $f_2(x)= x^2$.}
	\begin{subfigure}[b]{0.5 \textwidth}
		\includegraphics[width=\textwidth]{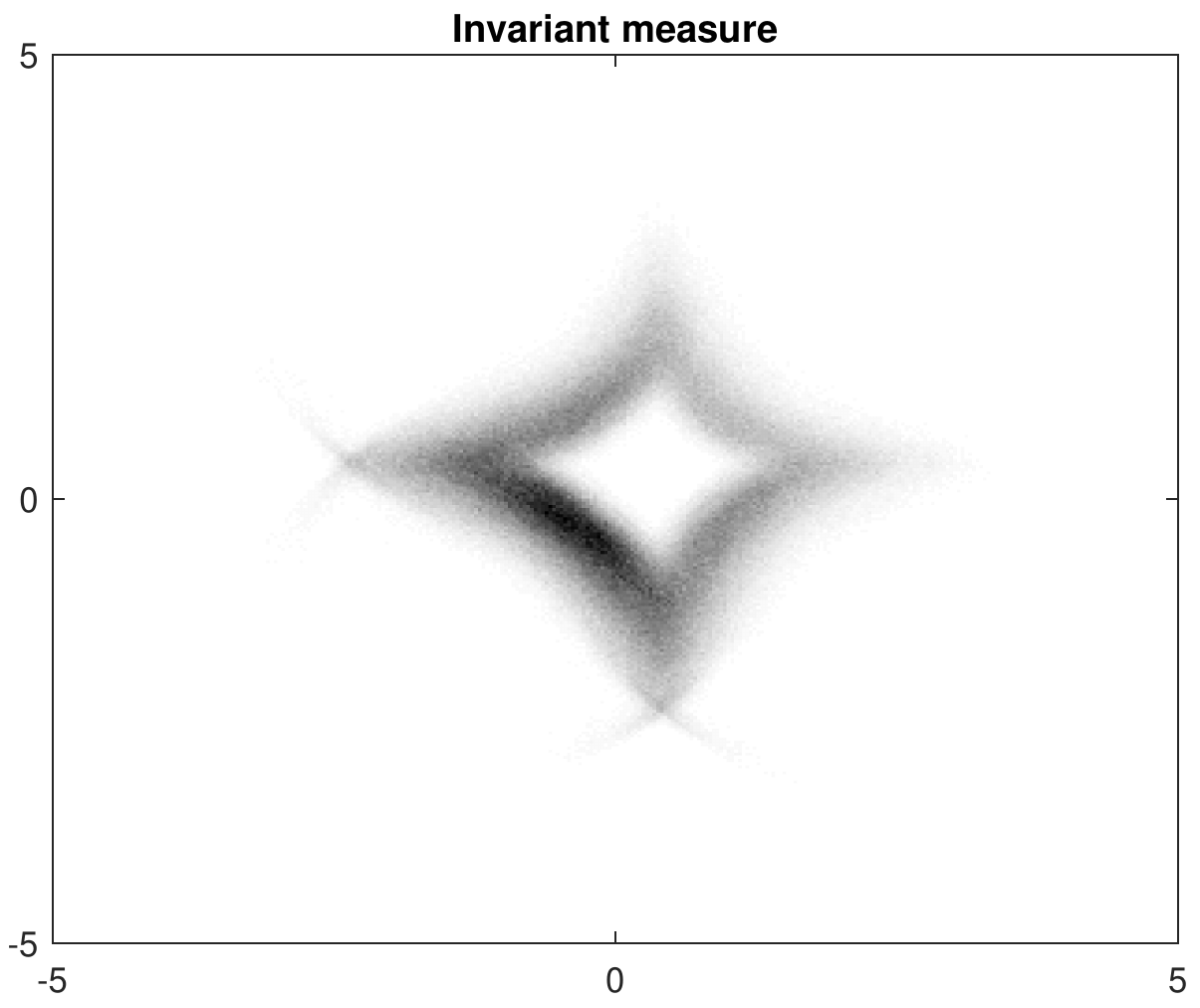}
	\end{subfigure}
	\begin{subfigure}[b]{0.5 \textwidth}
		\includegraphics[width=\textwidth]{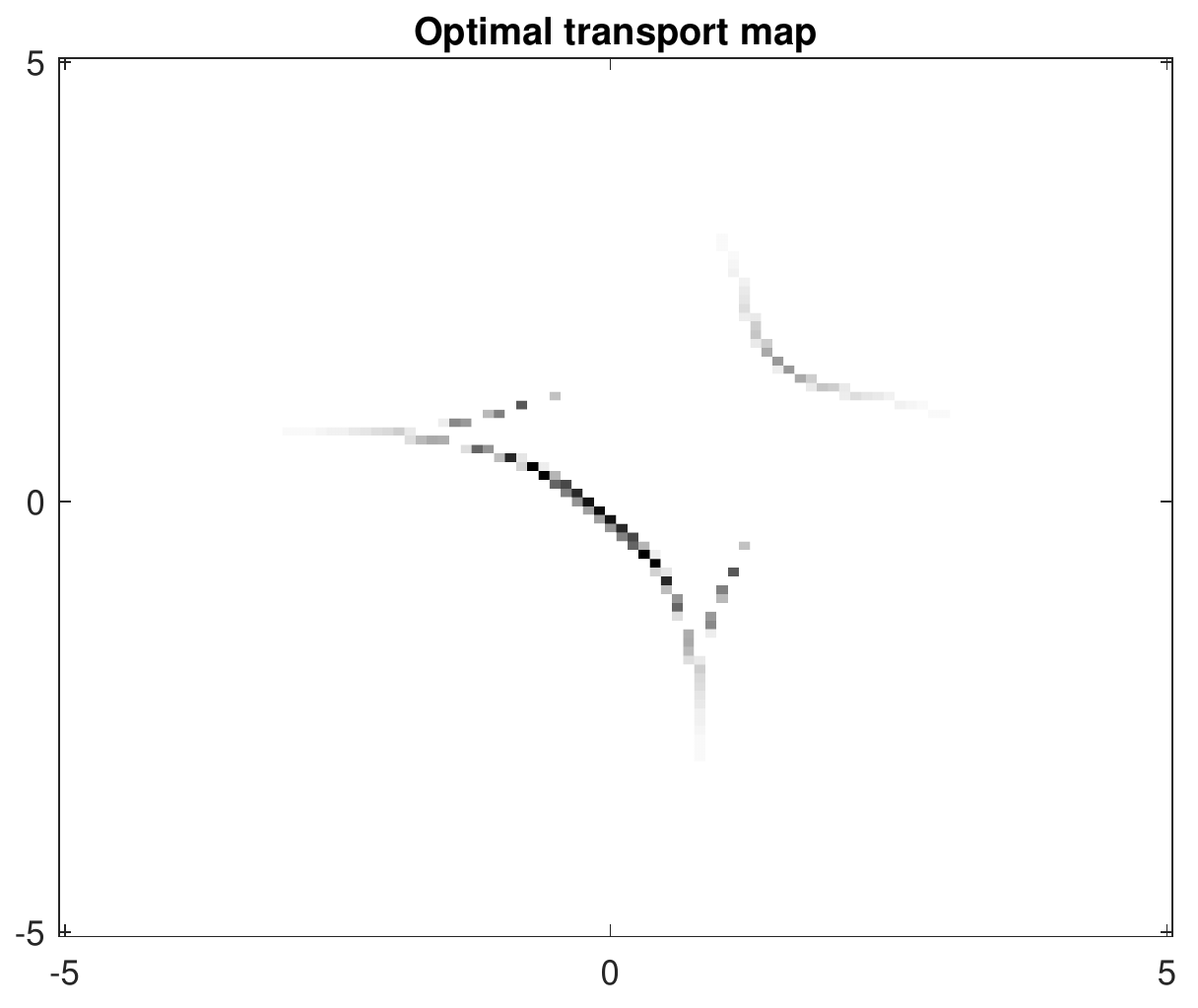}
	\end{subfigure}
	\subcaption{Mixed observable: $f_3(x) = x^2 -x$.}
\end{figure}
It is instructive to compare the solution of Problem \ref{prob:linearised OT} found in Proposition \ref{prop:1d overdamped opt asym var} to the solution of the usual Kantorovich problem. Recall that Problem \ref{prob:linearised OT} can be considered to be a linearisation of Problem \ref{prob:OT C0}, which in turn is related to the Kantorovich problem in the sense that the minimisation is carried out over a smaller set of couplings (namely those couplings that are invariant measures of coupled processes, see Definition \ref{def:admissible couplings}). For our experiments, we choose the quadratic potential $V(x) = \frac{1}{2} x^2$, i.e. the task of sampling from a Gaussian measure. Furthermore, we consider the linear observable $f_1(x) = x$, the quadratic observable $f_2(x) = x^2$ and the `mixed' observable $f_3(x) = x^2 - x$. In Figure \ref{fig:overdamped coupling comparison}, we plot the invariant measure of the coupled processes induced by \eqref{eq:1d overdamped optimal alpha} (left-hand side) and compare them to the solutions of the Kantorovich problem\footnote{The optimal transport map was computed using the Python library POT 0.4.0 (accessible from \href{https://pypi.python.org/pypi/POT/0.4.0}{https://pypi.python.org/pypi/POT/0.4.0}) which is based on the algorithm proposed in \cite{bonneel2011displacement}. } (right-hand side), with the appropriate cost function $c= -\bar{\mathcal{L}}_0 \xi$ as given in Section \ref{sec:OT}. As it turns out, the solutions to Problem \ref{prob:linearised OT} and the standard Kantorovich problem look remarkably similar (at least in shape). We hence conclude that in the example considered here, Problem \ref{prob:linearised OT} is a good approximation of Problem \ref{prob:OT C0}, keeping in mind that the solution of the Kantorovich problem provides a lower bound for the objective function of Problem \ref{prob:OT C0} (see Remark \ref{rem:OT lower bound}). 

The following lemma serves to examine a few test cases and gain further intuition. For convenience, let us assume that $f$ (and therefore, by elliptic regularity $\phi$) are smooth.
\begin{lemma}
	\label{lem:1d Poisson}
	Let $\phi \in L_0^2(\pi)$ solve the Poisson equation \eqref{eq:1d Poisson}.
	\begin{enumerate}
		\item 
		\label{it:monotone}
		Assume that $f \in L_0^2(\pi)$ is monotonically increasing (decreasing). Then $\phi'$ is nonnegative (nonpositive).
		\item
		\label{it:symmetry}
		Assume that $V$ and $f$ are symmetric, i.e. $V(-x)=V(x)$ and $f(-x)=f(x)$ for all $x \in \mathbb{R}$. Furthermore, let $f$ be monotonically decreasing (increasing) on $(-\infty, 0]$. Then $\phi'(x) \cdot x \le 0$ ($\phi'(x)\cdot x \ge 0$) for all $x \in \mathbb{R}$.
	\end{enumerate}
\end{lemma}

The proof can be found in Appendix \ref{sec:proofs_pert}. The following two corollaries are direct consequences of Lemma \ref{lem:1d Poisson} and Proposition \ref{prop:1d overdamped opt asym var}:
\begin{corollary}[`Mirror coupling']
	\label{cor:monotone}
	In the setting from the first part of Lemma \ref{lem:1d Poisson}, $$\alpha^* \equiv -1$$ solves Problem \ref{prob:linearised OT}, in the sense that $\delta\sigma_F^2 (\Gamma_{\alpha^*}) \le \delta\sigma_F^2 (\Gamma_{\alpha})$ for all $\alpha \in \mathcal{A}$.
\end{corollary}
\begin{corollary}[`Symmetric coupling']
	\label{cor:symmetry}
	In the setting from the second part of Lemma \ref{lem:1d Poisson},
	\begin{equation}
	\label{eq:1d overdamped optimal alpha symmetric}
	\alpha^*(x,y) = 
	\begin{cases}
	1 \quad & \mbox{if } x\cdot y \le 0, \\
	-1 \quad & \mbox{if } x \cdot y > 0, \\
	\end{cases}
	\end{equation}
	solves Problem \ref{prob:linearised OT}, in the sense that $\delta\sigma_F^2 (\Gamma_{\alpha^*}) \le \delta\sigma_F^2 (\Gamma_{\alpha})$ for all $\alpha \in \mathcal{A}$.
\end{corollary}
A few comments on the findings  from Corollaries \ref{cor:monotone} and \ref{cor:symmetry} are in order. If the observable is monotone (Corollary \ref{cor:monotone}), then it turns out that choosing the `mirror coupling' $\mathrm{d}B_t^x = - \mathrm{d}B_t^y$ in \eqref{eq:1d coupled BM} is optimal in the sense of Problem \ref{prob:linearised OT}. This result has a clear connection to popular variance reduction techniques such as `antithetic variates' \cite[Chapter 9.2]{kroese2013handbook}, where correlations between random variables are used to produce cancellations. In the case of symmetric observables (Corollary \ref{cor:symmetry}), optimal coupling in the sense of Problem \ref{prob:linearised OT} leads to a more sophisticated strategy: When the two particles (the locations of which are again denoted by $x$ and $y$) `are on the same side of the potential' (meaning that $x \ge 0$ and $y \ge 0$ or $x \le 0$ and $y \le 0$), then the Brownian motions should be coupled according to $\mathrm{d}B_t^x = - \mathrm{d}B_t^y$, as in the case of monotone observables. When the particles are on opposite sides ($x \ge 0$ and $y \le 0$ or $x \le 0$ and $y \ge 0$), according to Corollary \ref{cor:symmetry} it is best to switch to `synchronous coupling', $\mathrm{d}B_t^x = \mathrm{d}B_t^y$. Intuitively this can be understood as follows: By symmetry, the situation where $x \ge 0$ and $y \le 0$ with synchronous coupling ($\mathrm{d}B_t^x = \mathrm{d}B_t^y$) is equivalent to $x \ge 0$ and $y \ge 0$, with mirror coupling ($\mathrm{d}B_t^x = -\mathrm{d}B_t^y$). Since $f$ is monotone on $[0,\infty)$, this argument provides a plausible explanation for optimality by appealing to Corollary \ref{cor:monotone}. Finally, let us mention that numerical experiments show that using mirror coupling in the case of observables of the type encountered in Corollary \ref{cor:symmetry} (`naive antithetic variates') actually leads to a less effective sampler in terms of the asymptotic variance (see Figure \ref{fig:quad_obs}).

Let us consider now the same set-up as in the numerical experiments presented in Figure \ref{fig:overdamped coupling comparison}, i.e. we consider a Gaussian target measure ($V(x)=\frac{1}{2}x^2$), and the observables $f_1(x)=x$ (`linear'), $f_2(x)=x^2$ (`quadratic') and $f_3(x)=x^2-x$ (`mixed'). In Figure \ref{fig:overdamped variance}, we plot the asymptotic variances for $f_1$, $f_2$ and $f_3$, associated to different coupling schemes as a function of the coupling strength $\beta$. To be precise, the `Poisson' coupling is defined by 
\begin{equation}
\label{eq:beta coupling}
\alpha(x,y) = 
\begin{cases}
2 \sin \beta \cos \beta \quad & \mbox{if }\; \phi'(x) \phi'(y) \le 0, \\
-2\sin \beta \cos \beta \quad & \mbox{if }\; \phi'(x) \phi'(y) > 0, \\
\end{cases}
\end{equation}
$\phi$ being the solution to the Poisson equation \eqref{eq:1d Poisson} for the corresponding observable, and $\beta \in [0,\frac{\pi}{4}]$ denoting the coupling strength\footnote{For $\beta\in [0,\frac{\pi}{4}]$, the function $2 \sin \beta \cos$ is monotone, taking values in $[0,1]$. We chose this parametrisation in order for it to be consistent with \eqref{eq:1d coupling matrix}.}. For $\beta = 0$, we recover independent coupling, whereas $\beta = \frac{\pi}{4}$ leads to the optimal coupling from Proposition \ref{prop:1d overdamped opt asym var}. According to Corollaries \ref{cor:monotone} and \ref{cor:symmetry}, Poisson coupling coincides with mirror coupling for $f_1$ and with symmetric coupling for $f_2$. To illustrate the effect of couplings that are not tailored to the observable of interest, we also plot the asymptotic variances associated to symmetric coupling for $f_1$, mirror coupling for $f_2$, and both mirror and symmetric coupling for $f_3$. For $f_3$, we furthermore consider a coupling strategy that uses the derivative of the observable instead of the derivative of the solution to the Poisson equation, specifically, the coupling induced by \begin{equation}
\label{eq:observable coupling}
\alpha(x,y) = 
\begin{cases}
2 \sin \beta \cos \beta \quad & \mbox{if } \;f'(x) f'(y) \le 0, \\
-2\sin \beta \cos \beta \quad & \mbox{if } \; f'(x) f'(y) > 0. \\
\end{cases}
\end{equation}
The motivation for this is that in applications, the solution to the Poisson equation is often hard to obtain\footnote{However, often one aims to approximate the solution to the Poisson equation in order to use it as a control variate, see for instance \cite{dellaportas2012control,mijatovic2018poisson,RousselStoltz2017}. It suggests itself to use those approaches in conjunction with the coupling strategy developed here.}, whereas the gradient of the observable is readily available. By integration by parts we have
\begin{equation}
\label{eq:int by parts}
\int_{\mathbb{R}^d} \nabla f \cdot \nabla \phi \, \mathrm{d}\pi = \int_{\mathbb{R}^d} f^2 \, \mathrm{d}\pi \ge 0, 
\end{equation} 
suggesting us to use $\nabla f$ as a surrogate for $\nabla \phi$ (at equilibrium, the scalar product of $\nabla f $ and $\nabla \phi$ is positive on average).

In all the cases considered, the Poisson coupling turns out be the most efficient, uniformly in the coupling strength $\beta$. The fact that the absolute value of the derivative $\frac{\mathrm{d}\sigma_F^2}{\mathrm{d}\beta}\vert_{\beta=0}$  is maximal for Poisson coupling is precisely the content of Proposition \ref{prop:1d overdamped opt asym var}, whereas the fact that the asymptotic variance for Poisson coupling is maximal at $\beta = 0$ follows from Remark \ref{rem:optimal coupling improves}. It is interesting to note the monotonity of the asymptotic variance associated to Poisson coupling with respect to the coupling strength $\beta$; this phenomenon is not covered by our theory. Importantly, the efficiency of a certain coupling strongly depends on the considered observable. Indeed, the mirror coupling (which is excellent for the linear observable, see Figure \ref{fig:lin_obs}) leads to an increase of the asymptotic variance for the quadratic observable (see Figure \ref{fig:quad_obs}). Similarly, the symmetric coupling (suited for the quadratic observable), does not improve the performance for the linear observable (but the performance is also not worsened). In Figure \ref{fig:mixed_obs}, we observe that the coupling based on the derivative of the observable (see \eqref{eq:observable coupling}) works almost as well as the Poisson coupling, so this might be a reasonable choice in applications, although further studies are needed. For a comment about the minimum of the graph associated to mirror coupling for the mixed observable (see Figure \ref{fig:mixed_obs}) we refer to Remark \ref{rem:power expansion}. 
\begin{figure}
	\captionsetup[subfigure]{justification=centering}
	\caption{Dependence of the asymptotic variance on different coupling schemes and coupling strengths.}
	\label{fig:overdamped variance}
	\begin{subfigure}[b]{0.5 \textwidth}
		\includegraphics[width=\textwidth]{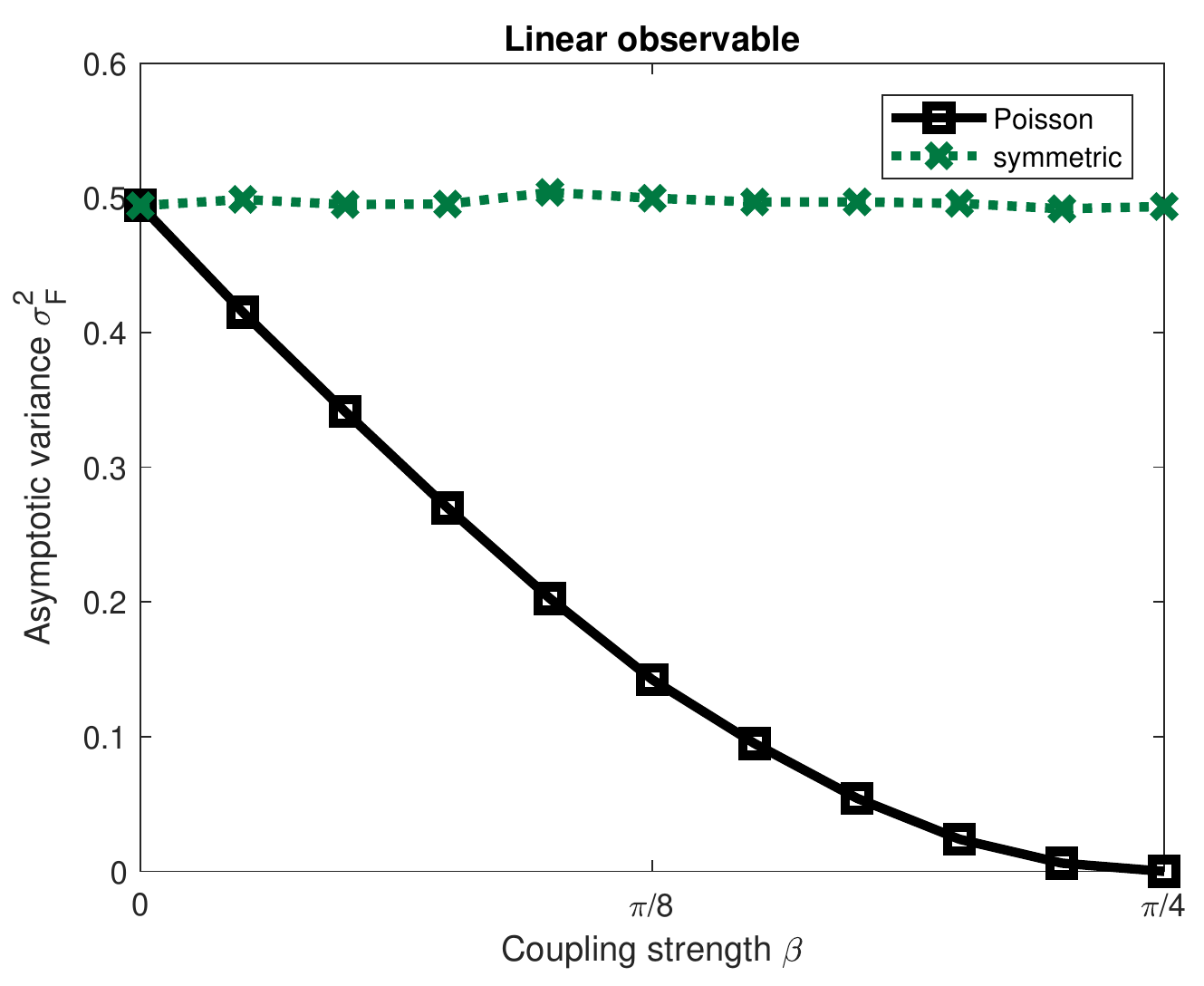}
		\subcaption{Linear observable $f_1(x) = x$.}
		\label{fig:lin_obs}
	\end{subfigure}
	\begin{subfigure}[b]{0.5 \textwidth}
		\includegraphics[width=\textwidth]{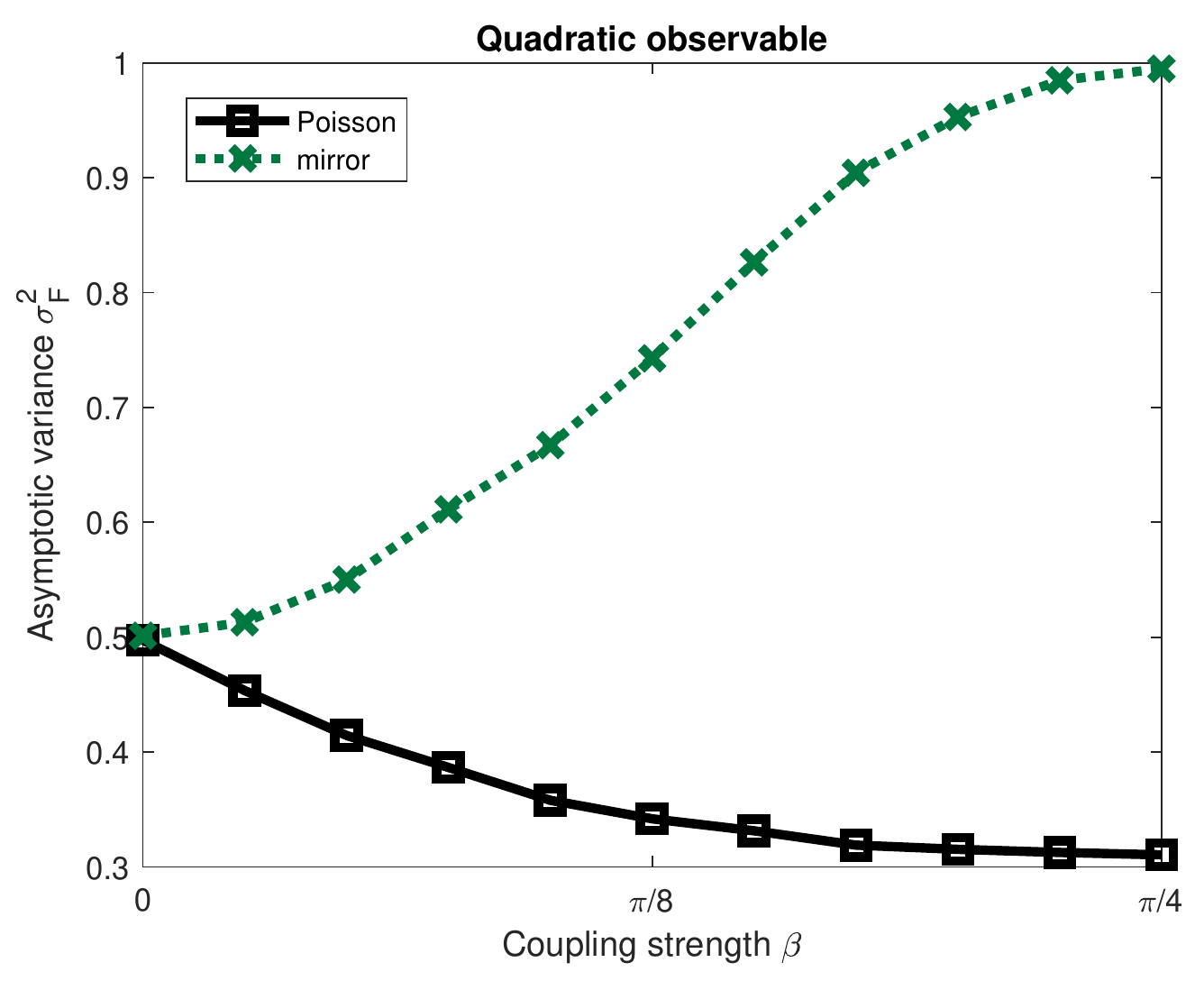}
		\subcaption{Quadratic observable $f_2(x) = x^2$.}
		\label{fig:quad_obs}
	\end{subfigure}
	\begin{subfigure}[b]{0.5 \textwidth}
		\includegraphics[width=\textwidth]{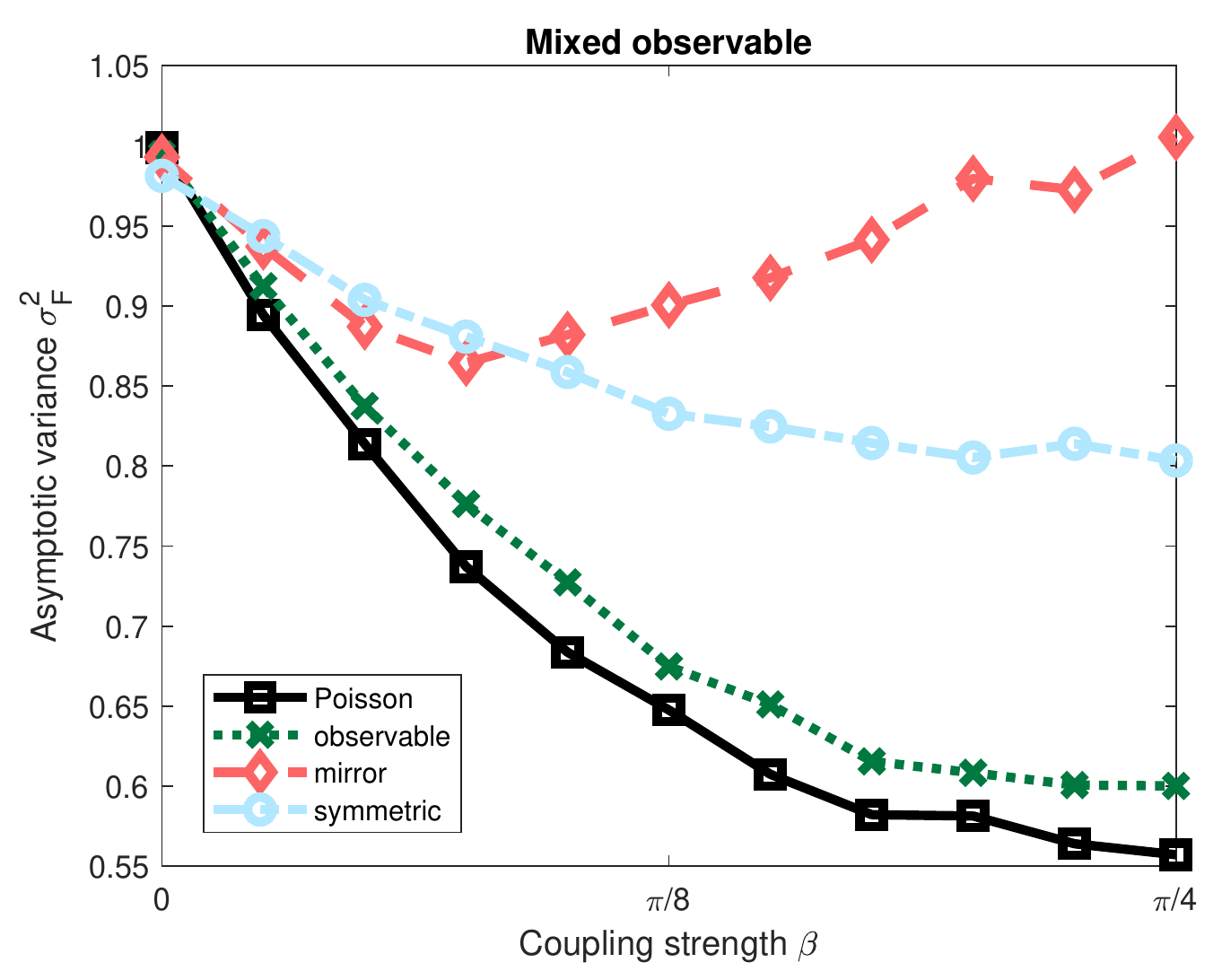}
		\subcaption{Mixed observable: $f_3(x) = x^2 - x$.}
		\label{fig:mixed_obs}
	\end{subfigure}
\end{figure}
\subsection{Overdamped Langevin dynamics with multiple particles in arbitrary dimensions}
\label{sec:many particles}

Let us extend the discussion from the previous section to arbitrary dimensions, first considering the case of two particles (as done in Example \ref{ex:overdamped 2 particles}). Using the expression \eqref{eq:Gamma 2 overdamped}, we see that
\begin{equation}
\delta\sigma_F^2(\Gamma_{\alpha}) = \int_{\mathbb{R}^d \times \mathbb{R}^d} \left(\nabla \phi (y) \cdot \alpha(x,y) \nabla \phi(x)\right) e^{-(V(x)+V(y))}\, \mathrm{d}x \mathrm{d}y,
\end{equation}
where $\alpha:\mathbb{R}^d \times \mathbb{R}^d \rightarrow \mathbb{R}^{d \times d}$ is a matrix-valued function satisfying \eqref{eq:overdamped 2 alpha bound}, i.e.
\begin{equation}
\alpha \in \mathcal{A} := \left\{ \alpha : \mathbb{R}^d \times \mathbb{R}^d \rightarrow \mathbb{R}^{d \times d} \;\, \text{measurable,  } \alpha(x,y)^T \alpha(x,y) \le I_{d \times d} \; \text{for all } x,y \in \mathbb{R}^d \right\},
\end{equation}
and $\phi$ is the solution to the Poisson equation
\begin{equation}
\label{eq:overdamped Poisson}
-(-\nabla V \cdot \nabla + \Delta) \phi = f, \quad \pi(\phi) = 0.
\end{equation}
Since \eqref{eq:overdamped 2 alpha bound} implies 
\begin{equation}
\vert \nabla \phi (y) \cdot \alpha(x,y) \nabla \phi(x) \vert \le \vert \nabla \phi (x) \vert \vert \nabla \phi(y) \vert, \quad x,y \in \mathbb{R},
\end{equation}
we get the following optimality result:
\begin{proposition}
	Assume that $\alpha^* \in \mathcal{A}$ is chosen such that
	\begin{equation}
	\label{eq:arbitraryd}
	\nabla\phi (y) \cdot \alpha^*(x,y) \nabla \phi (x) = - \vert \nabla \phi(x) \vert \vert \nabla \phi (y) \vert, \quad x,y \in \mathbb{R}^d.
	\end{equation} 
	Then $\alpha^*$ solves Problem \ref{prob:linearised OT}, i.e. $\delta\sigma_F^2(\Gamma_{\alpha^*}) \le \delta\sigma_F^2(\Gamma_\alpha)$ for all $\alpha \in \mathcal{A}$.
\end{proposition}
In the case when the solution $\phi$ to the Poisson equation is known it is straightforward to construct a matrix-valued function $\alpha^*$ such that  both \eqref{eq:overdamped 2 alpha bound} and \eqref{eq:arbitraryd} are satisfied. For instance, any orthogonal matrix trivially satisfies \eqref{eq:overdamped 2 alpha bound}, and \eqref{eq:arbitraryd} can be dealt with by choosing an appropriate rotation or reflection. As an example let us mention the following reflection in the plane spanned by $\nabla \phi(x)$ and $\nabla \phi(y)$:
\begin{equation}
\label{eq:alpha star}
\alpha^*(x,y) = \begin{cases}
I_{d \times d} - 2  \frac{\left( \widehat{\nabla \phi(x)} + \widehat{\nabla \phi(y)} \right)\left( \widehat{\nabla \phi(x)} + \widehat{\nabla \phi(y)} \right)^T}{\left( \widehat{\nabla \phi(x)} + \widehat{\nabla \phi(y)} \right)^2} \, &\text{if } \nabla\phi(x) \neq 0, \nabla \phi(y) \neq 0, \vspace{-5pt}\\
& \hspace{10pt}\widehat{\nabla \phi(x)} + \widehat{\nabla \phi(y)} \neq 0,
\vspace{10pt}\\
I_{d \times d} & \text{otherwise.}
\end{cases}
\end{equation}
Here, $\widehat{\nabla \phi} = \frac{\nabla\phi}{\vert \nabla \phi \vert}$ is used to denote the normalised gradient of $\phi$. As mentioned in Section \ref{sec:numerics 1d overdamped}, the solution to the Poisson equation is usually hard to obtain in applications (but the popular methodology using control variates relies on approximations thereof). Inspired by the integration by parts formula \eqref{eq:int by parts}, it seems reasonable to use the normalised gradient of the observable $\widehat{\nabla f}$ as a surrogate for $\widehat{\nabla \phi}$, i.e.
\begin{equation}
\label{eq:alpha f}
\alpha_f(x,y) = \begin{cases}
I_{d \times d} - 2  \frac{\left( \widehat{\nabla f(x)} + \widehat{\nabla f(y)} \right)\left( \widehat{\nabla f(x)} + \widehat{\nabla f(y)} \right)^T}{\left( \widehat{\nabla f(x)} + \widehat{\nabla f(y)} \right)^2} \, &\text{if } \nabla f(x) \neq 0, \nabla f(y) \neq 0, \vspace{-5pt}\\
& \hspace{10pt}\widehat{\nabla f(x)} + \widehat{\nabla f(y)} \neq 0,
\vspace{10pt}\\
I_{d \times d} & \text{otherwise.}
\end{cases}
\end{equation}  
We recall that a comparison of the couplings associated to \eqref{eq:alpha star} and \eqref{eq:alpha f} was performed in the one-dimensional case (see Figure \ref{fig:mixed_obs}) where $\alpha_f$ almost achieved the same reduction of the asymptotic variance as $\alpha^*$. Based on \eqref{eq:int by parts} and numerical experiments we conjecture that choosing $\alpha_f$ guarantees an improvement in terms of the asymptotic variance for small perturbations:
\begin{conjecture}
	Let $f \in L_0^2(\pi)$ and $\phi \in L_0^2(\pi)$ be the corresponding solution to the Poisson equation \eqref{eq:overdamped Poisson}. Then 
	\begin{equation}
	\delta\sigma_F^2(\Gamma_{\alpha}) = \int_{\mathbb{R}^d \times \mathbb{R}^d} \left(\nabla \phi (y) \cdot \alpha_f(x,y) \nabla \phi(x)\right) e^{-(V(x)+V(y))}\, \mathrm{d}x \mathrm{d}y \le 0.
	\end{equation}	
\end{conjecture}
The complexity of the foregoing optimisation problems is increased substantially when considering more than two particles. From a practical perspective, it is desirable to specify the coupling in terms of the matrix-valued function $G$ appearing in \eqref{eq:overdamped full SDE} since this formulation is needed for the implementation of the numerical scheme. The linearised optimisation objective (Problem \ref{prob:linearised OT}) however is formulated in terms of the coupling operator $\Gamma$. Passing from the latter to $G$ involves the computationally expensive task of computing the square root of the matrix $Q$ defined in \eqref{eq:Q definition}. The construction and effective implementation of optimally coupled samplers with multiple particles therefore remains a subject for future work, but could be based on the results for the case of two particles. To give an impression, let us outline an idea based on the notation introduced in Remark \ref{rem:general dynamics overdamped}. It is natural to choose the orthogonal matrices $g_{ij}$ describing the coupling between the $i$th and the $j$th particle according to \eqref{eq:alpha star}, i.e.
\begin{equation}
\label{eq:gij weights}
g_{ij}(x_1,\ldots,x_n) = \begin{cases}
I_{d \times d} - 2  \frac{\left( \widehat{\nabla \phi(x_i)} + \widehat{\nabla \phi(x_j)} \right)\left( \widehat{\nabla \phi(x_i)} + \widehat{\nabla \phi(x_j)} \right)^T}{\left( \widehat{\nabla \phi(x_i)} + \widehat{\nabla \phi(x_j)} \right)^2} \, &\text{if } \nabla\phi(x_i) \neq 0, \nabla \phi(x_j) \neq 0, \vspace{-5pt}\\
& \hspace{10pt}\widehat{\nabla \phi(x_i)} + \widehat{\nabla \phi(x_j)} \neq 0,
\vspace{10pt}\\
I_{d \times d} & \text{otherwise,}
\end{cases}
\end{equation}    
or, when the solution $\phi$ to the Poisson equation (or an approximation thereof) is not available, according to \eqref{eq:alpha f}. Since the benefit of the coupling in terms of reducing the asymptotic variance is directly related to the value of the expression in \eqref{eq:arbitraryd}, it is plausible to choose the weights $w_{ij}$ (see \eqref{eq:weights}) in such a way that particle $x_i$ is preferentially coupled to particle $x_j$ if $\vert \nabla \phi(x_i) \vert$ and $\vert \nabla \phi(x_j) \vert$ are similar in magnitude. To make this precise, denote by $\sigma : \{1,\ldots,n \} \rightarrow \{1,\ldots,n\}$ the permutation that orders the particles according to $\vert \nabla \phi \vert$, i.e. 
\begin{equation}
\vert \nabla \phi(x_{\sigma(1)}) \vert \ge \vert \nabla \phi(x_{\sigma(2)}) \vert \ge \ldots \ge \vert \nabla \phi(x_{\sigma(n-1)}) \vert \ge \vert \nabla \phi(x_{\sigma(n)}) \vert.
\end{equation}
Then, denoting the coupling strength by $\beta \in [0,\frac{\pi}{4}]$, we can set the weights as follows:
\begin{subequations}
	\label{eq:sort weights}
	\begin{align}
	& w_{ii} = \cos \beta, \quad \quad i=1, \ldots,n, \nonumber \\
	w_{\sigma(1)\sigma(2)} = \sin \beta, \quad & w_{\sigma(3)\sigma(4)}  = \sin \beta, \quad\ldots \quad
	w_{\sigma(n-1)\sigma(n)} = \sin \beta, \tag{\ref*{eq:sort weights}} \\
	& w_{ij} = 0\quad \text{otherwise.} \nonumber 
	\end{align}
\end{subequations}
Let us emphasize that the sorting of the particles according to $\vert \nabla \phi\vert$ is supposed to be performed at every time step. We have compared this coupling strategy to simple pairwise coupling without sorting\footnote{This is equivalent to running $n/2$ two-particle samplers independently in parallel.}, i.e. replacing the second line of  \eqref{eq:sort weights} by 
\begin{equation}
w_{12} = \sin \beta, \quad  w_{34}  = \sin \beta, \quad\ldots \quad
w_{n-1,n} = \sin \beta,
\end{equation}
for the example of sampling a standard Gaussian measure ($V=\frac{1}{2}\vert x \vert^2$) in $d=10$ dimensions with $n=10$ particles for the quadratic observable $f_1(x) = \frac{1}{2}\vert x \vert^2$ and the mixed observable $f_2(x) = 5 \vert x \vert^2 + l \cdot x$, where $l = (1,\ldots,1)$. As Figure \ref{fig:sorting} shows, the sorting strategy as detailed in \eqref{eq:sort weights} leads to a smaller asymptotic variance in comparison to simple pairwise couplings. 
\begin{figure}
	\captionsetup[subfigure]{justification=centering}
	\caption{Comparison between pairwise couplings with and without sorting according to $\vert \nabla \phi \vert$ for a Gaussian target measure in $d=10$ dimensions and with $n=10$ particles.}
	\label{fig:sorting}
	\begin{subfigure}[b]{0.5 \textwidth}
		\includegraphics[width=\textwidth]{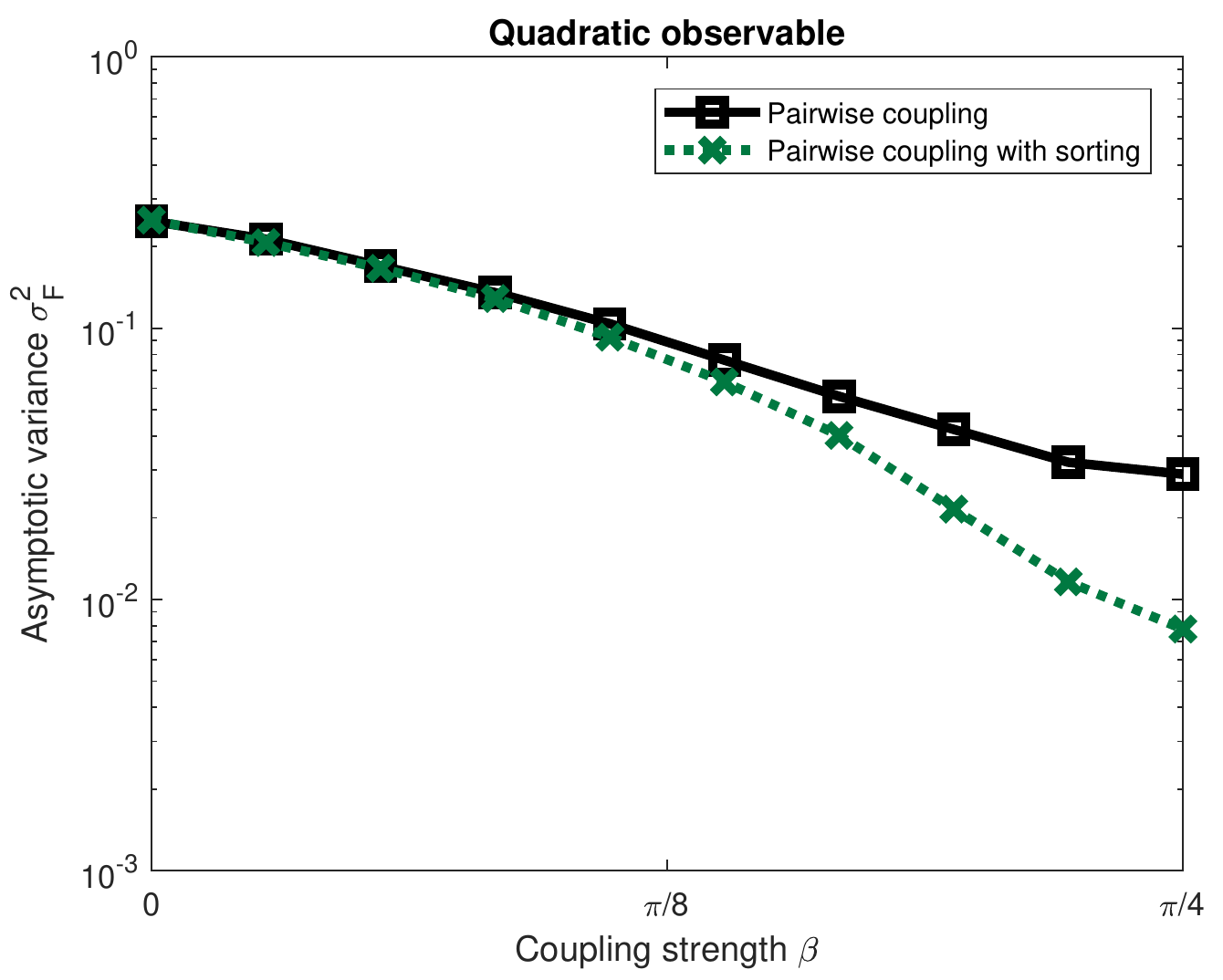}
		\caption{Quadratic observable: $f(x) = \frac{1}{2}\vert x \vert^2$.}
	\end{subfigure}
	\begin{subfigure}[b]{0.5 \textwidth}
		\includegraphics[width=\textwidth]{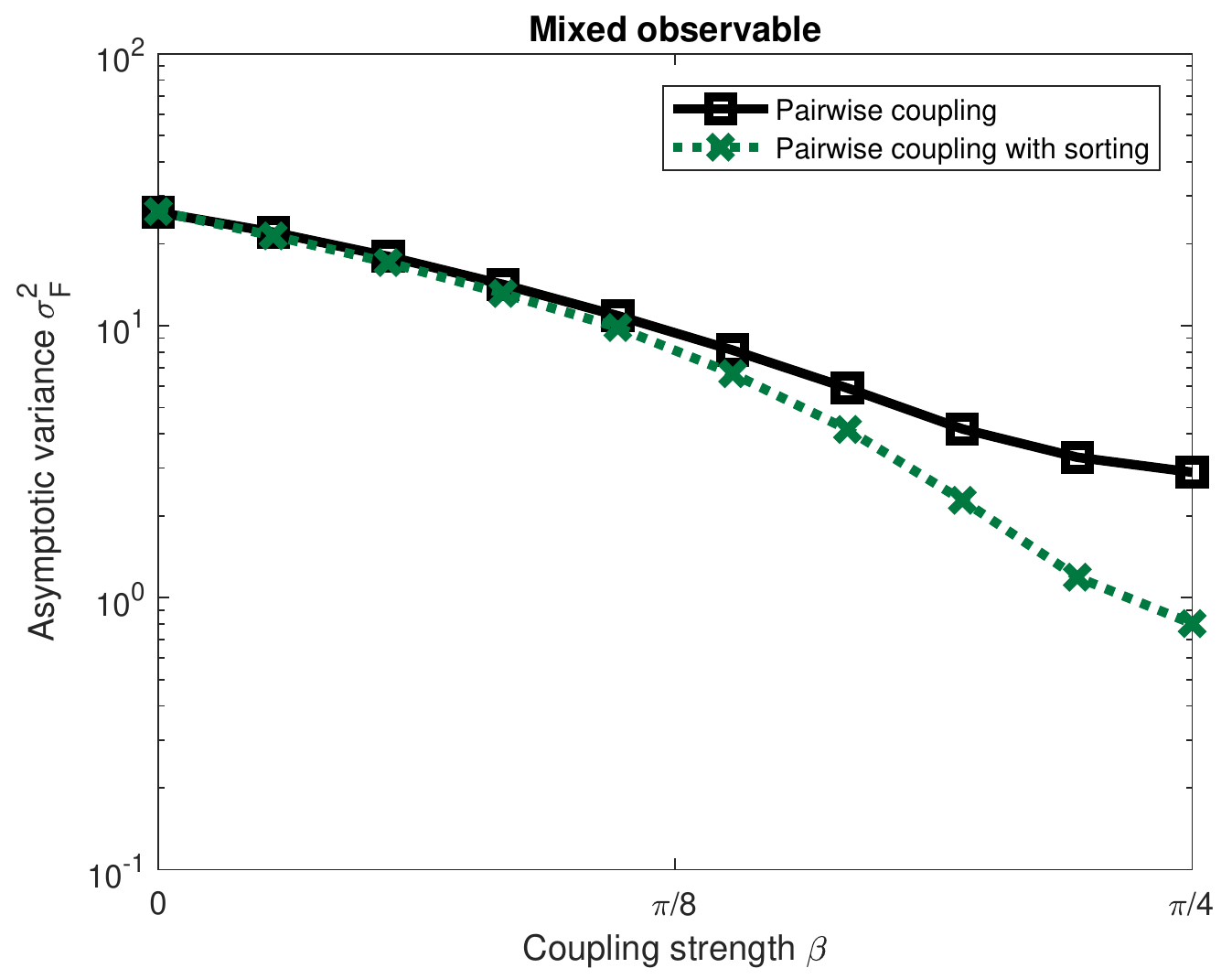}
		\caption{Mixed observable: $f(x) = 5\vert x \vert^2 + l \cdot x$.} 
	\end{subfigure}
\end{figure}
\subsection{The zigzag process}
Recall the setting from Section \ref{sec:zigzag} and fix an observable of interest $f \in L_0^2(\tilde{\pi})$ \footnote{We use the notation $\tilde{\pi}(\mathrm{d}x) = \frac{1}{Z} e^{-V(x)} \mathrm{d}x$ to distinguish it from the invariant measure $\pi$ on the full space $\mathbb{R} \times \{-1,1\}$, given in \eqref{eq:zigzag invariant measure}.}. For notational convenience, let us introduce the shorthands 
\begin{subequations}
	\begin{align*}
	& \alpha_{++}(x,y):= \alpha(x,y,+1,+1), \quad & \alpha_{+-}(x,y):= \alpha(x,y,+1,-1), \\
	& \alpha_{-+}(x,y):= \alpha(x,y,-1,+1), \quad 
	& \alpha_{--}(x,y):= \alpha(x,y,-1,-1). 
	\end{align*}
\end{subequations}
Taking the constraint \eqref{eq:zig zag constraint} into account, we will optimise over the set
\begin{equation*}
\mathcal{A} = \{\alpha:\mathbb{R}^2 \times\{-1,1\}^2 \rightarrow \mathbb{R} \,\vert \, \alpha\,\mbox{measurable, } 0 \le \alpha \le \min \left(\lambda(x,\theta_x),\lambda(y,\theta_y)\right) \}.
\end{equation*}
The corresponding coupling operators (see \eqref{eq:zigzag Gamma}) will be denoted by $\Gamma_{\alpha}$. We have the following lemma the proof of which can be found in Appendix \ref{sec:proofs_pert}.
\begin{lemma}
	\label{lem:zigzag dsigma}
	The zigzag process satisfies
	\begin{subequations}
		\label{eq:asym var loc zig zag}
		\begin{align}
		\delta\sigma_F^2(\Gamma_{\alpha}) & = \frac{1}{4} \int_{{\mathbb{R}^2}}  \tilde{\alpha}(x,y)  
		\cdot \tilde{\phi}'(x) \tilde{\phi}'(y) e^{-\left(V(x) + V(y)\right)}\,\mathrm{d}x \mathrm{d}y, \nonumber \\
		\tilde{\alpha}(x,y) & = \alpha_{++}(x,y) + \alpha_{--}(x,y) - \alpha_{+-}(x,y) - \alpha_{-+}(x,y),
		\tag{\ref*{eq:asym var loc zig zag}} 
		\end{align} 
	\end{subequations}
	where $\tilde{\phi} \in L^2(\tilde{\pi})$ is a solution to 
	\begin{equation}
	\label{eq:Zig zag auxiliary Poisson}
	-(-V'\tilde{\phi}' + \tilde{\phi}'') = f.
	\end{equation}
\end{lemma}
\begin{remark}
	Observe the remarkable coincidence that \eqref{eq:Zig zag auxiliary Poisson} coincides with the Poisson equation \eqref{eq:1d Poisson} for the overdamped Langevin dynamics. We employ the notation $\tilde{\phi}$ to distinguish \eqref{eq:Zig zag auxiliary Poisson} from the Poisson equation \eqref{eq:Poisson Zig-Zag} in the whole space $\bar{E} = \mathbb{R}^2 \times \{-1,1\}^2$. 
\end{remark}
The following result is immediate from the expression \eqref{eq:asym var loc zig zag}:
\begin{proposition}
	\label{prop:optimal zigzag}
	Let $\alpha^* \in \mathcal{A}$ be given by
	\begin{equation}
	\alpha^*(x,y,\theta_x,\theta_y) = 
	\begin{cases}
	\min \left(\lambda(x,\theta_x),\lambda(y,\theta_y)\right),\quad & \mbox{if } \tilde{\phi}'(x)\tilde{\phi}'(x) \theta_x \theta_y \le 0 \\
	0, \quad & \mbox{otherwise.}
	\end{cases}
	\end{equation}
	Then $\Gamma_{\alpha^*}$ solves Problem \ref{prob:linearised OT} in the sense that $\delta\sigma_F^2 (\Gamma_{\alpha^*}) \le \delta\sigma_F^2 (\Gamma_{\alpha})$ for all $\alpha \in \mathcal{A}$.
\end{proposition}
\begin{remark}
	The comment from Remark \ref{rem:optimal coupling improves} applies here as well.
\end{remark}
Combining Lemma \ref{lem:1d Poisson} with Proposition \ref{prop:optimal zigzag} immediately yields the following corollaries:
\begin{corollary}
	\label{cor:zigzag monotone}
	In the setting from the first part of Lemma \ref{lem:1d Poisson}, 
	\begin{equation}
	\label{eq:zigzag mirror}
	\alpha^*(x,y,\theta_x,\theta_y) = 
	\begin{cases}
	\min \left(\lambda(x,\theta_x),\lambda(y,\theta_y)\right),\quad & \mbox{if } \theta_x \theta_y \le 0 \\
	0, \quad & \mbox{otherwise,}
	\end{cases}
	\end{equation}
	solves Problem \ref{prob:linearised OT}, in the sense that $\delta\sigma_F^2 (\Gamma_{\alpha^*}) \le \delta\sigma_F^2 (\Gamma_{\alpha})$ for all $\alpha \in \mathcal{A}$.
\end{corollary}
\begin{corollary}
	\label{cor:zigzag symmetry}
	In the setting from the second part of Lemma \ref{lem:1d Poisson},
	\begin{equation}
	\label{eq:zigzag symmetric}
	\alpha^*(x,y,\theta_x,\theta_y) = 
	\begin{cases}
	\min \left(\lambda(x,\theta_x),\lambda(y,\theta_y)\right),\quad  & \mbox{if } x y \theta_x \theta_y \le 0 \\
	0, \quad & \mbox{otherwise, } \\
	\end{cases}
	\end{equation}
	solves Problem \ref{prob:linearised OT}, in the sense that $\delta\sigma_F^2 (\Gamma_{\alpha^*}) \le \delta\sigma_F^2 (\Gamma_{\alpha})$ for all $\alpha \in \mathcal{A}$.	
\end{corollary}
The results from Corollaries \ref{cor:zigzag monotone} and \ref{cor:zigzag symmetry} can be interpreted intuitively in the following way. As already pointed out in Section \ref{sec:zigzag}, setting $\alpha(x,y,\theta_x,\theta_y) = \min \left(\lambda(x,\theta_x),\lambda(y,\theta_y)\right)$ encourages simultaneous flips of the velocities $\theta_x$ and $\theta_y$ (when the particles are at locations $x$ and $y$, with velocities $\theta_x$ and $\theta_y$), whereas the flips occur independently if $\alpha(x,y,\theta_x,\theta_y) = 0$. The coupling associated to \eqref{eq:zigzag mirror} therefore leads to an increased probability of simultaneous flips precisely when the two particles move in opposite directions. Observe that simultaneous flips preserve the value of $\theta_x \theta_y$, while single flips change its sign. As a consequence, the relative amount of time during which the two particles move in opposite directions is increased by the coupling associated to \eqref{eq:zigzag mirror}. Similarly to the case of mirror coupling for overdamped Langevin diffusions (see the discussion following Corollary \ref{cor:symmetry}), it is plausible that this dynamics leads to cancellations for monotone observables in the spirit of antithetic variates. The interpretation of Corollary \ref{cor:zigzag symmetry} is analogous to the one of Corollary \ref{cor:symmetry}. For illustration, we consider again the case of a quadratic potential $V(x)= \frac{1}{2} x^2$ (i.e. a Gaussian target measure) and a linear observable $f(x)=x$. The coupling is chosen according to Corollary \ref{cor:zigzag monotone}, i.e. in a suitable manner for the linear observable, modulated by a parameter $\beta \in [0,1]$, analogously to \eqref{eq:beta coupling} and \eqref{eq:observable coupling}. In Figures \ref{fig:zigzag variance}, \ref{fig:zigzag rel time}, and \ref{fig:zigzag distance} we plot the associated asymptotic variance, the relative time the particles move in opposite directions, as well as the average distance between the particles. Those graphs support the foregoing intuitive arguments. The fact that the average distance between the particles increases with the strength of the coupling is interesting, since it suggests that the state space can be explored more efficiently by using appropriate couplings. In Figure \ref{fig:zigzag trajectory} we plot a typical trajectory of the joint system. Comparing this graph with the optimal transport map depicted in \ref{fig:lin OT map}, we conclude that the solution to Problem \ref{prob:linearised OT} found in Proposition \ref{prop:optimal zigzag} is somewhat close to the solution of the Kantorovich problem, but not nearly as much as the corresponding solution in the case of overdamped Langevin dynamics. Interestingly, the aforementioned similarity is much more pronounced in the case when the target distribution is heavy-tailed. As an example, we plotted a typical trajectory of a  mirror-coupled zigzag process targeting a Cauchy distribution in Figure \ref{fig:cauchy_trajectory}. We did perform numerical experiments for quadratic observables. For them, an improvement in the asymptotic variance is hardly noticeable. Furthermore, a typical trajectory for the coupling induced by \eqref{eq:zigzag symmetric} very much resembles the typical trajectories for the independent coupling. As it seems, couplings of zigzag processes are not very efficient in the setting of Lemma \ref{lem:1d Poisson}.\ref{it:symmetry}. A possible explanation is that piecewise deterministic Markov processes are more `rigid' than diffusions (in fact, by definition, they move deterministically during a considerable time span), allowing less flexibility in terms of couplings.

\begin{figure}
	\thisfloatpagestyle{empty}
	\captionsetup[subfigure]{justification=centering}
	\caption{Coupling for the zigzag process according to Corollary \ref{cor:zigzag monotone}.}
	\label{fig:zigzag}
	\begin{subfigure}[b]{0.5 \textwidth}
		\includegraphics[width=\textwidth]{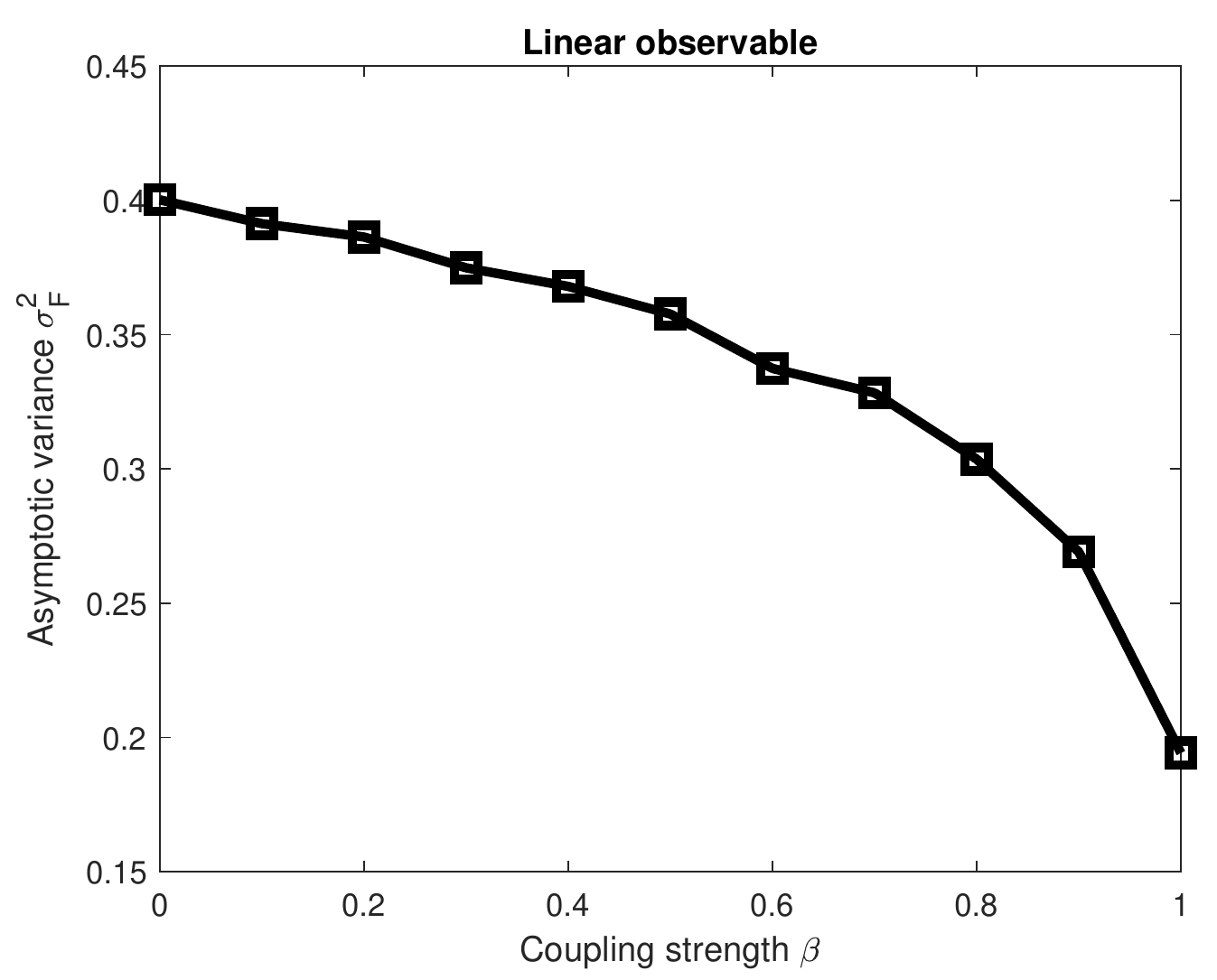}
		\subcaption{Asymptotic variance associated to the linear observable $f(x) = x$, depending on the coupling strength $\beta$.}
		\label{fig:zigzag variance}
	\end{subfigure}
	\begin{subfigure}[b]{0.5 \textwidth}
		\includegraphics[width=\textwidth]{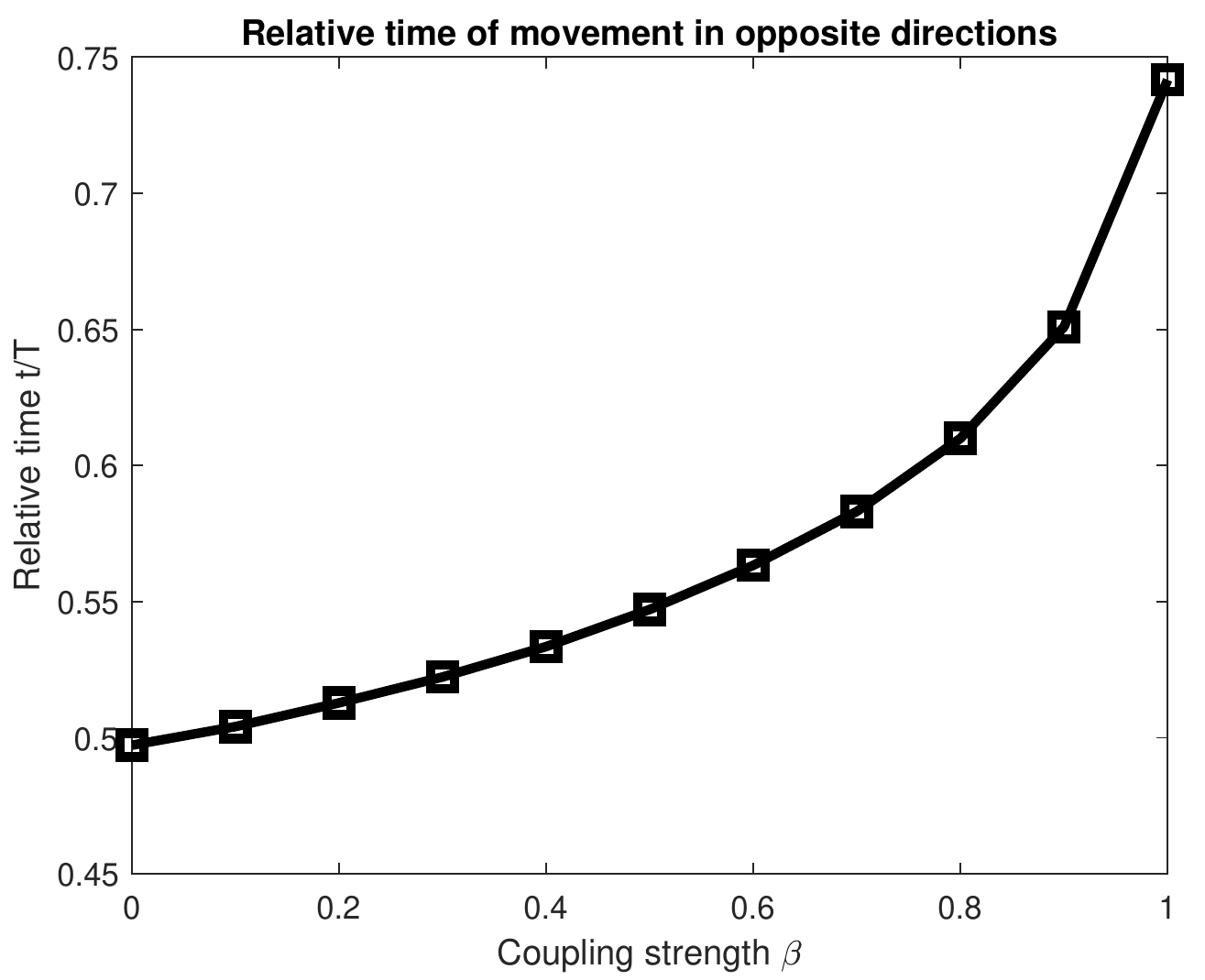}
		\subcaption{Relative amount of time that the particles move in opposite directions, depending on the coupling strength $\beta$.}
		\label{fig:zigzag rel time}
	\end{subfigure}
	\begin{subfigure}[b]{0.5 \textwidth}
		\includegraphics[width=\textwidth]{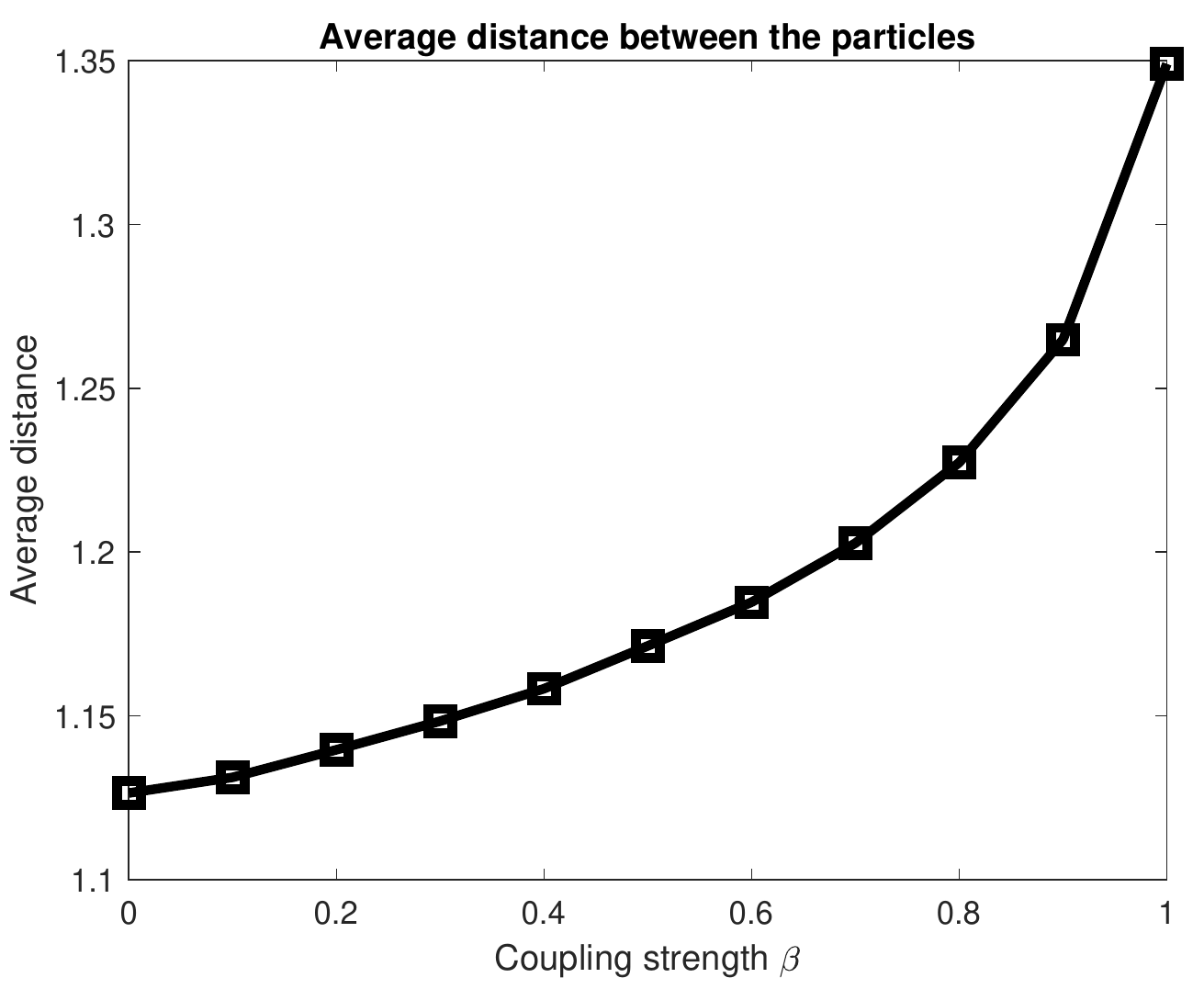}
		\subcaption{Average distance between the two particles, depending on the coupling strength $\beta$.}
		\label{fig:zigzag distance}
	\end{subfigure}
	\begin{subfigure}[b]{0.5 \textwidth}
		\includegraphics[width=\textwidth]{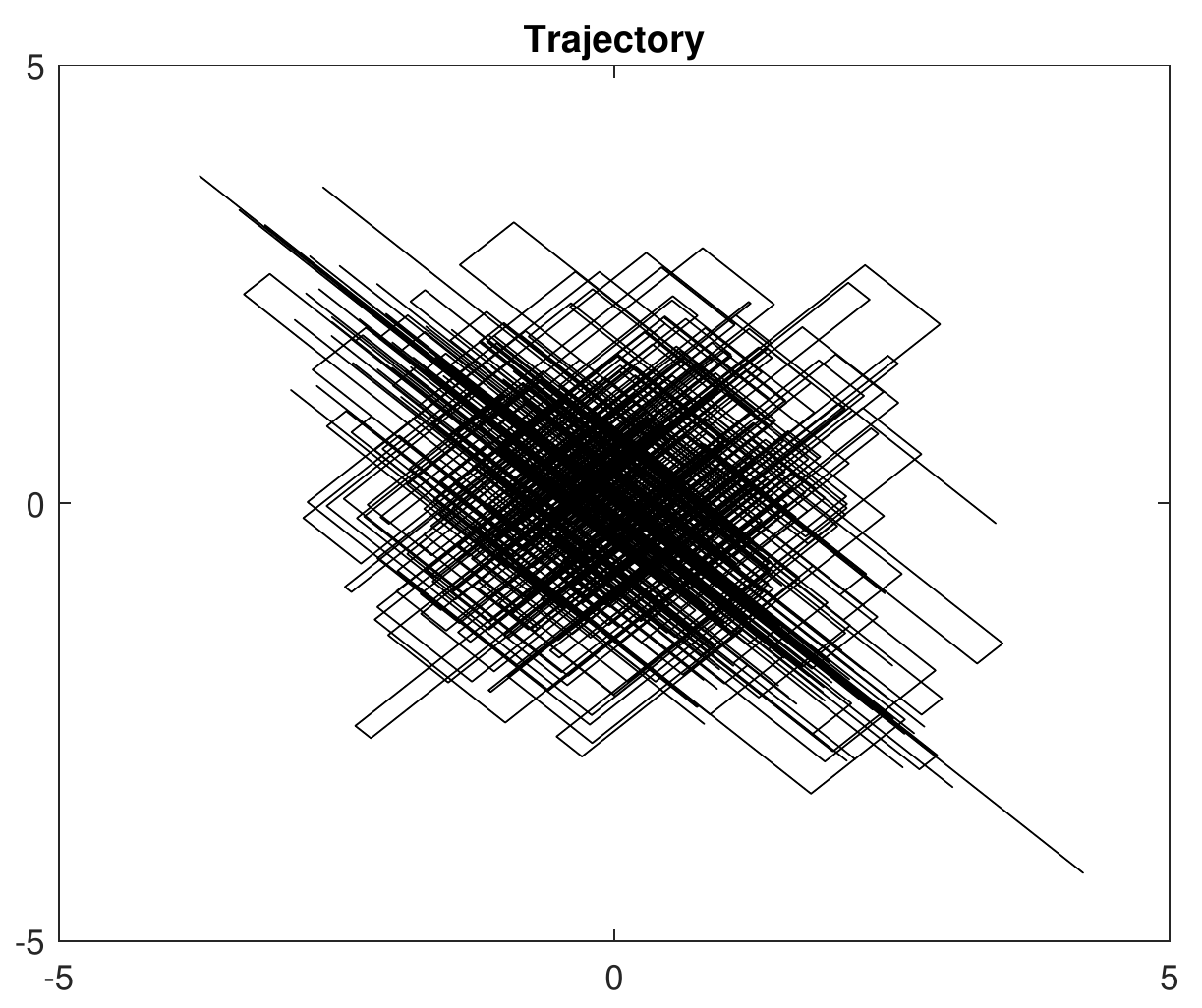}
		\subcaption{Part of the trajectory for a Gaussian target with mirror coupling.}
		\label{fig:zigzag trajectory}
	\end{subfigure}
	\begin{subfigure}[b]{0.5 \textwidth}
		\includegraphics[width=\textwidth]{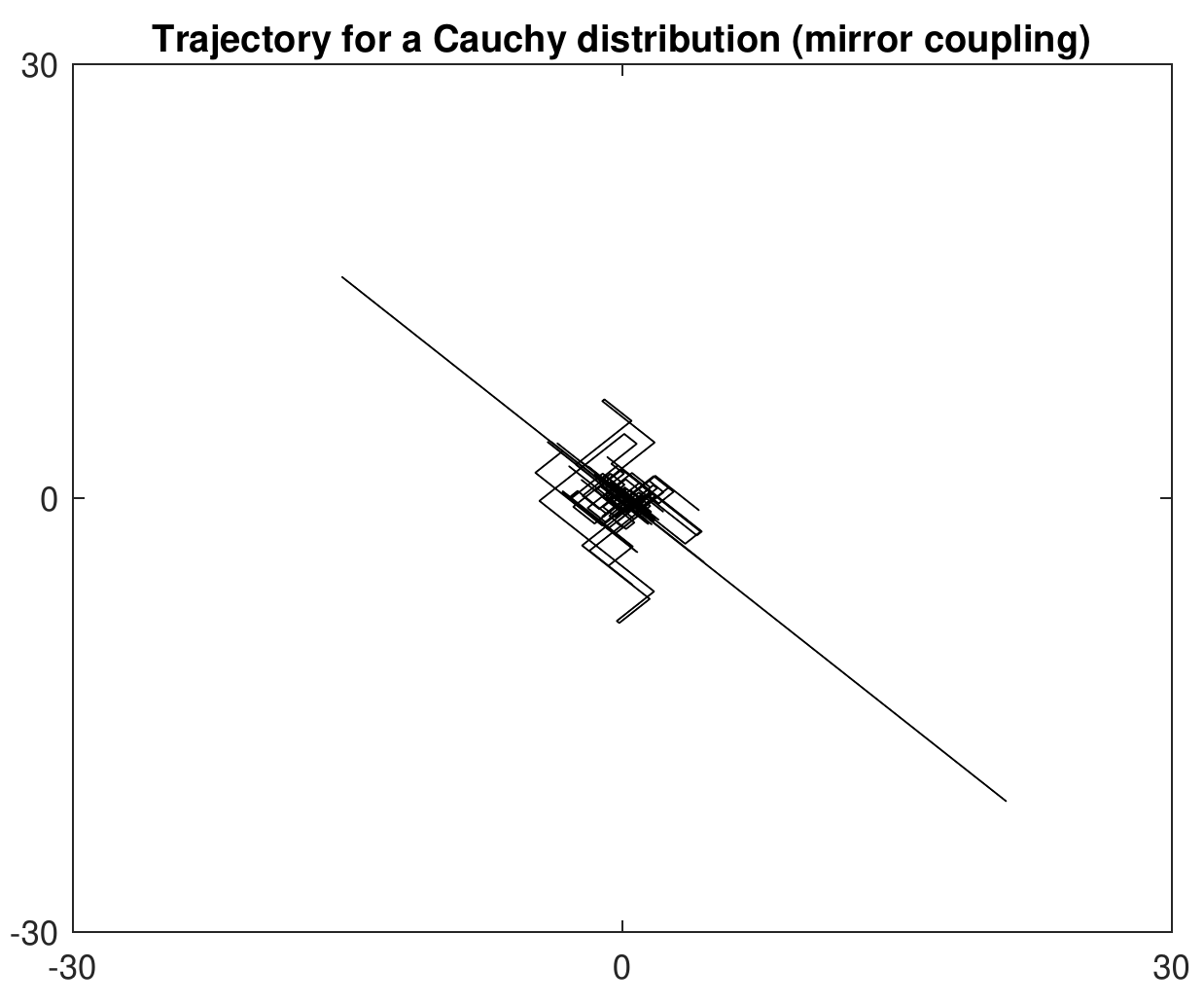}
		\subcaption{Part of the trajectory for a Cauchy target with mirror coupling.}
		\label{fig:cauchy_trajectory}
	\end{subfigure}
\end{figure} 

\section{A remark on the rate of convergence to equilibrium}
\label{ch:coupling_spectral_gap}

In this section, we study the rate of convergence to equilibrium for coupled processes. For convenience, let us assume that the spaces $E_i$ and the operators $\mathcal{L}_i$ are identical copies of each other. When addressing the marginal process(es), we will usually suppress the indices and write $\mathcal{L}$ and $E$. Furthermore, let us fix an ergodic coupling operator $\Gamma \in \mathcal{G}^0$ and denote as usual the corresponding generator and semigroup by $\bar{\mathcal{L}}_\Gamma = \bar{\mathcal{L}}_0 + \Gamma$ and $(\bar{S}^{\Gamma}_t)_{t \ge 0}$ respectively.  In the sequel, we will make use of the following subspace of centred  observables in $L_0^2(\bar{\pi}_\Gamma)$:
\begin{equation}
\tilde{L}_0^2(\bar{\pi}_\Gamma) := \left\{F \in L_0^2(\bar{\pi}_{\Gamma}) \vert\, \text{there exists } f \in L_0^2(\pi) \, \text{such that } F = \frac{1}{n} \sum_{i=1}^n f_i \right\} \subseteq L_0^2(\bar{\pi}_\Gamma).
\end{equation}
Clearly, the space $\tilde{L}_0^2(\bar{\pi}_\Gamma)$ comprises the observables of interest in our framework. By using the extension operator $\Pi^*$ from \eqref{eq:extension operator}, $\tilde{L}_0^2(\bar{\pi}_\Gamma)$ can equivalently be defined via $\tilde{L}_0^2(\bar{\pi}_\Gamma) = \Pi^* L_0^2(\bar{\pi}_\Gamma)$. The main result of this chapter is 
the following characterisation of exponential convergence to equilibrium. 
\begin{theorem}
	\label{thm:Poincare}
	For $\Lambda>0$, the following are equivalent:
	\begin{enumerate}
		\item 
		\emph{Poincar\'{e} inequality:}
		\begin{equation}
		\label{eq:coupling_Poincare inequality}
		\langle F, F \rangle_{L^2(\bar{\pi}_{\Gamma})} \le \frac{1}{\Lambda} \langle F, (-\bar{\mathcal{L}}_{\Gamma})F  \rangle_{L^2(\bar{\pi}_{\Gamma})} 
		\end{equation}
		for all $F \in \mathcal{D}(\bar{\mathcal{L}}_\Gamma)\cap\tilde{L}_0^2(\bar{\pi}_\Gamma)$.
		\item 
		\emph{Exponential decay:}
		\begin{equation}
		\label{eq:exp decay}
		\Vert \bar{S}^{\Gamma}_t F \Vert^2_{L^2(\bar{\pi}_{\Gamma})} \le e^{-2\Lambda t} \Vert F \Vert^2_{L^2(\bar{\pi}_{\Gamma})}, \quad t\ge 0,
		\end{equation}
		for all $F \in \tilde{L}_0^2(\bar{\pi}_\Gamma)$.
	\end{enumerate}
\end{theorem}
\begin{remark}
	Theorem \ref{thm:Poincare} is well known if $\tilde{L}^2_0(\bar{\pi}_\Gamma)$ is replaced by the whole space $L^2_0(\bar{\pi}_\Gamma)$, see \cite[Theorem 4.2.5]{bakry2013analysis}. For our purposes however, it is natural to restrict attention to the smaller space $\tilde{L}^2_0(\bar{\pi}_\Gamma)$. In particular, by the duality $(\Pi \bar{\pi})(f) = \bar{\pi}(\Pi^* f)$ explained in the introduction, the decay estimate \eqref{eq:exp decay} implies exponential convergence of the laws $(\Pi \bar{\pi}_t)_{t \ge 0}$, with the same rate. 
\end{remark}
The proof of Theorem \ref{thm:Poincare} relies on the following lemma:

\begin{lemma}
	\label{lem:flow invariance}
	Let $F \in \tilde{L}_0^2(\bar{\pi}_\Gamma)$ with $F = \frac{1}{n} \sum_{i=1}^n f_i$, $f \in L_0^2(\pi)$. Then 
	\begin{equation}
	\bar{S}^{\Gamma}_t F = \frac{1}{n}\sum_{i=1}^n (S_t f)_i, \quad t \ge 0.
	\end{equation}
	In particular, $\tilde{L}_0^2(\bar{\pi}_\Gamma)$ is invariant under the flow of $(\bar{S}^{\Gamma}_t)_{t \ge 0}$.
\end{lemma}

\begin{proof}
	For $F \in C_b(\bar{E}) \cap \tilde{L}_0^2(\bar{\pi}_\Gamma)$ we have
	\begin{subequations}
		\begin{align*}
		(\bar{S}^\Gamma_t F) (x_1, \ldots, x_n) 
		&  =  \mathbb{E} [F(\bar{X}_t) \vert \bar{X}_0 = (x_1,\ldots, x_n)] \\
		& =  \frac{1}{n} \sum_{i=1}^n \mathbb{E} [f(X_t^i) \vert \bar{X}_0 = (x_1,\ldots, x_n)] \\
		& = \frac{1}{n} \sum_{i=1}^n \mathbb{E} [f(X_t^i) \vert X^i_0 = x_i] = \frac{1}{n} \sum_{i=1}^n (S_t  f) (x_i).
		\end{align*}
	\end{subequations}
	Between the second and the third line, we used the fact that the process $(\bar{X}_t)_{t \ge 0}$ has $(X_t^i)_{t \ge 0}$ as its $i$th marginal, so in particular, the law of $f(X_t^i)$ depends on the initial condition $\bar{X}_0$ only through $X_0^i = x_i$. For arbitrary $F \in \tilde{L}_0^2(\bar{\pi}_\Gamma)$ the result follows by a standard density argument.
\end{proof}
\begin{proof}
	[Proof of Theorem \ref{thm:Poincare}] The proof is verbatim the same as for the usual result. However, the fact that $\tilde{L}_0^2(\bar{\pi}_\Gamma)$ is invariant under the flow $(\bar{S}^{\Gamma}_t)_{t \ge 0}$ is crucial.
	For completeness let us sketch the proof: Let $F \in \tilde{L}_0^2(\bar{\pi}_\Gamma)$ and assume that the Poincar\'{e} inequality \eqref{eq:coupling_Poincare inequality} holds for some constant $\Lambda > 0$. Then
	\begin{equation}
	\frac{\mathrm{d}}{\mathrm{d}t} \left( \frac{1}{2} \Vert \bar{S}^\Gamma_t F \Vert^2_{L^2(\bar{\pi}_{\Gamma})} \right) = \langle \bar{S}^\Gamma_t F, \bar{\mathcal{L}}_{\Gamma} \bar{S}^\Gamma_t F \rangle_{L^2(\bar{\pi}_{\Gamma})} \le - \Lambda \langle \bar{S}^\Gamma_t F, \bar{S}^\Gamma_t F \rangle_{L^2(\bar{\pi}_{\Gamma})},
	\end{equation}
	where the last inequality uses the fact that $\bar{S}^\Gamma_t f \in \tilde{L}_0^2(\bar{\pi}_\Gamma)$ according to Lemma \ref{lem:flow invariance}. Exponential decay as in \eqref{eq:exp decay} follows by Gronwall's Lemma. The converse direction follows by performing a Taylor expansion of the decay estimate \eqref{eq:exp decay} around $t=0$.
\end{proof}
To explain the significance of Theorem \ref{thm:Poincare}, let us start by writing \eqref{eq:coupling_Poincare inequality} in the form
\begin{equation}
\label{eq:Poincare plus coupling}
\langle f, f \rangle_{L^2(\pi)} + \frac{1}{n} \sum_{\substack{i,j = 1 \\ i \neq j}}^{n} \langle f_i, f_j \rangle_{L^2(\bar{\pi}_\Gamma)} \le \frac{1}{\Lambda} \left( \langle f, (-\mathcal{L}) f \rangle_{L^2(\pi)}  + \frac{1}{n} \sum_{\substack{i,j = 1 \\ i \neq j}}^{n} \langle f_i, (-\mathcal{L}_j) f_j \rangle_{L^2(\bar{\pi}_\Gamma)} \right),
\end{equation}
using the marginal property of $\bar{\pi}_\Gamma$. Clearly, \eqref{eq:Poincare plus coupling} deviates from the usual one-particle Poincar\'{e} inequality by the additional terms involving summation over pairs of particles. To make this more precise and analyse the impact of these terms, let us define the following bilinear form on $L_0^2(\pi)$:
\begin{equation}
\llangle f , g \rrangle := \langle f, g \rangle_{L^2(\pi)} + \frac{1}{n} \sum_{\substack{i,j = 1 \\ i \neq j}}^{n} \langle f_i, g_j \rangle_{L^2(\bar{\pi}_\Gamma)}, \quad f,g \in L_0^2(\pi). 
\end{equation}
For $F = \sum_{i=1}^n f_i$ and $G = \sum_{i=1}^n g_i$ we have that $\langle F, G \rangle_{L^2(\bar{\pi}_\Gamma)} = n\llangle f, g \rrangle$. Hence, $\llangle \cdot , \cdot \rrangle$ is both symmetric and nonnegative definite, but  $\llangle f, f \rrangle = 0$ is possible for $f \neq 0$. It is therefore natural to define the equivalence relation
\begin{equation}
f \sim g :\iff \llangle f - g, f - g \rrangle = 0, \quad f,g \in L_0^2(\pi),
\end{equation}
and the corresponding Hilbert space 
\begin{equation}
\mathcal{H} := L_0^2(\pi) / \sim.
\end{equation}
Using again the correspondence $F = \sum_{i=1}^n f_i$, we see that $\llangle f, f \rrangle = 0$ if and only if $F=0$ $\bar{\pi}_\Gamma$-almost surely. By ergodicity, this is also equivalent to $\bar{\mathcal{L}}_\Gamma F = 0$, $\bar{\pi}_\Gamma$-almost surely. We hence see that $\mathcal{L}$ respects $\sim$-equivalence classes, i.e. $f \sim g$ if and only if $\mathcal{L}f \sim \mathcal{L}g$. Denoting the induced operator on $\mathcal{H}$ by $\mathcal{L}_{\mathcal{H}}$, it is then immediate each of \eqref{eq:coupling_Poincare inequality} and \eqref{eq:Poincare plus coupling} is equivalent to
\begin{equation}
\label{eq:coupled Poincare}
\llangle f, f \rrangle \le \frac{1}{\Lambda} \llangle f, (-\mathcal{L}_{\mathcal{H}}) f \rrangle, \quad f \in \mathcal{H} \cap \mathcal{D}(\mathcal{L}).  
\end{equation}
By its similarity to the one-particle Poincar{\'e} inequality, the formulation \eqref{eq:coupled Poincare} is convenient for the comparison between the spectral gaps of the underlying and the coupled dynamics. 

Let us assume from now on that $\mathcal{L}$ is self-adjoint in $L_0^2(\pi)$ with discrete spectrum, with
\begin{equation}
\label{eq:coupling_spectrum}
-\mathcal{L} e_i = \mu_i e_i,\quad 0 <\mu_1 \le \mu_2 \le ,\ldots
\end{equation} 
where the eigenvectors $(e_i)_{i\in \mathbb{N}}$ form an orthonormal basis in $L_0^2(\pi)$. The optimal constant in the one-particle Poincar{\'e} inequality is then clearly given by $\lambda = \mu_1$. In the study of the coupled Poincar{\'e} inequality \eqref{eq:coupled Poincare}, two interesting effects might occur. Firstly, the spectrum of $\mathcal{L}_{\mathcal{H}}$ might be different from the spectrum of $\mathcal{L}$. Secondly, $\mathcal{L}_{\mathcal{H}}$ might not be symmetric with respect to $\llangle \cdot, \cdot \rrangle$. Let us start with the first point. Clearly, $\sigma(\mathcal{L}_\mathcal{H}) \subseteq  \sigma(\mathcal{L})$, more precisely
\begin{equation}
\label{eq:spectrum L_H}
\sigma(\mathcal{L}_\mathcal{H}) = \left\{ \lambda_i : \; \llangle e_i,e_i \rrangle \neq 0\right\}.
\end{equation}

\begin{example}
	\label{ex:faster}
	Consider the dynamics
	\begin{subequations}
		\label{eq:spectral gap coupled system}
		\begin{align}
		\mathrm{d}X_t & = - \nabla V(X_t) \,\mathrm{d}t + \sqrt{2}\,\mathrm{d}B_t, \\
		\mathrm{d}Y_t & = - \nabla V(Y_t) \,\mathrm{d}t - \sqrt{2}\,\mathrm{d}B_t,
		\end{align}
	\end{subequations}
	with a standard $\mathbb{R}^d$-valued Brownian motion $(B_t)_{t \ge 0}$. Let us assume that the potential $V$ grows sufficiently fast at infinity such that the one-particle generator $\mathcal{L}$ has compact resolvent and hence discrete spectrum in $L_0^2(\pi)$. Furthermore, suppose that the eigenvalues and eigenfunctions are labelled and ordered as in \eqref{eq:coupling_spectrum}. Let us now assume that $V$ is even, i.e $V(x)= V(-x)$, and that the process is ergodic. The invariant measure is then given by 
	\begin{equation}
	\bar{\pi}_\Gamma (\mathrm{d}x \mathrm{d}y) = \frac{1}{Z}e^{-V(x)} \delta_{x+y} \, \mathrm{d}x \mathrm{d}y,
	\end{equation}
	and the corresponding new (degenerate) scalar product in $L_0^2(\pi)$ turns out to be
	\begin{subequations}
		\begin{align}
		\llangle f, g \rrangle &  = \frac{1}{Z}\int_{\mathbb{R}^d} f(x)g(x) e^{-V(x)} \, \mathrm{d}x \\
		& + \frac{1}{2Z} \left( \int_{\mathbb{R}^d} f(x)g(-x) e^{-V(x)} \mathrm{d}x + \int_{\mathbb{R}^d} f(-x)g(x) e^{-V(x)} \mathrm{d}x\right).
		\end{align}
	\end{subequations}
	Notice that by the symmetry of $V$, all the eigenfunctions of $\mathcal{L}$ are either even or odd. Moreover, a short calculation shows that $\llangle f, f \rrangle = 0$ if and only if $f$ is odd (meaning that $-f(x)= f(-x)$). Using \eqref{eq:spectrum L_H}, we see that 
	\begin{equation}
	\sigma(\mathcal{L}_\mathcal{H}) = \left\{ \lambda_i : \; e_i \; \text{is odd} \right\}.
	\end{equation}
	Another short calculation shows that $\mathcal{L}_{\mathcal{H}}$ is symmetric with respect to $\llangle \cdot, \cdot \rrangle$, i.e.
	\begin{equation}
	\llangle f, \mathcal{L}_{\mathcal{H}} g\rrangle = \llangle \mathcal{L}_{\mathcal{H}}f,  g\rrangle, \quad f \in \mathcal{H} \cap \mathcal{D}(\mathcal{L}).
	\end{equation}
	
	If the first eigenfunction $e_1$ is odd\footnote{In one dimension, it can be proved that the first eigenfunction is always odd by appealing to the node theorem for Schr{\"o}dinger operators in Sturm-Liouville theory \cite[Chapter 9]{Teschl2014}. We conjecture that this fact might also be true in higher dimensions, but are not able to give a proof or a reference. Our special thanks go to Sabine B{\"o}gli and Ari Laptev for discussing this question with us.}, it therefore follows that the coupled Poincar{\'e} inequality \eqref{eq:Poincare plus coupling} holds with the constant $\Lambda = \mu_2$, showing an improved rate of convergence for the coupled dynamics. 
\end{example}
\begin{example}
	\label{ex:slower}
	Let us examine the second point, i.e. the possibility of $\mathcal{L}_{\mathcal{H}}$ not being symmetric with respect to $\llangle, \cdot , \cdot \rrangle$. For simplicity, assume that $\sigma(\mathcal{L}_\mathcal{H}) = \sigma(\mathcal{L})$, i.e. $\llangle e_i,e_i \rrangle \neq 0$ for all $i \in \mathbb{N}$. Consider the case when the measures $\bar{\pi}_0$ and $\bar{\pi}_\Gamma$ have densities with respect to a common dominating measure $m$ (for convenience denoted by the same symbols), and suppose there exist constants $c_1, c_2 > 0$ such that
	\begin{equation}
	\label{eq:upper lower density}
	c_1 \bar{\pi}_0(x) \le \bar{\pi}_\Gamma(x) \le c_2 \bar{\pi}_0(x), \quad x \in E.
	\end{equation}
	This is the case precisely when the norms in $L^2(\bar{\pi}_0)$ and $L^2(\bar{\pi}_\Gamma)$ are equivalent. For $F = \sum_{i=1}^n f_i$, we have that $n \llangle f, f \rrangle = \langle F, F \rangle_{L^2(\bar{\pi}_\Gamma)}$ as well as  $n \langle f, f \rangle_{L^2(\pi)} = \langle F, F \rangle_{L^2(\bar{\pi}_0)}$. Using \eqref{eq:upper lower density}, we hence conclude that
	\begin{equation}
	\label{eq:coupling_equivalence}
	c_1 \langle f, f \rangle_{L^2(\pi)} \le \llangle f, f \rrangle \le c_2 \langle f, f \rangle_{L^2(\pi)}, \quad f \in L^2(\pi).
	\end{equation}
	By assumption, the marginal process satisfies a Poincar{\'e} inequality as well as the equivalent decay estimate
	\begin{equation}
	\label{eq:one p decay}
	\Vert S_t f \Vert_{L^2(\pi)}^2 \le e^{-2 \lambda t} \Vert f \Vert^2_{L^2(\pi)}, \quad f \in L_0^2(\pi),
	\end{equation}
	with $\lambda = \mu_1$. By the equivalence \eqref{eq:coupling_equivalence}, we conclude that 
	\begin{equation}
	\label{eq:exp C decay}
	\Vert \bar{S}^{\Gamma}_t F \Vert^2_{L^2(\bar{\pi}_{\Gamma})} \le C e^{-2\lambda t} \Vert F \Vert^2_{L^2(\bar{\pi}_{\Gamma})}, \quad t\ge 0, \quad F \in \tilde{L}_0^2(\bar{\pi}_\Gamma),
	\end{equation}
	with $C = \frac{c_2}{c_1}$. Comparing \eqref{eq:one p decay} and \eqref{eq:exp C decay}, we see that the coupled process achieves the same exponential rate of convergence as the one-particles processes, but possibly with a worse constant $C$ in front of the exponential. The latter can be characterised in terms of the equivalence estimate \eqref{eq:upper lower density}. 
\end{example}

Conclusively, the speed of convergence to equilibrium can be both faster (as in Example \ref{ex:faster}) and slower (as in Example \ref{ex:slower}) for coupled processes, in comparison with the underlying one-particle processes. We leave a more thorough investigation of the Poincar{\'e} inequality \eqref{eq:Poincare plus coupling} for future work.

\section{Outlook and future work}

In this paper we have introduced a general framework for the construction and analysis of coupled MCMC samplers. Formulating the results in an abstract setting has allowed us to address both (possibly degenerate) diffusion processes as well as piecewise deterministic Markov processes, emphasising common structural properties. The analysis of appropriate central limit theorems has exposed notable connections to the theory of optimal transportation. We showed that the ensuing optimisation problem has singularity properties akin to those appearing in the usual Kantorovich formulation. We then studied a surrogate problem, leading to novel coupling strategies that seem promising for applications. Finally, we derived a functional inequality of Poincar{\'e} type suitable for the study of the exponential convergence to equilibrium for coupled processes. 

Our work can be extended in several directions. On the theoretical side, proving or disproving the Conjectures \ref{conj:G convex}, \ref{conj:G0 convex} and \ref{conj:general construction} would further illuminate the structural properties of the developed theory. Moreover, establishing a more rigorous connection between the optimal transport problems \ref{prob:OT C0} and \ref{prob:linearised OT} with the usual Kantorovich formulation might lead to further developments bridging the theories of Markov processes and optimal transportation. 

In terms of applications in sampling, a more detailed study of the couplings between many particles is needed, a starting point being the results in Section \ref{sec:many particles}. Furthermore, it would be desirable to relax our assumption that the laws of the marginal processes remain unchanged, as this would allow for more pronounced interactions between the particles. In this regard, the inclusion of the methodology put forward in \cite{LMW2018} in our framework would be of particular interest for practitioners.

In the broader context of statistical computation, it seems that coupling approaches along the lines developed here could be fruitfully applied in the context of the calculation of transport coefficients and sensitivities \cite{assaraf2015computation,hall2016uncertainty}. More speculatively, it would be interesting to investigate the use of our ideas in the context of multilevel Monte Carlo \cite{Giles2015} or computational optimal transport \cite{peyre2017computational}. We leave these directions for future investigations.

\subsection*{Acknowledgements}
NN is supported by the EPSRC through a Roth Departmental Scholarship. GP is supported by the EPSRC under grants No. EP/P031587/1, EP/L024926/1 and EP/L020564/1. The authors would like to thank Pedro Aceves Sanchez, Andreas Eberle, Julien Roussel, Gabriel Stoltz and Urbain Vaes for stimulating discussions.

\appendix
\section{Random orthogonal transformations of Brownian motions}
\label{sec:orth transform}
The following lemma has been extracted from \cite[page 56]{Eberle2015}, see also \cite[Theorem 8.4.2]{Oks2003}. This result states that the set of Brownian motions is preserved under possibly time-dependent linear transformations possessing certain orthogonality properties.  Importantly, no regularity constraints with regard to the time-dependence are required beyond measurability. This fact is crucial in the proofs of Lemmas \ref{lem:1d BMs} and \ref{lem:overdamped BM}. 
\begin{lemma}[Random orthogonal transformations]
	\label{lem:random orthogonal transformations}
	Suppose that the $\mathbb{R}^N$-valued stochastic process $(X_t)_{t \ge 0}$ is a solution to the SDE
	\begin{equation}
	\mathrm{d} X_t = O_t \mathrm{d} W_t, \quad X_0 = x_0,
	\end{equation}
	where $(W_t)_{t \ge 0}$ is an $M$-dimensional standard Brownian motion generating the filtration $(\mathcal{F}_t)_{t \ge 0}$, and $(O_t)_{t \ge 0}$ is a product-measurable $(\mathcal{F}_t)_{t \ge 0}$-adapted process taking values in $\mathbb{R}^{N \times M}$. Assume furthermore that
	\begin{equation}
	\label{eq:orthogonality condition}
	O_t O_t^T = I_{N \times N}
	\end{equation}
	for all $t \ge 0$, almost surely. Then $(X_t)_{t \ge 0}$ is an $N$-dimensional standard Brownian motion. 
\end{lemma}

\section{The derivative formula for invariant measures}
\label{sec:proofs_OT}
Here, we provide the proof of the derivative formula \eqref{eq:derivative} that allows us to compute the change of the invariant measure under an infinitesimal change of the coupling.
\begin{proof}[Proof of Proposition \ref{prop:derivative}]
	The idea of the proof stems from \cite{LMS2016} in the context of invariant measures for discretised SDEs and was also advertised in \cite[Remark 5.5]{LS2016}. 
	
	For convenience, let us first introduce the notation $g=-(\bar{\mathcal{L}}_{\Gamma}^*)^{-1} \mathrm{d}\Gamma^* \mathbf{1}$. Furthermore, we will make use of the projection operators 
	\begin{equation}
	\label{eq:projections}
	\Pi \phi = \int_{\bar{E}} \phi \, \mathrm{d}\bar{\pi}_{\Gamma}^0, \quad \Pi^{\perp} \phi = \phi - \Pi \phi,
	\end{equation}
	acting on $L^2(\bar{\pi}_\Gamma)$. 	Using $c = (\Pi + \Pi^{\perp})c$ in \eqref{eq:derivative}, we see that \eqref{eq:derivative} is equivalent to 
	\begin{equation}
	\frac{\mathrm{d}}{\mathrm{d}\varepsilon} \bigg\rvert_{\varepsilon = 0} \int_{\bar{E}} \Pi^{\perp}c \, \left(\mathrm{d}\bar{\pi}^{\varepsilon}_{\Gamma} \right) = - \int_{\bar{E}} \left(\Pi^{\perp}c\right) \left[ \bar{\mathcal{L}}_{\Gamma} ^*\right]^{-1} (\mathrm{d}\Gamma^* 1) \,\mathrm{d}\bar{\pi}_{\Gamma}^{0}.
	\end{equation}
	We may thus without loss of generality assume that $\Pi^{\perp}c = c$ (i.e. $\Pi c =0$), and will do so in the following. Furthermore, let us also assume that $c \in C^{\infty}(\bar{E})$ such that the calculations in the sequel are justified. The general case then follows by a standard approximation argument.
	A short calculation (using the fact that $\bar{\pi}_{\Gamma}^0(\bar{\mathcal{L}}_{\Gamma}c) = 0$)  shows that
	\begin{equation}
	\int_{\bar{E}}\bar{\mathcal{L}}_{\Gamma}^{\varepsilon} c \cdot (1+\varepsilon g) \,\mathrm{d}\bar{\pi}_{\Gamma}^0 = \varepsilon^2 \int_{\bar{E}} (\mathrm{d}\Gamma c) g\, \mathrm{d}\bar{\pi}_{\Gamma}^0 = \mathcal{O}(\varepsilon^2).
	\end{equation}
	Inserting $\Pi + \Pi^{\perp} = I$, we see that the above is equivalent to
	\begin{equation}
	\label{eq:pi_0_pert}
	\int_{\bar{E}} \Pi^{\perp} \bar{\mathcal{L}}_{\Gamma}^{\varepsilon} \Pi^{\perp} c \cdot (1 + \varepsilon g) \, \mathrm{d}\bar{\pi}_{\Gamma}^0 + \varepsilon \int_{\bar{E}} \Pi \mathrm{d} \Gamma c \cdot (1 + \varepsilon g) \, \mathrm{d}\bar{\pi}_0 = \mathcal{O}(\varepsilon^2).
	\end{equation} 
	At the same time, we have that
	\begin{equation}
	\label{eq:pi_eps_null}
	\int_{\bar{E}} \bar{\mathcal{L}}_{\Gamma}^{\varepsilon} c \,\mathrm{d}\bar{\pi}_{\Gamma}^{\varepsilon} = 0.
	\end{equation}
	Using again $\Pi + \Pi^{\perp} = I$ and $\Pi \bar{\mathcal{L}}_{\Gamma}=0$, \eqref{eq:pi_eps_null} can be expressed as
	\begin{equation}
	\label{eq:pi_eps_pert}
	\int_{\bar{E}} \Pi^{\perp} \bar{\mathcal{L}}_{\Gamma}^{\varepsilon} \Pi^{\perp} c \, \mathrm{d}\bar{\pi}_{\Gamma}^{\varepsilon} = -\varepsilon \int_{\bar{E}} \Pi \mathrm{d}\Gamma c \, \mathrm{d}\bar{\pi}_{\Gamma}^{\varepsilon}.	\end{equation}
	We can now combine \eqref{eq:pi_0_pert} and \eqref{eq:pi_eps_pert} to arrive at
	\begin{equation}
	\label{eq:pi_0_pi_eps}
	\int_{\bar{E}} \Pi^{\perp}\bar{\mathcal{L}}_{\Gamma}^\varepsilon \Pi^ \perp c  \cdot (1 + \varepsilon g) \, \mathrm{d}\bar{\pi}_{\Gamma}^0 - \int_{\bar{E}} \Pi^\perp \bar{\mathcal{L}}_\Gamma^\varepsilon \Pi^\perp c \, \mathrm{d}\bar{\pi}_\Gamma^\varepsilon = \mathcal{O}(\varepsilon^2).
	\end{equation}
	Let us introduce the `pseudo-inverse' 
	\begin{equation}
	Q_{\varepsilon} = \Pi^\perp \bar{\mathcal{L}}_{\Gamma}^{-1} \Pi^\perp - \varepsilon \Pi^\perp \bar{\mathcal{L}}_\Gamma^{-1}\Pi^\perp \mathrm{d} \Gamma \Pi^\perp \bar{\mathcal{L}}_{\Gamma}^{-1}\Pi^\perp,
	\end{equation}
	acting on $L^2_0(\bar{\pi}_\Gamma^0)$. We have that
	\begin{equation}
	\Pi^\perp \bar{\mathcal{L}}_\Gamma^\varepsilon \Pi^\perp Q_\varepsilon = \Pi^\perp - \varepsilon^2 \Pi^\perp \mathrm{d} \Gamma \Pi^\perp \bar{\mathcal{L}}_\Gamma^{-1} \Pi^\perp \mathrm{d}\Gamma \Pi^\perp \bar{\mathcal{L}}_\Gamma^{-1}\Pi^\perp,  
	\end{equation}
	i.e. in $L^2_0(\bar{\pi}_\Gamma^0)$, $Q_\varepsilon$ inverts $\bar{\mathcal{L}}_\Gamma^\varepsilon$ up to an error of order $\varepsilon^2$. Upon replacing $c$ by $Q_{\varepsilon} c$ in \eqref{eq:pi_0_pi_eps}, it follows that
	\begin{equation}
	\label{eq:difference quotient}
	\int_{\bar{E}} c (1 + \varepsilon g) \, \mathrm{d}\bar{\pi}_{\Gamma}^0 - \int_{\bar{E}} c \, \mathrm{d}\bar{\pi}_\Gamma^\varepsilon = \mathcal{O}(\varepsilon^2),
	\end{equation}
	recalling that $\Pi^\perp c = c$ by assumption. In the last step, we have used the fact that there exists a constant $C>0$ such that
	\begin{equation}
	\label{eq:bounded integrand}
	\left\vert\int_{\bar{E}}\Pi^\perp \mathrm{d} \Gamma \Pi^\perp \bar{\mathcal{L}}_\Gamma^{-1} \Pi^\perp \mathrm{d}\Gamma \Pi^\perp \bar{\mathcal{L}}_\Gamma^{-1}\Pi^\perp c \, \mathrm{d}\bar{\pi}_\Gamma^\varepsilon\right\vert \le C,
	\end{equation}
	uniformly in $\varepsilon$. Indeed, the integrand is bounded by the third condition of Definition \ref{def:interior point} and the fact that the coefficients of $\mathrm{d}\Gamma$ have compact support. The bound \eqref{eq:bounded integrand} is required to ensure that the corresponding integral expression hidden on the right-hand side of  \eqref{eq:difference quotient} is indeed of order $\varepsilon^2$. Finally, deviding by $\varepsilon$ in \eqref{eq:difference quotient} and letting $\varepsilon \rightarrow 0$ yields the desired result. 
\end{proof}

\section{Properties of the solutions to one-dimensional Poisson equations}
\label{sec:proofs_pert}
The proofs in this section essentially leverage the fact that the Poisson equations under consideration can be solved up to quadratures in one dimension.
\begin{proof}[Proof of Lemma \ref{lem:1d Poisson}]
	Variation of constants shows that $\phi'$ is given by
	\begin{equation}
	\label{eq:phi'}
	\phi'(x) = \left(-\int_{-\infty}^x f(s)e^{-V(s)} \mathrm{d}s + C \right) e^{V(x)}, 
	\end{equation}
	for some constant $C \in \mathbb{R}$. The requirement that $\pi(\phi)=0$ necessitates $C=0$. Indeed, from $\pi(f) = 0$ it follows that the integral term in \eqref{eq:phi'} goes to zero as $x\rightarrow \pm \infty$, and therefore
	\begin{equation}
	\lim_{x \rightarrow \pm \infty} \frac{\phi'(x)}{e^{V(x)}} = C.
	\end{equation}
	By L'H{\^o}pital's rule, we have that
	\begin{equation}
	\lim_{x \rightarrow \pm \infty} \frac{\phi(x) e^{-V(x)}}{\int_0^x e^{V(s)}\,\mathrm{d}s \cdot e^{-V(x)}} = C.
	\end{equation}
	The requirement that $\phi$ is integrable with respect to $\pi(\mathrm{d}x) \propto e^{-V(x)}\,\mathrm{d}x$ implies that $$\lim_{x \rightarrow \pm \infty} \phi(x)e^{-V(x)} = 0.$$
	Furthermore (again by L'H{\^o}pital's rule), 
	\begin{equation}
	\lim_{x \rightarrow \pm \infty} \int_0^x e^{V(s)}\,\mathrm{d}s \cdot e^{-V(x)} = -\lim_{x \rightarrow \pm\infty} V'(x),
	\end{equation}
	which cannot be zero since $\int_{-\infty}^{\infty} e^{-V(x)}\,\mathrm{d}x < \infty$. Hence, $C=0$.
	
	To prove 1.), notice that from $\pi(f) = 0$ and monotonicity, it follows that there exists $x^* \in \mathbb{R}$ such that $f(x^*)=0$. Let us assume that $f$ is monotonically increasing (for monotonically decreasing $f$ the reasoning is analogous). We then have that $f \le 0$ on $(-\infty, x^*]$ and $f \ge 0$ on $[x^*, \infty)$. Consider now the function
	\begin{equation}
	\label{eq:def Phi}
	\Phi(x) = -\int_{-\infty}^x f(s)e^{-V(s)} \,\mathrm{d}s, \quad x \in \mathbb{R}.
	\end{equation}
	Clearly $\Phi$ is increasing on $(-\infty,x^*]$ and decreasing on $[x^*,\infty)$. From $\pi(f)=0$ it follows that $\lim_{x \rightarrow \pm \infty} \Phi(x) = 0$ and hence $\Phi(x) \ge 0$ for all $x \in \mathbb{R}$. This proves the claim since $\phi'(x) = \Phi(x)e^{V(x)}$.
	
	To prove 2.), first observe that $\pi(f)=0$ implies
	\begin{equation}
	\label{eq:integral sum}
	\int_{-\infty}^0 f(s) e^{-V(s)} \mathrm{d}s + \int_0^{\infty} f(s)e^{-V(s)}\mathrm{d}s =0. 
	\end{equation}
	Furthermore, the symmetry properties of $f$ and $V$ show that 
	\begin{equation}
	\label{eq:integral symmetry}
	\int_{-\infty}^0 f(s) e^{-V(s)} \mathrm{d}s - \int_0^{\infty} f(s)e^{-V(s)}\mathrm{d}s = 0,
	\end{equation}
	using the substitution $s \mapsto -s$.
	Equations \eqref{eq:integral sum} and \eqref{eq:integral symmetry} together imply that $\Phi(0) = 0$, for $\Phi$ as defined in \eqref{eq:def Phi}. The claim now follows using an analogous argument to the one used in the proof of 1.).
\end{proof}
\begin{proof}[Proof of Lemma \ref{lem:zigzag dsigma}]
	Recall from \eqref{eq:linearised objective} that
	\begin{equation}
	\delta \sigma_F^2(\Gamma_{\alpha}) = \int_{\bar{E}} \xi \,\mathrm{d}\bar{\pi}_0,
	\end{equation}
	where $\xi(x,y) = \phi(x)\phi(y)$ and $\phi:\mathbb{R} \times \{-1,+1\}\rightarrow \mathbb{R}$ is the solution to the (one-particle) Poisson equation
	\begin{equation}
	\label{eq:Poisson Zig-Zag}
	\theta \partial_x \phi(x,\theta) + \lambda(x,\theta) \left(\phi(x,-\theta) - \phi(x,\theta) \right) = f(x), \quad \pi(\phi) =0.
	\end{equation}
	Note that for convenience, we have assumed without loss of generality that $\tilde{\pi}(f) = 0$.
	Let us now calculate
	\begin{subequations}
		\label{eq:Zig zag calculation}
		\begin{align}
		\int_{\bar{E}} \Gamma \xi \,\mathrm{d}\pi_{0} & = \frac{1}{4} \sum_{\theta_x = \pm 1,\, \theta_y = \pm 1} \int_{\mathbb{R}^2} \Bigg(\alpha(x,y,\theta_x,\theta_y)
		\cdot \bigg[ \phi(x,\theta_x) \phi(y,\theta_y) - \phi(x,\theta_x) \phi(y, -\theta_y) -  \nonumber \\
		& - \phi(x,-\theta_x) \phi(y,\theta_y) + \phi(x,-\theta_x) \phi(y,-\theta_y) \bigg]\Bigg)e^{-\left(V(x) + V(y)\right)}\,\mathrm{d}x \mathrm{d}y \nonumber \\
		& = \frac{1}{4} \int_{{\mathbb{R}^2}} \Bigg( \bigg[ \alpha_{++}(x,y) + \alpha_{--}(x,y) - \alpha_{+-}(x,y) - \alpha_{-+}(x,y)\bigg] \cdot \nonumber \\
		& \cdot \bigg[\phi_{+}(x) \phi_{+}(y) + \phi_{-}(x)\phi_{-}(y) - \phi_{+}(x)\phi_{-}(y) - \phi_{-}(x) \phi_{+}(y) \bigg] \Bigg)e^{-\left(V(x) + V(y)\right)}\mathrm{d}x \mathrm{d}y \nonumber \\
		& = \frac{1}{4} \int_{{\mathbb{R}^2}} \Bigg( \bigg[ \alpha_{++}(x,y) + \alpha_{--}(x,y) - \alpha_{+-}(x,y) - \alpha_{-+}(x,y)\bigg] \cdot \nonumber \\
		& \cdot \bigg[\phi_{+}(x) - \phi_{-}(x) \bigg] \cdot \bigg[ \phi_{+}(y) - \phi_{-}(y) \bigg] \Bigg)e^{-\left(V(x) + V(y)\right)}\,\mathrm{d}x \mathrm{d}y,  \tag{\ref*{eq:Zig zag calculation}}
		\end{align}
	\end{subequations}
	where again we employed the notation $\phi_{\pm}(x) = \phi(x,\pm 1)$.
	Observe now that equation \eqref{eq:Poisson Zig-Zag} can be recast as
	\begin{subequations}
		\begin{align}
		\label{eq:Poisson Zig1}
		\partial_x \phi_{+}(x) + \lambda_+(x) \left(\phi_{-}(x) - \phi_{+}(x) \right) &  = f(x), \\
		\label{eq:Poisson Zig2}
		-\partial_x \phi_{-}(x) + \lambda_{-}(x) \left(\phi_{+}(x) - \phi_{-}(x) \right) &  = f(x), 
		\end{align}
	\end{subequations}
	where both $\phi_+$ and $\phi_{-}$ have to be integrable with respect to the measure $\frac{1}{Z}e^{-V(x)} \mathrm{d} x$ and satisfy
	\begin{equation}
	\int_{\mathbb{R}} \left(\phi_{+}(x) + \phi_{+}(x) \right) e^{-V(x)} \mathrm{d}x = 0.
	\end{equation}
	Adding \eqref{eq:Poisson Zig1} and \eqref{eq:Poisson Zig2} leads to 
	\begin{equation}
	\label{eq:Poisson phi+ phi-}
	\left(\partial_x - V'\right) (\phi_+ - \phi_-) = 2f,
	\end{equation}
	using $\lambda_+(x) - \lambda_{-}(x) = V'(x)$. Finally setting $\partial_x \tilde{\phi} = \frac{1}{2} (\phi_+ - \phi_-)$ and comparing with \eqref{eq:Zig zag calculation} leads to the desired result. Note that as in the proof of Lemma \ref{lem:1d Poisson}, \eqref{eq:Poisson phi+ phi-} determines $\phi_+ - \phi_-$ uniquely under the condition that $\phi_+ - \phi_-$ is integrable with respect to $e^{-V(x)}\mathrm{d}x$.   
\end{proof}

\bibliographystyle{abbrv}

\begin{thebibliography}{10}

\bibitem{AS1957}
I.~Amemiya and K.~Shiga.
\newblock On tensor products of {B}anach spaces.
\newblock {\em K\=odai Math. Sem. Rep.}, 9:161--178, 1957.

\bibitem{andrieu2010particle}
C.~Andrieu, A.~Doucet, and R.~Holenstein.
\newblock Particle {M}arkov chain {M}onte {C}arlo methods.
\newblock {\em Journal of the Royal Statistical Society: Series B (Statistical
  Methodology)}, 72(3):269--342, 2010.

\bibitem{LMS2010}
L.~Angiuli, G.~Metafune, and C.~Spina.
\newblock Feller semigroups and invariant measures.
\newblock {\em Riv. Math. Univ. Parma (N.S.)}, 1(2):347--406, 2010.

\bibitem{Oneparsempos}
W.~Arendt, A.~Grabosch, G.~Greiner, U.~Groh, H.~P. Lotz, U.~Moustakas,
  R.~Nagel, F.~Neubrander, and U.~Schlotterbeck.
\newblock {\em One-parameter semigroups of positive operators}, volume 1184 of
  {\em Lecture Notes in Mathematics}.
\newblock Springer-Verlag, Berlin, 1986.

\bibitem{assaraf2015computation}
R.~Assaraf, B.~Jourdain, T.~Leli{\`e}vre, and R.~Roux.
\newblock Computation of sensitivities for the invariant measure of a parameter
  dependent diffusion.
\newblock {\em Stochastics and Partial Differential Equations: Analysis and
  Computations}, pages 1--59, 2015.

\bibitem{bakry2013analysis}
D.~Bakry, I.~Gentil, and M.~Ledoux.
\newblock {\em Analysis and geometry of Markov diffusion operators}, volume
  348.
\newblock Springer Science \& Business Media, 2013.

\bibitem{Baxendale1991}
P.~H. Baxendale.
\newblock Statistical equilibrium and two-point motion for a stochastic flow of
  diffeomorphisms.
\newblock In {\em Spatial stochastic processes}, volume~19 of {\em Progr.
  Probab.}, pages 189--218. Birkh\"auser Boston, Boston, MA, 1991.

\bibitem{B1982}
R.~N. Bhattacharya.
\newblock On the functional central limit theorem and the law of the iterated
  logarithm for {M}arkov processes.
\newblock {\em Z. Wahrsch. Verw. Gebiete}, 60(2):185--201, 1982.

\bibitem{BD2017}
J.~Bierkens and A.~Duncan.
\newblock Limit theorems for the zig-zag process.
\newblock {\em Adv. in Appl. Probab.}, 49(3):791--825, 2017.

\bibitem{BFR2016}
J.~Bierkens, P.~Fearnhead, and G.~Roberts.
\newblock The zig-zag process and super-efficient sampling for {B}ayesian
  analysis of big data.
\newblock {\em arXiv:1607.03188}, 2016.

\bibitem{BRZ2018}
J.~Bierkens, G.~Roberts, and P.-A. Zitt.
\newblock Ergodicity of the zigzag process.
\newblock {\em arXiv:1712.09875}, 2018.

\bibitem{bonneel2011displacement}
N.~Bonneel, M.~Van De~Panne, S.~Paris, and W.~Heidrich.
\newblock Displacement interpolation using {L}agrangian mass transport.
\newblock In {\em ACM Transactions on Graphics (TOG)}, volume~30, page 158.
  ACM, 2011.

\bibitem{LevyMatters2013}
B.~B\"ottcher, R.~Schilling, and J.~Wang.
\newblock {\em L\'evy matters. {III}}, volume 2099 of {\em Lecture Notes in
  Mathematics}.
\newblock Springer, Cham, 2013.
\newblock L\'evy-type processes: construction, approximation and sample path
  properties, With a short biography of Paul L\'evy by Jean Jacod, L\'evy
  Matters.

\bibitem{BouRabee17}
N.~Bou-Rabee and J.~M. Sanz-Serna.
\newblock Randomized {H}amiltonian {M}onte {C}arlo.
\newblock {\em Ann. Appl. Probab.}, 27(4):2159--2194, 2017.

\bibitem{BVD2017}
A.~Bouchard-C{\^o}t{\'e}, S.~J. Vollmer, and A.~Doucet.
\newblock The bouncy particle sampler: A non-reversible rejection-free {M}arkov
  chain {M}onte {C}arlo method.
\newblock {\em Journal of the American Statistical Association}, 0(ja):0--0,
  2017.

\bibitem{BV2004}
S.~Boyd and L.~Vandenberghe.
\newblock {\em Convex optimization}.
\newblock Cambridge University Press, Cambridge, 2004.

\bibitem{CM2005}
M.-F. Chen.
\newblock {\em Eigenvalues, inequalities, and ergodic theory}.
\newblock Probability and its Applications (New York). Springer-Verlag London,
  Ltd., London, 2005.

\bibitem{CFK2013}
C.~Cotar, G.~Friesecke, and C.~Kl\"uppelberg.
\newblock Density functional theory and optimal transportation with {C}oulomb
  cost.
\newblock {\em Comm. Pure Appl. Math.}, 66(4):548--599, 2013.

\bibitem{Cour65}
P.~Courr{\`e}ge.
\newblock Sur la forme int{\'e}gro-diff{\'e}rentielle des op{\'e}rateurs de
  ${C}^\infty _k$ dans {$C$} satisfaisant au principe du maximum.
\newblock {\em S{\'e}minaire Brelot-Choquet-Deny. Th{\'e}orie du potentiel},
  10(1):1--38, 1965-1966.

\bibitem{craiu2007acceleration}
R.~V. Craiu and C.~Lemieux.
\newblock Acceleration of the multiple-try {M}etropolis algorithm using
  antithetic and stratified sampling.
\newblock {\em Statistics and computing}, 17(2):109, 2007.

\bibitem{craiu2005}
R.~V. Craiu and X.-L. Meng.
\newblock Multiprocess parallel antithetic coupling for backward and forward
  {M}arkov chain {M}onte {C}arlo.
\newblock {\em Ann. Statist.}, 33(2):661--697, 04 2005.

\bibitem{DaZa1996}
G.~Da~Prato and J.~Zabczyk.
\newblock {\em Ergodicity for infinite-dimensional systems}, volume 229 of {\em
  London Mathematical Society Lecture Note Series}.
\newblock Cambridge University Press, Cambridge, 1996.

\bibitem{Davis1984}
M.~H.~A. Davis.
\newblock Piecewise-deterministic {M}arkov processes: a general class of
  nondiffusion stochastic models.
\newblock {\em J. Roy. Statist. Soc. Ser. B}, 46(3):353--388, 1984.
\newblock With discussion.

\bibitem{Davis1993}
M.~H.~A. Davis.
\newblock {\em Markov models and optimization}, volume~49 of {\em Monographs on
  Statistics and Applied Probability}.
\newblock Chapman \& Hall, London, 1993.

\bibitem{dM2013}
P.~Del~Moral.
\newblock {\em Mean field simulation for {M}onte {C}arlo integration}, volume
  126 of {\em Monographs on Statistics and Applied Probability}.
\newblock CRC Press, Boca Raton, FL, 2013.

\bibitem{dellaportas2012control}
P.~Dellaportas and I.~Kontoyiannis.
\newblock Control variates for estimation based on reversible {M}arkov chain
  {M}onte {C}arlo samplers.
\newblock {\em Journal of the Royal Statistical Society: Series B (Statistical
  Methodology)}, 74(1):133--161, 2012.

\bibitem{DLP2016}
A.~B. Duncan, T.~Leli\`evre, and G.~A. Pavliotis.
\newblock Variance reduction using nonreversible {L}angevin samplers.
\newblock {\em J. Stat. Phys.}, 163(3):457--491, 2016.

\bibitem{DNP2017}
A.~B. Duncan, N.~N\"usken, and G.~A. Pavliotis.
\newblock Using perturbed underdamped {L}angevin dynamics to efficiently sample
  from probability distributions.
\newblock {\em J. Stat. Phys.}, 169(6):1098--1131, 2017.

\bibitem{Eberle2015}
A.~Eberle.
\newblock Stochastic analysis.
\newblock
  \url{https://wt.iam.uni-bonn.de/fileadmin/WT/Inhalt/people/Andreas_Eberle/StoAn15/StochasticAnalysis2015.pdf},
  2015.
\newblock Lecture Notes, accessed 12/03/2018.

\bibitem{eberle2016reflection}
A.~Eberle.
\newblock Reflection couplings and contraction rates for diffusions.
\newblock {\em Probability theory and related fields}, 166(3-4):851--886, 2016.

\bibitem{EN2000}
K.-J. Engel and R.~Nagel.
\newblock {\em One-parameter semigroups for linear evolution equations}, volume
  194 of {\em Graduate Texts in Mathematics}.
\newblock Springer-Verlag, New York, 2000.
\newblock With contributions by S. Brendle, M. Campiti, T. Hahn, G. Metafune,
  G. Nickel, D. Pallara, C. Perazzoli, A. Rhandi, S. Romanelli and R.
  Schnaubelt.

\bibitem{EthierKu86}
S.~N. Ethier and T.~G. Kurtz.
\newblock {\em Markov processes}.
\newblock Wiley Series in Probability and Mathematical Statistics: Probability
  and Mathematical Statistics. John Wiley \& Sons, Inc., New York, 1986.
\newblock Characterization and convergence.

\bibitem{fearnhead2016piecewise}
P.~Fearnhead, J.~Bierkens, M.~Pollock, and G.~O. Roberts.
\newblock Piecewise deterministic {M}arkov processes for continuous-time
  {M}onte {C}arlo.
\newblock {\em arXiv preprint arXiv:1611.07873}, 2016.

\bibitem{frigessi2000antithetic}
A.~Frigessi, J.~Gasemyr, and H.~Rue.
\newblock Antithetic coupling of two {G}ibbs sampler chains.
\newblock {\em Annals of Statistics}, pages 1128--1149, 2000.

\bibitem{ghanem2017handbook}
R.~Ghanem, D.~Higdon, and H.~Owhadi.
\newblock {\em Handbook of uncertainty quantification}.
\newblock Springer, 2017.

\bibitem{Giles2015}
M.~B. Giles.
\newblock Multilevel {M}onte {C}arlo methods.
\newblock {\em Acta Numer.}, 24:259--328, 2015.

\bibitem{GlynnMeyn1996}
P.~W. Glynn and S.~P. Meyn.
\newblock A {L}iapounov bound for solutions of the {P}oisson equation.
\newblock 24(2), 1996.

\bibitem{HairerMattingly2011}
M.~Hairer and J.~C. Mattingly.
\newblock Yet another look at {H}arris' ergodic theorem for {M}arkov chains.
\newblock In {\em Seminar on {S}tochastic {A}nalysis, {R}andom {F}ields and
  {A}pplications {VI}}, volume~63 of {\em Progr. Probab.}, pages 109--117.
  Birkh\"auser/Springer Basel AG, Basel, 2011.

\bibitem{hall2016uncertainty}
E.~J. Hall, M.~A. Katsoulakis, and L.~Rey-Bellet.
\newblock Uncertainty quantification for generalized {L}angevin dynamics.
\newblock {\em The Journal of chemical physics}, 145(22):224108, 2016.

\bibitem{heng2015gibbs}
J.~Heng, A.~Doucet, and Y.~Pokern.
\newblock {G}ibbs flow for approximate transport with applications to
  {B}ayesian computation.
\newblock {\em arXiv preprint arXiv:1509.08787}, 2015.

\bibitem{holmes2009antithetic}
C.~Holmes and A.~Jasra.
\newblock Antithetic methods for {G}ibbs samplers.
\newblock {\em Journal of Computational and Graphical Statistics},
  18(2):401--414, 2009.

\bibitem{Hwang2005}
C.-R. Hwang, S.-Y. Hwang-Ma, and S.-J. Sheu.
\newblock Accelerating diffusions.
\newblock {\em Ann. Appl. Probab.}, 15(2):1433--1444, 2005.

\bibitem{Jac2001vol1}
N.~Jacob.
\newblock {\em Pseudo differential operators and {M}arkov processes. {V}ol.
  {I}}.
\newblock Imperial College Press, London, 2001.
\newblock Fourier analysis and semigroups.

\bibitem{K2002}
O.~Kallenberg.
\newblock {\em Foundations of modern probability}.
\newblock Probability and its Applications (New York). Springer-Verlag, New
  York, second edition, 2002.

\bibitem{kliemann1987recurrence}
W.~Kliemann.
\newblock Recurrence and invariant measures for degenerate diffusions.
\newblock {\em The {A}nnals of {P}robability}, pages 690--707, 1987.

\bibitem{KomorowskiLandimOlla2012}
T.~Komorowski, C.~Landim, and S.~Olla.
\newblock {\em Fluctuations in {M}arkov processes}, volume 345 of {\em
  Grundlehren der Mathematischen Wissenschaften [Fundamental Principles of
  Mathematical Sciences]}.
\newblock Springer, Heidelberg, 2012.
\newblock Time symmetry and martingale approximation.

\bibitem{kroese2013handbook}
D.~P. Kroese, T.~Taimre, and Z.~I. Botev.
\newblock {\em Handbook of {M}onte {C}arlo methods}, volume 706.
\newblock John Wiley \& Sons, 2013.

\bibitem{Kuehn2017}
F.~K\"uhn.
\newblock {\em L\'evy matters. {VI}}, volume 2187 of {\em Lecture Notes in
  Mathematics}.
\newblock Springer, Cham, 2017.
\newblock L\'evy-type processes: moments, construction and heat kernel
  estimates, With a short biography of Paul L\'evy by Jean Jacod, L\'evy
  Matters.

\bibitem{Kuehn2018}
F.~K{\"u}hn.
\newblock Existence of ({M}arkovian) solutions to martingale problems
  associated with {L}{\'e}vy-type operators.
\newblock {\em arXiv:1803.05646}, 2018.

\bibitem{kwak2016antithetic}
J.~Kwak et~al.
\newblock An antithetic coupling approach to multi-chain based csma scheduling
  algorithms.
\newblock In {\em INFOCOM 2016-The 35th Annual IEEE International Conference on
  Computer Communications, IEEE}, pages 1--9. IEEE, 2016.

\bibitem{LMS2016}
B.~Leimkuhler, C.~Matthews, and G.~Stoltz.
\newblock The computation of averages from equilibrium and nonequilibrium
  {L}angevin molecular dynamics.
\newblock {\em IMA J. Numer. Anal.}, 36(1):13--79, 2016.

\bibitem{LMW2018}
B.~Leimkuhler, C.~Matthews, and J.~Weare.
\newblock Ensemble preconditioning for {M}arkov chain {M}onte {C}arlo
  simulation.
\newblock {\em Stat. Comput.}, 28(2):277--290, 2018.

\bibitem{Free_energy_computations}
T.~Leli{\`e}vre, M.~Rousset, and G.~Stoltz.
\newblock {\em Free energy computations}.
\newblock Imperial College Press, London, 2010.
\newblock A mathematical perspective.

\bibitem{LS2016}
T.~Leli\`evre and G.~Stoltz.
\newblock Partial differential equations and stochastic methods in molecular
  dynamics.
\newblock {\em Acta Numer.}, 25:681--880, 2016.

\bibitem{LPP2015}
V.~Lemaire, G.~Pag\`es, and F.~Panloup.
\newblock Invariant measure of duplicated diffusions and application to
  {R}ichardson-{R}omberg extrapolation.
\newblock {\em Ann. Inst. Henri Poincar\'e Probab. Stat.}, 51(4):1562--1596,
  2015.

\bibitem{lindvall2002}
T.~Lindvall.
\newblock {\em Lectures on the coupling method}.
\newblock Dover Publications, Inc., Mineola, NY, 2002.
\newblock Corrected reprint of the 1992 original.

\bibitem{liu2016stein}
Q.~Liu and D.~Wang.
\newblock Stein variational gradient descent: A general purpose {B}ayesian
  inference algorithm.
\newblock In {\em Advances In Neural Information Processing Systems}, pages
  2378--2386, 2016.

\bibitem{LM2007}
L.~Lorenzi and M.~Bertoldi.
\newblock {\em Analytical methods for {M}arkov semigroups}, volume 283 of {\em
  Pure and Applied Mathematics (Boca Raton)}.
\newblock Chapman \& Hall/CRC, Boca Raton, FL, 2007.

\bibitem{MPW2012}
R.~J. McCann, B.~Pass, and M.~Warren.
\newblock Rectifiability of optimal transportation plans.
\newblock {\em Canad. J. Math.}, 64(4):924--934, 2012.

\bibitem{MKK2014}
M.~Michel, S.~C. Kapfer, and W.~Krauth.
\newblock Generalized event-chain {M}onte {C}arlo: Constructing rejection-free
  global-balance algorithms from infinitesimal steps.
\newblock {\em Journal of Chemical Physics}, 140(5):054116, 2014.

\bibitem{Michel17}
M.~{Michel} and S.~{S{\'e}n{\'e}cal}.
\newblock {Forward Event-Chain {M}onte {C}arlo: a general rejection-free and
  irreversible {M}arkov chain simulation method}.
\newblock {\em M2AN}, 2017.

\bibitem{mijatovic2018poisson}
A.~Mijatovi{\'c} and J.~Vogrinc.
\newblock On the {P}oisson equation for {M}etropolis--{H}astings chains.
\newblock {\em Bernoulli}, 24(3):2401--2428, 2018.

\bibitem{neal1998suppressing}
R.~M. Neal.
\newblock Suppressing random walks in {M}arkov chain {M}onte {C}arlo using
  ordered overrelaxation.
\newblock In {\em Learning in graphical models}, pages 205--228. Springer,
  1998.

\bibitem{neal2017circularly}
R.~M. Neal.
\newblock Circularly-coupled {M}arkov chain sampling.
\newblock {\em arXiv preprint arXiv:1711.04399}, 2017.

\bibitem{neal2011mcmc}
R.~M. Neal et~al.
\newblock Mcmc using hamiltonian dynamics.
\newblock {\em Handbook of {M}arkov Chain {M}onte {C}arlo}, 2(11), 2011.

\bibitem{Oks2003}
B.~{\O}ksendal.
\newblock {\em Stochastic differential equations}.
\newblock Universitext. Springer-Verlag, Berlin, sixth edition, 2003.
\newblock An introduction with applications.

\bibitem{ottobre2016markov}
M.~Ottobre.
\newblock Markov chain {M}onte {C}arlo and irreversibility.
\newblock {\em Reports on Mathematical Physics}, 77(3):267--292, 2016.

\bibitem{pavliotis2014stochastic}
G.~A. Pavliotis.
\newblock {\em Stochastic processes and applications: diffusion processes, the
  {F}okker-{P}lanck and {L}angevin equations}, volume~60.
\newblock Springer, 2014.

\bibitem{peyre2017computational}
G.~Peyr{\'e}, M.~Cuturi, et~al.
\newblock Computational optimal transport.
\newblock Technical report, 2017.

\bibitem{pinnau2017consensus}
R.~Pinnau, C.~Totzeck, O.~Tse, and S.~Martin.
\newblock A consensus-based model for global optimization and its mean-field
  limit.
\newblock {\em Mathematical Models and Methods in Applied Sciences},
  27(01):183--204, 2017.

\bibitem{Propp1998coupling}
J.~Propp and D.~Wilson.
\newblock Coupling from the past: a user's guide.
\newblock In {\em Microsurveys in discrete probability ({P}rinceton, {NJ},
  1997)}, volume~41 of {\em DIMACS Ser. Discrete Math. Theoret. Comput. Sci.},
  pages 181--192. Amer. Math. Soc., Providence, RI, 1998.

\bibitem{RC2015}
S.~Reich and C.~Cotter.
\newblock {\em Probabilistic forecasting and {B}ayesian data assimilation}.
\newblock Cambridge University Press, New York, 2015.

\bibitem{revees2004genetic}
C.~R. Revees and J.~E. Rowe.
\newblock Genetic algorithms--principles and perspectives, 2004.

\bibitem{RBS2016}
L.~Rey-Bellet and K.~Spiliopoulos.
\newblock Improving the convergence of reversible samplers.
\newblock {\em J. Stat. Phys.}, 164(3):472--494, 2016.

\bibitem{rigat2012parallel}
F.~Rigat and A.~Mira.
\newblock Parallel hierarchical sampling: A general-purpose interacting markov
  chains monte carlo algorithm.
\newblock {\em Computational Statistics \& Data Analysis}, 56(6):1450--1467,
  2012.

\bibitem{RousselStoltz2017}
J.~Roussel and G.~Stoltz.
\newblock Spectral methods for {L}angevin dynamics and associated error
  estimates.
\newblock {\em ESAIM: Mathematical Modelling and Numerical Analysis}, 2017.

\bibitem{rousset:hal-00008276}
M.~Rousset and G.~Stoltz.
\newblock An interacting particle system approach for molecular dynamics.
\newblock preprint, Aug. 2005.

\bibitem{Sch98}
R.~Schilling.
\newblock Conservativeness and extensions of {F}eller semigroups.
\newblock {\em Positivity}, 2(3):239--256, 1998.

\bibitem{smith2013uncertainty}
R.~C. Smith.
\newblock {\em Uncertainty quantification: theory, implementation, and
  applications}, volume~12.
\newblock Siam, 2013.

\bibitem{sullivan2015introduction}
T.~J. Sullivan.
\newblock {\em Introduction to uncertainty quantification}, volume~63.
\newblock Springer, 2015.

\bibitem{Teschl2014}
G.~Teschl.
\newblock {\em Mathematical methods in quantum mechanics}, volume 157 of {\em
  Graduate Studies in Mathematics}.
\newblock American Mathematical Society, Providence, RI, second edition, 2014.
\newblock With applications to Schr\"odinger operators.

\bibitem{thorisson2000}
H.~Thorisson.
\newblock {\em Coupling, stationarity, and regeneration}.
\newblock Probability and its Applications (New York). Springer-Verlag, New
  York, 2000.

\bibitem{VBDD2017}
P.~Vanetti, A.~Bouchard-C{\^o}t{\'e}, G.~Deligiannidis, and A.~Doucet.
\newblock Piecewise deterministic {M}arkov chain {M}onte {C}arlo.
\newblock {\em arXiv:1707.05296}, 2017.

\bibitem{V2003}
C.~Villani.
\newblock {\em Topics in optimal transportation}, volume~58 of {\em Graduate
  Studies in Mathematics}.
\newblock American Mathematical Society, Providence, RI, 2003.

\bibitem{villani2009hypocoercivity}
C.~Villani.
\newblock {\em Hypocoercivity}.
\newblock Number 949-951. American Mathematical Soc., 2009.

\bibitem{V2009}
C.~Villani.
\newblock {\em Optimal transport}, volume 338 of {\em Grundlehren der
  Mathematischen Wissenschaften [Fundamental Principles of Mathematical
  Sciences]}.
\newblock Springer-Verlag, Berlin, 2009.
\newblock Old and new.

\end{thebibliography}

\end{document}